\documentclass[a4paper,11pt,reqno]{amsart}
\usepackage[margin=25mm]{geometry}
\usepackage{amsmath,amssymb,amsfonts,amsthm}
\usepackage{fixmath}
\usepackage{paralist}
\usepackage{multirow}
\usepackage{booktabs}
\usepackage{subfig}
\usepackage{tikz}
\usepackage{cite}
\usepackage[backref=page]{hyperref}

\AtBeginDocument{\setdefaultleftmargin{1.8em}{1em}{1em}{.5em}{.5em}{.5em}}

\newtheorem{theorem}{Theorem}[section]
\newtheorem{proposition}[theorem]{Proposition}
\newtheorem{corollary}[theorem]{Corollary}
\newtheorem{lemma}[theorem]{Lemma}
\theoremstyle{remark}
\newtheorem{remark}[theorem]{Remark}

\hypersetup{%
pdftitle={Uniform spanning trees on Sierpi\'nski graphs},
pdfauthor={Masato Shinoda and Elmar Teufl and Stephan Wagner},
pdfsubject={Mathematics},
pdfkeywords={uniform spanning trees, Sierpi\'nski graphs, loop erased random walk, limit behavior},
plainpages=false,pdfborder={0 0 0},colorlinks=true}

\def\N{\mathbb N}
\def\R{\mathbb R}
\def\C{\mathbb C}
\def\Words{\mathbb W}
\def\Prob{\qopname\relax o{\mathbb{P}}}
\def\Expect{\qopname\relax o{\mathbb{E}}}
\def\Unif{\qopname\relax o{\mathsf{Unif}}}
\def\LE{\qopname\relax o{\mathsf{LE}}}
\def\LI{\qopname\relax o{\mathsf{LI}}}
\def\Tr{\qopname\relax o{\mathsf{Tr}}}
\def\PGF{\qopname\relax o{\mathsf{PGF}}}
\def\hit{\mathsf h}
\def\time{\mathsf t}
\def\spend{\mathsf s}
\def\ltime{\mathsf T}
\def\lspend{\mathsf S}
\def\Metrics{\mathbb M}
\def\SigmaAlgebraMetrics{\mathcal M}
\def\abs#1{\lvert#1\rvert}
\def\card#1{\lvert#1\rvert}
\def\norm#1{\lVert#1\rVert}
\def\Re{\qopname\relax o{Re}}

\def\num{{\scriptscriptstyle\#}}
\def\type{\chi}
\let\phi=\varphi
\let\emptyset=\varnothing
\let\emptyword=\emptyset

\def\ndash{\nobreakdash-\hskip0pt}

\def\mytikzglyphsize{0.34} 
\def\mytikzglyphbase{.3pt} 
\tikzstyle{vertex}=[circle, fill=black, inner sep=0pt, minimum width=2pt]
\tikzstyle{dist}=[circle, draw, fill=black!40, inner sep=0pt, minimum width=4pt]
\tikzstyle{cond}=[midway,rectangle,inner sep=1pt,fill=white,midway,font=\tiny]
\def\mytikztrics[#1,#2]{\pgfpoint{((#1)+(#2)/2)*1cm}{(#2)*sqrt(3)/2*1cm}}
\tikzdeclarecoordinatesystem{tri}{\mytikztrics#1}
\def\mytikztri#1#2{(0:0) #2 #1 (0:1) #2 #1 (60:1) #2}
\def\mytikzscale#1#2#3{%
	\scope[shift={(0:0)},scale=1/2]#1\endscope
	\scope[shift={(0:0.5)},scale=1/2]#2\endscope
	\scope[shift={(60:0.5)},scale=1/2]#3\endscope}
\def\mytikzsg#1#2#3{%
	\def\first{#1}\def\second{#2}
	\ifx\first\second
		#3
	\else
		\mytikzscale{\mytikzsg{#1x}{#2}{#3}}{\mytikzsg{#1x}{#2}{#3}}{\mytikzsg{#1x}{#2}{#3}}
	\fi}
\def\mytikzfill{black!40}
\def\mytikzfilla{\mytikzfill}
\def\mytikzfillb{\mytikzfill}
\def\mytikzfillc{\mytikzfill}
\def\mytikztype#1{%
	\ifcase#1
		\fill[\mytikzfilla] (0,0) -- (1,0) -- (60:1) -- cycle;
	\or
		\fill[\mytikzfilla] (0,0) -- (0.5,0) -- (60:0.5) -- cycle;
		\fill[\mytikzfillb] (0.7,0) -- (1,0) -- (60:1) -- (60:0.7) -- cycle;
	\or
		\fill[\mytikzfilla] (0.5,0) -- (1,0) -- (tri cs:[0.5,0.5]) -- cycle;
		\fill[\mytikzfillb] (0,0) -- (0.3,0) -- (tri cs:[0.3,0.7]) -- (60:1) -- cycle;
	\or
		\fill[\mytikzfilla] (60:0.5) -- (tri cs:[0.5,0.5]) -- (60:1) -- cycle;
		\fill[\mytikzfillb] (0,0) -- (1,0) -- (tri cs:[0.7,0.3]) -- (60:0.3) -- cycle;
	\or
		\fill[\mytikzfilla] (0,0) -- (0.4,0) -- (tri cs:[4/15,4/15]) -- (60:0.4) -- cycle;
		\fill[\mytikzfillb] (1,0) -- (0.6,0) -- (tri cs:[7/15,4/15]) -- (tri cs:[0.6,0.4]) -- cycle;
		\fill[\mytikzfillc] (60:1) -- (60:0.6) -- (tri cs:[4/15,7/15]) -- (tri cs:[0.4,0.6]) -- cycle;
	\or
		\fill[\mytikzfilla] (0,0) -- (1,0) -- (tri cs:[0.7,0.3]) -- (tri cs:[0.3,0.3]) --
					(tri cs:[0.3,0.7]) -- (60:1) -- cycle;
	\or
		\fill[\mytikzfilla] (0,0) -- (1,0) -- (60:1) -- (60:0.7) --
					(tri cs:[0.4,0.3]) -- (60:0.3) -- cycle;
	\or
		\fill[\mytikzfilla] (0,0) -- (0.3,0) -- (tri cs:[0.3,0.4]) -- (0.7,0) --
					(1,0) -- (60:1) -- cycle;
	\fi}
\def\mytikzconn#1{%
	\ifcase#1
	\or
		\fill[black] (0.87,0) -- (1,0) -- (60:1) -- (60:0.87) -- cycle;
	\or
		\fill[black] (0,0) -- (0.13,0) -- (tri cs:[0.13,0.87]) -- (60:1) -- cycle;
	\or
		\fill[black] (0,0) -- (1,0) -- (tri cs:[0.87,0.13]) -- (60:0.13) -- cycle;
	\or
	\or
		\fill[black] (tri cs:[0.87,0.13]) -- (1,0) -- (0.235,0) -- (60:0.235) -- (60:1) --
				(tri cs:[0.13,0.87]) -- (tri cs:[0.13,0.235]) -- (tri cs:[0.235,0.13]) -- cycle;
	\or
		\fill[black] (60:0.13) -- (tri cs:[0.635,0.13]) -- (tri cs:[0.635,0.235]) -- (60:0.87) -- (60:1) --
				(tri cs:[0.765,0.235]) -- (0.765,0) -- (0,0) -- cycle;
	\or
		\fill[black] (0.13,0) -- (tri cs:[0.13,0.635]) -- (tri cs:[0.235,0.635]) -- (0.87,0) -- (1,0) --
				(tri cs:[0.235,0.765]) -- (60:0.765) -- (0,0) -- cycle;
	\fi}
\def\img#1{{\def\mytikzfill{black}\tikzpicture[scale=\mytikzglyphsize,baseline=\mytikzglyphbase]\mytikztype#1\endtikzpicture}}
\def\conn#1#2{\tikzpicture[scale=\mytikzglyphsize,baseline=\mytikzglyphbase]\mytikztype#1\mytikzconn#2\endtikzpicture}
\def\comp#1#2#3#4{%
	\tikzpicture[scale=\mytikzglyphsize,baseline=\mytikzglyphbase]
		\def\cmp{y}
		\def\arg{#2}\ifx\arg\cmp\def\mytikzfilla{black}\fi
		\def\arg{#3}\ifx\arg\cmp\def\mytikzfillb{black}\fi
		\def\arg{#4}\ifx\arg\cmp\def\mytikzfillc{black}\fi
		\mytikztype#1
	\endtikzpicture}
\def\tree#1#2#3#4#5#6{%
	\tikzpicture[scale=\mytikzglyphsize,baseline=\mytikzglyphbase]
		\def\cmp{y}
		\def\arg{#1}\ifx\arg\cmp\def\ca{}\else\def\ca{transparent}\fi
		\def\arg{#2}\ifx\arg\cmp\def\cb{}\else\def\cb{transparent}\fi
		\def\arg{#3}\ifx\arg\cmp\def\cc{}\else\def\cc{transparent}\fi
		\draw (0,0) node[vertex,\ca] {} (1,0) node[vertex,\cb] {} (60:1) node[vertex,\cc] {};
		\def\arg{#4}\ifx\arg\cmp\draw (1,0) {} -- (60:1) {};\fi
		\def\arg{#5}\ifx\arg\cmp\draw (60:1) {} -- (0,0) {};\fi
		\def\arg{#6}\ifx\arg\cmp\draw (0,0) {} -- (1,0) {};\fi
	\endtikzpicture}

\tikzstyle{walk}=[circle, inner sep=0pt, minimum width=4pt]

\def\mywalksize{0.25}
\def\mywalk[#1,#2]#3#4{
	\pgfmathparse{((#1)+(#2)/2)*\mywalksize}\xdef\myxpos{\pgfmathresult}
	\pgfmathparse{(#2)*sqrt(3)/2*\mywalksize}\xdef\myypos{\pgfmathresult}
	\draw[#4,line width=2pt] (\myxpos,\myypos) node[walk,fill=#4] {};
	\foreach \i in {#3} {
		\edef\oldxpos{\myxpos}\edef\oldypos{\myypos}
		\pgfmathparse{\oldxpos+\mywalksize*cos(60*\i)}\xdef\myxpos{\pgfmathresult}
		\pgfmathparse{\oldypos+\mywalksize*sin(60*\i)}\xdef\myypos{\pgfmathresult}
		\draw[#4,line width=2pt] (\oldxpos,\oldypos) -- (\myxpos,\myypos) node[walk,fill=#4] {};
	}}

\def\mytikzg#1#2#3#4{%
	\scope[shift={#1}]
		\mytikzsg{}{#2}{#3}
		\node[above left] at (60:1/2) {#4};
	\endscope}

\def\mytikzscaleb#1#2#3#4#5#6{%
	\scope[shift={(0:0)},scale=1/3]#1\endscope
	\scope[shift={(0:0.33333)},scale=1/3]#2\endscope
	\scope[shift={(60:0.33333)},scale=1/3]#3\endscope
	\scope[shift={(0:0.66667)},scale=1/3]#4\endscope
	\scope[shift={(30:0.57735)},scale=1/3]#5\endscope
	\scope[shift={(60:0.66667)},scale=1/3]#6\endscope}

\def\mytikzsgb#1#2#3{%
	\def\first{#1}\def\second{#2}
	\ifx\first\second
		#3
	\else
		\mytikzscaleb{\mytikzsgb{#1x}{#2}{#3}}{\mytikzsgb{#1x}{#2}{#3}}{\mytikzsgb{#1x}{#2}{#3}}{\mytikzsgb{#1x}{#2}{#3}}{\mytikzsgb{#1x}{#2}{#3}}{\mytikzsgb{#1x}{#2}{#3}}
	\fi}

\def\mytikzscalec#1#2#3#4#5{%
	\scope[shift={(0:0)},scale=1/3]#1\endscope
	\scope[shift={(0:0.33333)},scale=1/3]#2\endscope
	\scope[shift={(0:0.66667)},scale=1/3]#3\endscope
	\scope[shift={(0:0.33333)},scale=1/3,rotate=60]#4\endscope
	\scope[shift={(30:0.57735)},scale=1/3,rotate=300]#5\endscope}

\def\mytikzkoch#1#2#3{%
	\def\first{#1}\def\second{#2}
	\ifx\first\second
		#3
	\else
		\mytikzscalec{\mytikzkoch{#1x}{#2}{#3}}{\mytikzkoch{#1x}{#2}{#3}}{\mytikzkoch{#1x}{#2}{#3}}{\mytikzkoch{#1x}{#2}{#3}}{\mytikzkoch{#1x}{#2}{#3}}{\mytikzkoch{#1x}{#2}{#3}}
	\fi}

\begin{document}

\title{Uniform spanning trees on Sierpi\'nski graphs}

\author{Masato Shinoda}
\address{Masato Shinoda\\
	Department of Mathematics\\
	Nara Women's University\\
	Kitauoya-Nishimachi\\
	Nara city, Nara 630-8506\\
	Japan}
\email{shinoda@cc.nara-wu.ac.jp}

\author{Elmar Teufl}
\address{Elmar Teufl\\
	Mathematisches Institut\\
	Eberhard Karls Universit\"at T\"ubingen\\
	Auf der Morgenstelle 10\\
	72076 T\"ubingen\\
	Germany}
\email{elmar.teufl@uni-tuebingen.de}

\author{Stephan Wagner}
\address{Stephan Wagner\\
	Department of Mathematical Sciences\\
	Stellenbosch University\\
	Private Bag X1\\
	Matieland 7602\\
	South Africa}
\email{swagner@sun.ac.za}


\begin{abstract}
We study spanning trees on Sierpi\'nski graphs (i.e., finite approximations to the Sierpi\'nski gasket) that are chosen uniformly at random. We construct a joint probability space for uniform spanning trees on every finite Sierpi\'nski graph and show that this construction gives rise to a multi-type Galton-Watson tree. We derive a number of structural results, for instance on the degree distribution. The connection between uniform spanning trees and loop-erased random walk is then exploited to prove convergence of the latter to a continuous stochastic process. Some geometric properties of this limit process, such as the Hausdorff dimension, are investigated as well. The method is also applicable to other self-similar graphs with a sufficient degree of symmetry.
\end{abstract}

\maketitle

\tableofcontents

\section{Introduction}
\label{section:intro}

The Sierpi\'nski gasket is certainly one of the most famous fractals, and the Sierpi\'n\-ski graphs, which can be seen as finite approximations of the Sierpi\'nski gasket, are among the most thoroughly studied self-similar graphs. The number of spanning trees in the $n$\ndash th Sierpi\'nski graph $G_n$ (starting with a single triangle $G_0$, see Figure~\ref{figure:sg}) turns out to be given by the remarkable explicit formula
\[ \tau(G_n) = \biggl( \frac{3}{20} \biggr)^{1/4} \cdot \biggl( \frac35 \biggr)^{n/2} \cdot 540^{3^n/4}, \]
which was obtained by different methods in several recent works: by setting up and solving a system of recursions \cite{chang2007trees,teufl2006number,teufl2011number}, or by electrical network theory \cite{teufl2011resistance}. In \cite{lawler2010random}, a proof using probabilistic results is sketched. Moreover, the Laplacian spectrum of $G_n$ can be described rather explicitly by means of a technique known as ``spectral decimation'' \cite{fukushima1992spectral,shima1991eigenvalue}, from which another proof can be derived \cite{anema2012counting}.

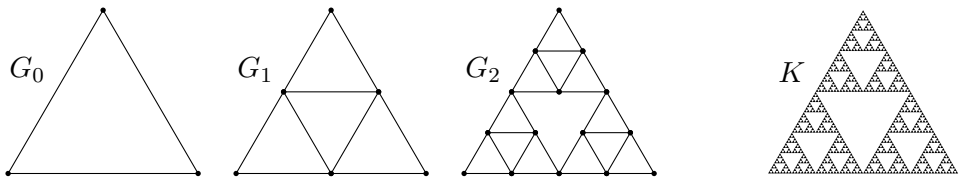
\begin{figure}[htb]
\centering
\def\fig#1#2#3#4{%
	\scope[shift={#1}]
		\mytikzsg{}{#2}{#3}
		\node[above left] at (60:1/2) {#4};
	\endscope}
\begin{tikzpicture}[scale=2.5]
	\fig{(0,0)}{}{\draw \mytikztri{--}{node[vertex] {}} -- cycle;}{$G_0$}
	\fig{(1.2,0)}{x}{\draw \mytikztri{--}{node[vertex] {}} -- cycle;}{$G_1$}
	\fig{(2.4,0)}{xx}{\draw \mytikztri{--}{node[vertex] {}} -- cycle;}{$G_2$}
	\fig{(4,0)}{xxxxxx}{\fill \mytikztri{--}{} -- cycle;}{$K$}
\end{tikzpicture}
\caption{Sierpi\'nski graphs $G_0$, $G_1$, $G_2$, and the Sierpi\'nski gasket $K$.}
\label{figure:sg}
\end{figure}

Once the counting problem is solved, it is natural to consider \emph{uniformly random spanning trees} of $G_n$ and to study their structure. Uniform spanning trees are known to have strong connections to other probabilistic models, such as loop-erased random walk (Wilson's celebrated algorithm \cite{lawler2010random,wilson1996generating} to construct uniform spanning trees being a particular application), and they are also of interest in mathematical physics. For this reason, the structure of uniformly random spanning trees in other important families of graphs, such as square grids \cite{burton1993local} has been studied thoroughly.

The recursive nature of Sierpi\'nski graphs and the strong symmetry enables us to derive a number of results on uniform spanning trees, as will be shown in this paper. After some preliminaries, we construct a joint probability space for uniform spanning trees on every finite Sierpi\'nski graph. An important tool in this context is the theory of (a rather general kind of) Galton-Watson processes. Making use of this tool, we also prove some structural results on uniform spanning trees of $G_n$, for instance a strong law of large numbers for the degree distribution of a uniform spanning tree. This extends the work of Chang and Chen \cite{chang2010structure}, who prove convergence of expected values (for which they also give explicit formulae). Similar results for the two-dimensional square lattice were obtained by Manna, Dhar and Majumdar \cite{manna1992spanning}.

\emph{Loop-erased random walk} on the Sierpi\'nski gasket was studied in the paper of Hattori and Mizuno \cite{hattori2012looperased}; our results on uniform spanning trees provide an alternative approach to this topic and were obtained independently of Hattori and Mizuno and approximately at the time (see for instance \cite{teufl2011uniform}). The expected length of such a walk from one corner to another was studied earlier in the physics literature by Dhar and Dhar \cite{dhar1997distribution}; it grows asymptotically like $\bigl(\frac43 + \frac1{15}\sqrt{205}\,\bigr)^n$. As it was also shown by Hattori and Mizuno, we find that, upon renormalization, loop-erased random walk converges to a limit process. The analogue of this process for the square lattice is the celebrated Schramm-Loewner evolution \cite{schramm2000scaling,lawler2004conformal}, whose analysis is notoriously complicated. However, the different geometry of the Sierpi\'nski graphs makes it possible to prove rather strong theorems on the shape of this limiting process comparatively easily, including parameters such as the \emph{Hausdorff dimension}. Similar results on the limit process of the self-avoiding walk were obtained by Hattori, Hattori and Kusuoka \cite{hattori1991selfavoiding,hattori1992exponent,hattori1990selfavoiding,hattori1993selfavoiding} and by Hambly, Hattori and Hattori \cite{hambly2002selfrepelling} for the self-repelling walk.

In Section~\ref{sec:metric}, we study the metric induced by a random spanning tree on the Sierpi\'nski graph $G_n$. We prove almost sure convergence to a limit metric, and show that the resulting metric space is a so-called $\R$\ndash tree. We also study the \emph{interface}, which is, loosely speaking, the set where different branches of a spanning tree embedded in the plane ``touch'', and estimate its Hausdorff dimension.

In the following list, the main results of this paper are summarised. For the sake of simplicity, all results and their derivation are only given for the Sierpi\'nski gasket, but there are other fractals to which the same approach applies, see Section~\ref{sec:other}.

\begin{itemize}
\item We construct a joint probability space for uniform spanning trees
	on every finite Sierpi\'nski graph using a projective limit.
	As part of the construction, we also have to consider
	spanning forests with the property
	that each of the components contains one of the three corner vertices.
	We show that the distribution of the component sizes
	in random spanning forests of this type converges (upon renormalisation)
	to a limiting distribution---see Section~\ref{subsection:component}.
\item We prove almost sure convergence of the \emph{degree distribution}
	(see Section~\ref{subsection:degree}):
	the proportion of vertices of degree $i$ ($i \in \{1,2,3,4\}$ fixed)
	in a random spanning tree of $G_n$ converges almost surely
	to a limit constant $w(i)$.
\item Section~\ref{section:lerw} is concerned with loop erased random walk
	on Sierpi\'nski graphs $G_n$:
	using the connection between spanning trees and
	loop-erased random walk, we recover the aforementioned result
	that the length of such a walk from one corner to another grows
	asymptotically like $\bigl(\frac43 + \frac1{15}\sqrt{205}\,\bigr)^n$,
	and that the renormalised length has a limit distribution
	(cf.~\cite[Theorem~5]{hattori2012looperased}).
	We also provide tail estimates for this limit distribution,
	see Lemma~\ref{lemma:mg-bounds}.
\item In Section~\ref{section:convergence}, we study the limit process and
	prove some geometric properties:
	specifically, we show that the limit curve is
	almost surely self-avoiding (Theorem~\ref{theorem:prop}),
	and has Hausdorff dimension
	$\log_2 \bigl(\frac43 + \frac1{15}\sqrt{205}\,\bigr) \approx 1.193995$
	(Theorem~\ref{theorem:prop-lerw}, (5)).
	These results were also obtained in the aforementioned paper
	of Hattori and Mizuno (see \cite[Theorems 9 and 10]{hattori2012looperased}).
	Moreover, we prove H\"older continuity with an
	explicit exponent (Theorem~\ref{theorem:prop-lerw}, (4)).
\item The limit of the tree metric is the main topic of Section~\ref{sec:metric}.
	It is shown (Theorem~\ref{theorem:randommetric}) that
	we almost surely obtain a random metric on the ``rational points''
	(i.e., all points which are vertices in some finite approximation)
	of the Sierpinski gasket whose Cauchy completion is an $\R$\ndash tree,
	i.e., a metric space in which there is a unique arc between any two points and
	this arc is geodesic (that is an isometric embedding of a real interval).
\end{itemize}

\section{Notation and Preliminaries}
\label{section:notation}

A \emph{graph} $G$ is a pair $(VG,EG)$, where $VG$ is the vertex set and
\[ EG \subseteq \bigl\{ \{x,y\} \,:\, x,y\in VG, x\ne y \bigr\} \]
is the edge set. Two vertices $x,y$ are \emph{adjacent} if $\{x,y\}\in EG$. The \emph{degree} $\deg x$ of a vertex $x$ is the number of adjacent vertices. A \emph{walk} in $G$ is a finite or infinite sequence $(x_0,x_1,\dotsc)$ of vertices in $G$, such that consecutive entries are adjacent. A walk is called \emph{self-avoiding} if its entries are mutually distinct. The edge set $E(x)$ of a walk $x=(x_0,x_1,\dotsc)$ is the set
\[ E(x) = \bigl\{ \{x_0,x_1\}, \{x_1,x_2\}, \dotsc \bigr\}. \]
Equipped with the edge set $E(x)$ a walk $x$ gives rise to a subgraph of $G$. The \emph{length} of the walk $(x_0,\dotsc,x_n)$ is equal to $n$, the number of edges. The \emph{distance} $d_G(v,w)$ of two vertices $v,w$ is the least integer $n$ such that there is a walk of length $n$ in $G$ connecting $v$ and $w$.

A \emph{tree} is a connected graph without cycles. A \emph{spanning tree} of a graph $G$ is a subgraph of $G$ which is a tree and contains all vertices of $G$. Similarly, a \emph{forest} is a graph without cycles and a \emph{spanning forest} of a graph $G$ is a subgraph of $G$ which is a forest and contains all vertices of $G$. Let $F$ be a forest and $v,w$ be two vertices in $F$. If $v,w$ are in the same component of $F$, then we write $vFw$ to denote the unique self-avoiding walk in $F$ connecting $v$ and $w$.

Next we need some ingredients from probability theory. We use multi-index notation: If $r\in\N$, $\mathbold z=(z_1,\dotsc,z_r)$, and $\mathbold k=(k_1,\dotsc,k_r)$, then $\mathbold z^\mathbold k = z_1^{k_1} \dotsm z_r^{k_r}$. If $\mathbold X=(X_1,\dotsc,X_r)$ is a random vector in $\N_0^r$, then
\[ \PGF(\mathbold X, \mathbold z)
	= \Expect\bigl(\mathbold z^{\mathbold X}\bigr)
	= \sum_{\mathbold k\in\N_0^r} \Prob(\mathbold X=\mathbold k) \mathbold z^\mathbold k \]
is the (multivariate) \emph{probability generating function} of $\mathbold X$.

An \emph{$r$\ndash type Galton-Watson process} $(\mathbold X_n)_{n\ge0} = (X_{1,n},\dotsc,X_{r,n})_{n\ge0}$ is a stochastic process that starts with one or more individuals, each of which has a type associated to it. Each individual gives birth to zero or more children according to the offspring probabilities
\[ \mathbold p(\mathbold k) = \bigl( p_1(\mathbold k), \dotsc, p_r(\mathbold k) \bigr) \]
for $\mathbold k=(k_1,\dotsc,k_r)\in\N_0^r$. Here $p_i(\mathbold k)$ is the probability that an individual of type $i$ has $k_j$ children of type $j$ for $j\in\{1,\dotsc,r\}$. The vector $\mathbold X_n$ represents the number of individuals in the $n$\ndash th generation by their type (i.e., $X_{i,n}$ is the number of individuals of type $i$ in the $n$\ndash th generation). It is convenient to describe the offspring distributions by their multidimensional multivariate probability generating function $\mathbold f=(f_1,\dotsc,f_r)$, which is called \emph{offspring generating function} and given by
\[ \mathbold f(\mathbold z) = \sum_{\mathbold k\in\N_0^r} \mathbold p(\mathbold k) \mathbold z^\mathbold k \]
for $\mathbold z=(z_1,\dotsc,z_r)$. Then
\[ \PGF(\mathbold X_n,\mathbold z)
	= \PGF(\mathbold X_{n-1}, \mathbold f(\mathbold z)) = \dotsb
	= \PGF(\mathbold X_0, \mathbold f^n(\mathbold z)), \]
where $\mathbold f^n$ is the $n$\ndash fold iteration of $\mathbold f$. The \emph{mean matrix} $\mathbold M = (m_{ij})_{1\le i,j\le r}$ is given by $m_{ij} = (\partial f_i/\partial z_j)(\mathbf 1)$. The process is called
\begin{itemize}
\item \emph{positively regular} if $\mathbold M$ is primitive
	(i.e., all entries of $\mathbold M^k$ are positive for some $k$),
\item \emph{singular} if $\mathbold f(\mathbold z) = \mathbold M \mathbold z$,
\item \emph{subcritical}, \emph{critical} or \emph{supercritical} depending on whether the largest eigenvalue of the mean matrix is less than, equal to, or greater than $1$.
\end{itemize}
see for instance \cite{mode1971multitype} for these notions and the theory of multi-type Galton-Watson processes.

Let $\Words$ be a subset of $\N$ (in the following $\Words$ will always be $\{1,2,3\}$); elements of $\Words^n$ are written as words over the alphabet $\Words$, e.g.\ $12133$ means $(1,2,1,3,3)$. Let $\Words^0 = \{\emptyword\}$ consist of the empty word $\emptyword$ only and set
\[ \Words^* = \biguplus_{n\ge0} \Words^n. \]
Concatenation of two words $v,w\in\Words^*$ is written by juxtaposition $vw$. The set $\Words^*$ carries a graph structure in a canonical way: two words $v,w\in\Words^*$ are adjacent if $w=v\iota$ or $v=w\iota$ for some $\iota\in\Words$. This turns $\Words^*$ into a tree with root $\emptyword$. If $w=v\iota$, then $w$ is called \emph{child} or \emph{offspring} of $v$ and $\iota$ is the \emph{suffix} of $w$.

Let $R$ be a finite set and fix some $\Words\subseteq\N$. Consider the set of all subtrees with an element of $R$ associated to each vertex, i.e.,
\[ \mathcal W_R = \bigl\{ (W,f) \,:\,
	W\subseteq\Words^* \text{ induces a subtree},
	\emptyword\in W,
	f=(f_w)_{w\in W}\in R^W \bigr\}. \]
If $(W,f)\in\mathcal W_R$ and $w\in W$, then we say that the word $w$ is the \emph{label} and $f_w$ is the \emph{type} of the pair $(w,f_w)$. A \emph{labelled multi-type Galton-Watson tree} with labels in $\Words^*$ and types in $R$ is a random element of the set $\mathcal W_R$, whose distribution is uniquely determined by the following:
\begin{itemize}
\item The type of the root (or \emph{ancestor}) $\emptyword$ is given by a fixed distribution on $R$.
\item The random offspring generation of a vertex (or \emph{individual}) $w$ only depends on the type of $w$.
	It is given by a probability distribution (depending on the type of $w$) on the set
	\[ \bigl\{ (S,g) \,:\, S\subseteq\Words, g=(g_\iota)_{\iota\in S}\in R^S \bigr\}. \]
	The interpretation is that once a pair $(S,g)$ is chosen,
	the individual $w$ gives birth to $\card{S}$ children
	with labels $w\iota$ for $\iota\in S$, and type $g_\iota$ is assigned to child $w\iota$.
\end{itemize}
A labelled multi-type Galton-Watson tree is denoted by $(F_w)_{w\in W}\in\mathcal W_R$. Notice that in this notation $W\subseteq\Words^*$ is the random set of individuals and $F_w$ is the random type of an individual $w\in W$.

To every labelled multi-type Galton-Watson tree $(F_w)_{w\in W}$ with labels in $\Words^*$ and types in $R$, the (random) number of individuals of a certain type in the $n$\ndash th generation yields a multi-type Galton-Watson process with $r=\card{R}$ types. To this end, let $a_1,\dotsc,a_r$ be the elements of $R$ and set
\[ X_{i,n} = \card{\{ w\in W\cap\Words^n \,:\, F_w = a_i \}}, \qquad
	\mathbold X_n = (X_{1,n},\dotsc,X_{r,n}) \]
for $n\ge0$ and $i\in\{1,\dotsc,r\}$. Then $(\mathbold X_n)_{n\ge0}$ is an $r$\ndash type Galton-Watson processes.

\section{Construction of Sierpi\'nski graphs}
\label{section:sg}

The Sierpi\'nski gasket $K$ (see \cite{sierpinski1915courbe} for its origin in mathematical literature) can be defined formally by means of the following three similitudes:
\[ \psi_i(x) = \tfrac12(x-u_i) + u_i \]
for $i\in\{1,2,3\}$, where $u_1=(0,0)$, $u_2=(1,0)$, and $u_3=\tfrac12(1,\sqrt3)$. Then $K$ is the unique non-empty compact set such that
\[ K = \psi_1(K) \cup \psi_2(K) \cup \psi_3(K). \]
Its Hausdorff dimension is given by
\[ \operatorname{dim}_H K = \frac{\log3}{\log2} = 1.5849625\dotsc \]
The Sierpi\'nski graphs $G_0, G_1, \dotsc$ are discrete approximations to $K$ and are constructed inductively: The vertex set $VG_0$ and edge set $EG_0$ of $G_0$ are given by
\[ VG_0 = \{u_1, u_2, u_3\} \qquad\text{and}\qquad EG_0 = \{\{u_1,u_2\},\{u_2,u_3\},\{u_3,u_1\}\}, \]
respectively. Then, for any $n\ge0$, the sets $VG_{n+1}$ and $EG_{n+1}$ are defined as follows:
\begin{align*}
VG_{n+1} &= \psi_1(VG_n) \cup \psi_2(VG_n) \cup \psi_3(VG_n), \\
EG_{n+1} &= \psi_1(EG_n) \cup \psi_2(EG_n) \cup \psi_3(EG_n).
\end{align*}
Notice that $G_{n+1}$ is an amalgam of three scaled images of $G_n$, which we denote by $\psi_1(G_n)$, $\psi_2(G_n)$, and $\psi_3(G_n)$, i.e.,
\[ G_{n+1} = \psi_1(G_n) \cup \psi_2(G_n) \cup \psi_3(G_n). \]
The vertices in $VG_0\subset VG_n$ are often called \emph{corner vertices} or \emph{boundary vertices} of the graph $G_n$. The vertex sets are nested, i.e., $VG_0\subset VG_1\subset VG_2 \subset \dotsb$, and the Sierpi\'nski gasket $K$ is the closure of the union $VG_0\cup VG_1\cup VG_2\dotsb$. Figure~\ref{figure:sg} shows the graphs $G_0$, $G_1$, $G_2$ and the Sierpi\'nski gasket $K$. The self-similar nature and the fact that the three scaled images only intersect in the three points
\[ \tfrac12(u_2+u_3), \qquad \tfrac12(u_3+u_1), \qquad \tfrac12(u_1+u_2), \]
allows to solve many problems concerning the Sierpi\'nski gasket and Sierpi\'nski graphs exactly.

As explained above, we may view $G_n$ as an amalgam of three copies of $G_{n-1}$. More generally, we may consider $G_n$ as an amalgam of $3^{n-k}$ copies of $G_k$ ($0\le k\le n$). For any word $w=w_1\dotsb w_n\in\Words^n$ ($n\ge1$), set $\psi_w = \psi_{w_1} \circ \dotsb \circ \psi_{w_n}$, and let $\psi_\emptyword$ be the identity map. Then, for $0\le k\le n$,
\[ G_n = \bigcup_{w\in\Words^k} \psi_w(G_{n-k}). \]
If $w\in\Words^k$, we call $\psi_w(G_{n-k})$ (respectively $\psi_w(K)$) a \emph{$k$\ndash part} of $G_n$ (respectively $K$). Note that the $k$\ndash parts are in one-to-one correspondence with the words in $\Words^k$. For any word $w\in\Words^k$ and any subgraph $H\subseteq G_n$ the \emph{restriction} $\pi_w(H)$ is the subgraph of $G_{n-k}$ given by
\[ \pi_w(H) = \psi_w^{-1}(H\cap\psi_w(G_{n-k})). \]

\section{Spanning trees on Sierpi\'nski graphs}
\label{section:trees}

For the sake of completeness we reproduce the computation of the number of spanning trees of the Sierpi\'nski graphs following the approach given in \cite{chang2007trees,teufl2006number,teufl2011number}. Using a decomposition of certain spanning forests a recursion for the number of spanning trees and two other quantities is derived. In the physics literature this approach is often called the renormalization group. We write
\begin{itemize}
\item $\mathcal T_n$ to denote the set of spanning trees of $G_n$,
\item $\mathcal S_n^i$ ($i\in\{1,2,3\}$) to denote
	the set of spanning forests in $G_n$ with two connected components,
	so that one component contains $u_i$ and the other contains $VG_0\setminus\{u_i\}$, and
\item $\mathcal R_n$ to denote the set of spanning forests of $G_n$ with three connected components,
	each of which contains exactly one vertex from the set $VG_0$.
\end{itemize}
By symmetry, the sets $\mathcal S_n^1$, $\mathcal S_n^2$, and $\mathcal S_n^3$ all have the same cardinality. For notational convenience, we set
\[ \mathcal Q_n = \mathcal T_n \uplus \mathcal S_n^1 \uplus \mathcal S_n^2 \uplus \mathcal S_n^3 \uplus \mathcal R_n. \]
The crucial observation is that the restriction of a spanning forest in $\mathcal Q_{n+1}$ to one of $\psi_1(G_n)$, $\psi_2(G_n)$, or $\psi_3(G_n)$ can be identified with a spanning forest in $\mathcal Q_n$. If $f\in\mathcal Q_{n+1}$, then $\pi_1(f), \pi_2(f), \allowbreak \pi_3(f) \in \mathcal Q_n$ and
\begin{equation}\label{eq:restrict}
f = \psi_1(\pi_1(f)) \cup \psi_2(\pi_2(f)) \cup \psi_3(\pi_3(f)).
\end{equation}
Here and in the following we use lowercase letters for elements of $\mathcal Q_n$ and capital letters for random elements of $\mathcal Q_n$. Since $\mathcal T_0$ consists of the three elements \tree yyynyy, \tree yyyyny, \tree yyyyyn, whereas $\card{\mathcal S_0^i}=\card{\mathcal R_0^{}}=1$ for $i\in\{1,2,3\}$, the subdivision of $\mathcal T_n$ into three families of equal size turns out to be advantageous. In the following we describe one subdivision which is convenient and induced by symmetry. Set
\[ \mathcal T_0^1 = \{\tree yyynyy\}, \qquad
	\mathcal T_0^2 = \{\tree yyyyny\}, \qquad
	\mathcal T_0^3 = \{\tree yyyyyn\}, \]
and in general, for $n\ge1$ and $i\in\{1,2,3\}$,
\[ \mathcal T_n^i = \bigl\{ t\in\mathcal T_n \,:\,
	u_itu_j \subseteq \psi_i(G_{n-1})\cup\psi_j(G_{n-1}) \text{ for all } j \in \{1,2,3\}\setminus\{i\} \bigr\}. \]
Here we consider the self-avoiding walk $u_itu_j$ as the subgraph consisting of the vertices and the edges connecting consecutive vertices. In words, $\mathcal T_n^i$ is the set of spanning trees with the property that the unique paths from $u_i$ to the other corner vertices $u_j$, $j \neq i$, only pass through $\psi_i(G_{n-1})$ and $\psi_j(G_{n-1})$ and do not ``make a detour''. Then
\[ \mathcal T_n^{} = \mathcal T_n^1 \uplus \mathcal T_n^2 \uplus \mathcal T_n^3 \]
and $\card{\mathcal T_n^{}} = 3\,\card{\mathcal T_n^i}$ for $i\in\{1,2,3\}$. Define
\[ \tau_n = \card{\mathcal T_n^1} = \card{\mathcal T_n^2} = \card{\mathcal T_n^3}, \qquad
	\sigma_n = \card{\mathcal S_n^1} = \card{\mathcal S_n^2} = \card{\mathcal S_n^3}, \qquad
	\rho_n = \card{\mathcal R_n^{}}. \]

\begin{lemma}\label{lemma:count}
If $n\ge0$, then
\begin{align*}
\tau_{n+1}   &= 18 \tau_n^2 \sigma_n, \\
\sigma_{n+1} &= 21 \tau_n \sigma_n^2 + 9 \tau_n^2 \rho_n, \\
\rho_{n+1}   &= 14 \sigma_n^3 + 36 \tau_n \sigma_n \rho_n,
\end{align*}
and
\begin{align*}
\tau_n   &= \bigl(\tfrac53\bigr)^{-n/2} \, 540^{(3^n-1)/4}, \\
\sigma_n &= \bigl(\tfrac53\bigr)^{ n/2} \, 540^{(3^n-1)/4}, \\
\rho_n   &= \bigl(\tfrac53\bigr)^{3n/2} \, 540^{(3^n-1)/4}.
\end{align*}
\end{lemma}

\begin{proof}
The recursion satisfied by $\tau_n$, $\sigma_n$, $\rho_n$ follows from the decomposition \eqref{eq:restrict}. For a graphical explanation of the specific terms, see Figure~\ref{figure:rectree}. The initial values are $(\tau_0,\sigma_0,\rho_0)=(1,1,1)$, and using induction, it is easy to verify that the formulae stated in the lemma are indeed the explicit solution to the recursion.
\end{proof}

\begin{figure}[htb]
\centering
\def\figscale{1.3}
\def\fig#1#2#3#4#5{%
	\scope[shift={#1}]
		\mytikzscale{\mytikztype#3}{\mytikztype#4}{\mytikztype#5}
		\mytikzsg{}{x}{\draw \mytikztri{}{node[vertex] {}};}
		\node[right,font=\footnotesize] at (tri cs:[0.5,0.5]) {#2};
	\endscope}
\begin{tikzpicture}[scale=\figscale]
	\scope[shift={(0,0)}]
		\fig{(0.5,0)}{$\times 2$}200
		\draw (1.0,-0.05) node[anchor=base]
			{$\underbrace{\hbox{\pgfmathparse{\figscale * 1.0 + 0.4}\hskip\pgfmathresult cm}}$};
		\draw (1.0,-0.6) node[anchor=base] {$2\cdot(3\tau_n)^2 \sigma_n$};

		\draw (0,-0.6) node[anchor=base east] {$\tau_{n+1} = $};
	\endscope

	\scope[shift={(0,-2)}]
		\fig{(0.0,0)}{$\times 2$}031
		\fig{(1.2,0)}{$\times 2$}013
		\fig{(2.4,0)}{$\times 2$}033
		\fig{(3.6,0)}{$\times 1$}330
		\draw (2.3,-0.05) node[anchor=base]
			{$\underbrace{\hbox{\pgfmathparse{\figscale * 4.6 + 0.4}\hskip\pgfmathresult cm}}$};
		\draw (2.3,-0.6) node[anchor=base] {$7\cdot(3\tau_n) \sigma_n^2$};

		\draw (5.0,-0.6) node[anchor=base] {$+$};

		\fig{(5.4,0)}{$\times 1$}004
		\draw (5.9,-0.05) node[anchor=base]
			{$\underbrace{\hbox{\pgfmathparse{\figscale * 1 + 0.4}\hskip\pgfmathresult cm}}$};
		\draw (5.9,-0.6) node[anchor=base] {$(3\tau_n)^2 \rho_n$};

		\draw (0,-0.6) node[anchor=base east] {$\sigma_{n+1} = $};
	\endscope

	\scope[shift={(0,-4)}]
		\fig{(0.0,0)}{$\times 6$}323
		\fig{(1.2,0)}{$\times 6$}322
		\fig{(2.4,0)}{$\times 2$}312
		\draw (1.7,-0.05) node[anchor=base]
			{$\underbrace{\hbox{\pgfmathparse{\figscale * 3.4 + 0.4}\hskip\pgfmathresult cm}}$};
		\draw (1.7,-0.6) node[anchor=base] {$14\cdot\sigma_n^3$};

		\draw (3.8,-0.6) node[anchor=base] {$+$};

		\fig{(4.2,0)}{$\times 6$}024
		\fig{(5.4,0)}{$\times 6$}014
		\draw (5.3,-0.05) node[anchor=base]
			{$\underbrace{\hbox{\pgfmathparse{\figscale * 2.2 + 0.4}\hskip\pgfmathresult cm}}$};
		\draw (5.3,-0.6) node[anchor=base] {$12\cdot(3\tau_n) \sigma_n \rho_n$};

		\draw (0,-0.6) node[anchor=base east] {$\rho_{n+1} = $};
	\endscope
\end{tikzpicture}
\caption{All arrangements (up to symmetry) for the construction of spanning trees and spanning forests
	used in the recursion for $\tau_n$, $\sigma_n$, and $\rho_n$.
	Shaded area indicates connected parts.}
\label{figure:rectree}
\end{figure}

We define the \emph{trace} $\Tr f\in\mathcal Q_n$ of a spanning forest $f\in\mathcal Q_{n+1}$ as follows: For $f\in\mathcal Q_1$, the trace is given in Table~\ref{table:trace}.

\begin{table}[htb]
\caption{Traces of spanning forests in $\mathcal Q_1$.}
\label{table:trace}
\centering
\begin{tabular}{@{}*{8}{c}@{}}
\toprule
$M$ & $\mathcal T_1^1$ & $\mathcal T_1^2$ & $\mathcal T_1^3$ %
    & $\mathcal S_1^1$ & $\mathcal S_1^2$ & $\mathcal S_1^3$ & $\mathcal R_1^{}$ \\
\midrule
$\Tr f$ for $f\in M$ %
    & \tree yyynyy & \tree yyyyny & \tree yyyyyn %
    & \tree yyyynn & \tree yyynyn & \tree yyynny & \tree yyynnn \\
\bottomrule
\end{tabular}
\end{table}

If $n>0$ and $f\in\mathcal Q_{n+1}$, then consider the $3^n$ $n$\ndash parts of $G_{n+1}$, which are isomorphic to $G_1$. On each of these parts $f$ induces (up to scaling and translation) a forest on $\mathcal Q_1$. In order to obtain the trace $\Tr f$, replace each of these small forests by their respective trace:
\[ \Tr f = \bigcup_{w\in\Words^n} \psi_w(\Tr \pi_w(f)). \]
Note that $\Tr$ maps $\mathcal T_{n+1}^i$ onto $\mathcal T_n^i$, $\mathcal S_{n+1}^i$ onto $\mathcal S_n^i$ ($i\in\{1,2,3\}$), and $\mathcal R_{n+1}^{}$ onto $\mathcal R_n^{}$. In order to emphasize the dependence on $n$, we write $\Tr^{n+1}_n t$ instead of $\Tr f$ if $f\in\mathcal Q_{n+1}$. For $m\ge n$ define $\Tr^m_n = \Tr^{n+1}_{n} \circ \dotsb \circ \Tr^m_{m-1}$. Then
\begin{align*}
\mathcal T_n^1 &= \bigl\{ t\in\mathcal T_n \,:\, \Tr^n_0 t = \tree yyynyy \bigr\}, \\
\mathcal T_n^2 &= \bigl\{ t\in\mathcal T_n \,:\, \Tr^n_0 t = \tree yyyyny \bigr\}, \\
\mathcal T_n^3 &= \bigl\{ t\in\mathcal T_n \,:\, \Tr^n_0 t = \tree yyyyyn \bigr\}.
\end{align*}

\begin{figure}[htb]
\centering
\begin{tikzpicture}[scale=3]
	\scope[shift={(0.0,0)}]
		\mytikzg{(0,0)}{xx}{\draw[dashed] \mytikztri{--}{node[vertex] {}} -- cycle;}{}
		\mywalk[0,0]{1,0,4,0,0,0}{black}
		\mywalk[1,1]{2,0,1,2,4}{black}
		\mywalk[1,3]{5,5}{black}\mywalk[2,2]{4}{black}
	\endscope
	\scope[shift={(1.4,0)}]
		\mytikzg{(0,0)}{x}{\draw[dashed] \mytikztri{--}{node[vertex] {}} -- cycle;}{}
		\def\mywalksize{0.5}
		\mywalk[0,0]{1,5,0}{black}\mywalk[0,1]{0,2}{black}
	\endscope
	\scope[shift={(2.8,0)}]
		\mytikzg{(0,0)}{}{\draw[dashed] \mytikztri{--}{node[vertex] {}} -- cycle;}{}
		\def\mywalksize{1}\mywalk[1,0]{3,1}{black}
	\endscope
	\draw (1.2,0.475) node {$\stackrel{\!\Tr}{\longmapsto}$};
	\draw (2.6,0.475) node {$\stackrel{\!\Tr}{\longmapsto}$};
\end{tikzpicture}
\caption{A spanning tree $t$ on $G_2$ and the traces $\Tr t = \Tr_1^2 t$ and $\Tr_0^2 t$.}
\label{figure:trace}
\end{figure}

Figure~\ref{figure:trace} shows a spanning tree on $G_2$ and its traces on $G_1$ and $G_0$. The importance of the trace stems from the fact that $(\mathcal Q_n, \Tr^m_n)$ is a projective system. Hence we can define $\mathcal Q_\infty = \varprojlim \mathcal Q_n$ and write $\Tr^\infty_n$ to denote the canonical projection from $\mathcal Q_\infty$ to $\mathcal Q_n$. Similarly, set
\[ \mathcal T_\infty^i = \varprojlim \mathcal T_n^i, \qquad
	\mathcal S_\infty^i = \varprojlim \mathcal S_n^i, \qquad
	\mathcal R_\infty^{} = \varprojlim \mathcal R_n^{} \]
for $i\in\{1,2,3\}$. Then
\[ \mathcal T_\infty^{} = \varprojlim \mathcal T_n^{}
	= \mathcal T_\infty^1 \uplus \mathcal T_\infty^2 \uplus \mathcal T_\infty^3 \]
and
\[ \mathcal Q_\infty^{} = \mathcal T_\infty \uplus
	\mathcal S_\infty^1 \uplus \mathcal S_\infty^2 \uplus \mathcal S_\infty^3 \uplus \mathcal R_\infty. \]
Let $w\in\Words^*$ be a word of length $n\ge0$ and let $f\in\mathcal Q_\infty$. Then
\[ \pi_w(f) = (\pi_w(\Tr^\infty_n f),\pi_w(\Tr^\infty_{n+1} f),\dotsc)\in\mathcal Q_\infty \]
extends the definition of the restriction operator $\pi_w$ to $\pi_w\colon\mathcal Q_\infty\to\mathcal Q_\infty$.

Next we define the \emph{type} of an element of $\mathcal Q_{\infty}$ (or a part of it). Set
\[ \mathcal C = \{\img5,\img6,\img7,\img1,\img2,\img3,\img4\}. \]
For $f\in\mathcal Q_\infty$ let $\type_\emptyword(f)\in\mathcal C$ be given by Table~\ref{table:type}. The symbol $\type_\emptyword(f)$ gives a crude indication of the shape of $f$.

\begin{table}[htb]
\caption{The definition of $\type_\emptyword(f)$ for $f\in\mathcal Q_\infty$.}
\label{table:type}
\centering
\begin{tabular}{@{}*{8}{c}@{}}
\toprule
$M$ & $\mathcal T_\infty^1$ & $\mathcal T_\infty^2$ & $\mathcal T_\infty^3$ %
    & $\mathcal S_\infty^1$ & $\mathcal S_\infty^2$ & $\mathcal S_\infty^3$ & $\mathcal R_\infty^{}$ \\
\midrule
$\type_\emptyword(f)$ for $f\in M$ & \img5 & \img6 & \img7 & \img1 & \img2 & \img3 & \img4 \\
\bottomrule
\end{tabular}
\end{table}

For any non-empty word $w\in\Words^*$ define $\type_w(f)$ by $\type_w(f) = \type_\emptyword(\pi_w(f))$. This yields a map $\mathbold\type\colon\mathcal Q_\infty\to\mathcal C^{\Words^*}$ given by $\mathbold\type(f) = (\type_w(f))_{w\in\Words^*}$: $\mathbold\type(f)$ encodes the shape of $f$ at every level. In order to reconstruct $\Tr^\infty_n f$ from $\mathbold\type(f)$, let $\eta$ be the map from $\mathcal C$ to the set of subgraphs of $G_0$ defined in the obvious way, see Table~\ref{table:sub1}. Then
\[ \Tr^\infty_n f = \bigcup_{w\in\Words^n} \psi_w(\eta(\type_w(f))). \]
Hence $\mathbold\type$ is one-to-one.

\begin{table}[htb]
\caption{The mappings $\eta$ and $\nu$.}
\label{table:sub1}
\centering
\begin{tabular}{@{}*{8}{c}@{}}
\toprule
$x$ & \img5 & \img6 & \img7 & \img1 & \img2 & \img3 & \img4 \\
\midrule
$\eta(x)$ & \tree yyynyy & \tree yyyyny & \tree yyyyyn & \tree yyyynn & \tree yyynyn & \tree yyynny & \tree yyynnn \\
$\nu(x)$ & 1 & 2 & 3 & 4 & 5 & 6 & 7 \\
\bottomrule
\end{tabular}
\end{table}

Let $\nu$ be the bijection from $\mathcal C$ to $\{1,\dotsc,7\}$ given in Table~\ref{table:sub1}. Define the functions $\type^\num_{i,n}(f) = \card{\{ w\in\Words^n \,:\, \nu(\type_w(f)) = i \}}$, which count the number of $n$\ndash parts of type $i \in \{1,2,\ldots,7\}$, and set
\[ \mathbold\type^\num_n(f) = \bigl(\type^\num_{1,n}(f),\dotsc,\type^\num_{7,n}(f)\bigr) \]
for $n\ge0$. Of course all these definitions also make sense for finite forests $f\in\mathcal Q_m$ (where $m$ is some nonnegative integer). For $w\in\Words^n$, $n\le m$, define $\type_w(f)$ and $\mathbold\type_n(f)$ in analogy to the definitions above.

Finally, we define the number of connected components $c(x)$, where $x$ is a symbol in $\mathcal C$, in the canonical way as follows:
\[ c(x) = \begin{cases}
		1 & \text{if } x \in \{\img5,\img6,\img7\}, \\
		2 & \text{if } x \in \{\img1,\img2,\img3\}, \\
		3 & \text{if } x = \img4.
	\end{cases} \]
For $f\in\mathcal Q_\infty$, we set $c(f)=c(\type_\emptyword(f))$, which is also the number of components of $\Tr^\infty_n f$ for any $n\ge0$. The following simple lemma relates all our counting functions (cf.~Lemma~5.1 in \cite{teufl2011number}). We write $\mathbold v^t$ to denote the transpose of a vector $\mathbold v$.

\begin{lemma}\label{lemma:constraints}
For any $f\in\mathcal Q_\infty$ and any $n\ge0$,
\begin{align*}
\mathbold\type^\num_n(f) \cdot (1,1,1,1,1,1,1)^t &= 3^n, \\
\mathbold\type^\num_n(f) \cdot (2,2,2,1,1,1,0)^t &= \tfrac32(3^n+1) - c(f), \\
\mathbold\type^\num_n(f) \cdot (1,1,1,-1,-1,-1,-3)^t &= 3 - 2c(f).
\end{align*}
\end{lemma}

\begin{proof}
The first equation is immediate. In order to prove the second, notice that $\mathbold\type^\num_n(f)\cdot(2,2,2,1,1,1,0)^t$ is the number of edges in the spanning forest $\Tr^\infty_n f$ of $G_n$. As this spanning forest has $c(f)$ components, its number of edges is given by $\card{VG_n}-c(f) = \frac32(3^n+1)-c(f)$. The third equation follows from the first and the second by elimination of $3^n$.
\end{proof}

\section{Uniform spanning trees}
\label{section:ust}

We now come to the core part of this paper: the discussion of the structure of uniform spanning trees of $G_n$. Let us write $\Unif\mathcal X$ to denote the uniform distribution on a finite, non-empty set $\mathcal X$. For $i\in\{1,2,3\}$ let $T_n^i$ be a uniformly random element in $\mathcal T_n^i$. If $B\sim\Unif\{1,2,3\}$ is independent of $T_n^i$, then $T_n^B$ is clearly a uniform spanning tree of $G_n$, i.e., $T_n^B\sim\Unif\mathcal T_n$. In the following lemma, we prove the important fact that the trace preserves probabilities:

\begin{lemma}\label{lemma:extension}
Let $i\in\{1,2,3\}$. If $T_n^i$ is uniformly random on $\mathcal T_n^i$, $S_n^i$ is uniformly random on $\mathcal S_n^i$, and $R_n^{}$ is uniformly random on $\mathcal R_n^{}$, then
\begin{align*}
\Prob(\Tr T_{n+1}^i  \in A) &= \Prob(T_n^i  \in A), \\
\Prob(\Tr S_{n+1}^i  \in B) &= \Prob(S_n^i  \in B), \\
\Prob(\Tr R_{n+1}^{} \in C) &= \Prob(R_n^{} \in C)
\end{align*}
for any $A\subseteq\mathcal T_n^i$, $B\subseteq\mathcal S_n^i$, $C\subseteq\mathcal R_n^{}$.
\end{lemma}

\begin{proof}
In order to prove the first identity, we have to show that
\[ \Prob(\Tr T_{n+1}^i = t) = \Prob(T_n^i = t) \]
for any $t\in\mathcal T_n^i$. This is equivalent to
\[ \card{\Tr^{-1} t} = \frac{\tau_{n+1}}{\tau_n} \]
for any $t\in\mathcal T_n^i$. Since
\[ \card{\mathcal T_1^k} = \tau_1 = 18, \qquad
	\card{\mathcal S_1^k} = \sigma_1 = 30, \qquad
	\card{\mathcal R_1^{}} = \rho_1 = 50 \]
for $k\in\{1,2,3\}$, Lemma~\ref{lemma:constraints} implies that
\[ \card{\Tr^{-1} t}
	= 18^{\type^\num_{n,1}(t)+\type^\num_{n,2}(t)+\type^\num_{n,3}(t)} \cdot
	  30^{\type^\num_{n,4}(t)+\type^\num_{n,5}(t)+\type^\num_{n,6}(t)} \cdot
	  50^{\type^\num_{n,7}(t)}
	= 18 \cdot 540^{(3^n-1)/2}. \]
Using Lemma~\ref{lemma:count}, it is easy to see that
\[ \frac{\tau_{n+1}}{\tau_n} = 18 \cdot 540^{(3^n-1)/2}. \]
The same argument applies to the second and third identity, too.
\end{proof}

In light of Lemma~\ref{lemma:extension} and Kolmogorov's Extension Theorem there is a probability measure $P_{T^i}$ on $\mathcal T_\infty^i$ such that $P_{T^i}(\{t \in \mathcal T_{\infty}^i\,:\,\Tr^\infty_n t\in\cdot\}) = \Prob(T_n^i\in\cdot)$. Let $P_{S^i}$ and $P_{R^{}}$ be the analogous measures on $\mathcal S_\infty^i$ and $\mathcal R_\infty^{}$, respectively. Set
\begin{align*}
\Omega &= \{1,2,3\} \times \mathcal T_\infty^1 \times \mathcal T_\infty^2 \times \mathcal T_\infty^3 \times
	\mathcal S_\infty^1 \times \mathcal S_\infty^2 \times \mathcal S_\infty^3 \times \mathcal R_\infty^{}, \\
P &= \Unif\{1,2,3\} \times P_{T^1} \times P_{T^2} \times P_{T^3}
	\times P_{S^1} \times P_{S^2} \times P_{S^3} \times P_{R^{}}.
\end{align*}
Let $B, T_\infty^i$, $S_\infty^i$, $R_\infty^{}$ be the canonical projections from $\Omega$ to $\{1,2,3\}$, $\mathcal T_\infty^i$, $\mathcal S_\infty^i$, $\mathcal R_\infty^{}$, respectively. Set $T_\infty = T_\infty^B$ and, for $n\ge0$, $T_n = \Tr^\infty_n T_\infty$. Then $T_n$ is a uniform spanning tree on $G_n$ and $T_n=\Tr^m_n T_m=\Tr^\infty_n T_{\infty}$ for $m\ge n\ge 0$. Analogous statements hold for $S_n^i = \Tr^\infty_n S_\infty^i$ and $R_n = \Tr^\infty_n R_\infty$.

In the following we write $\Prob$ instead of $P$ and always use $\Omega$ equipped with $\Prob$ as probability space, whenever the random elements $T_n^{}$, $T_n^i$, $S_n^i$, etc. are considered.

Let $[\img5,\img5,\img2]$ (suppressing the dependence on $n$) be a shorthand for the set
\[ \{ \psi_1(f_1) \cup \psi_2(f_2) \cup \psi_3(f_3) \,:\,
	f_1\in\mathcal T_{n-1}^1, f_2\in\mathcal T_{n-1}^1, f_3\in\mathcal S_{n-1}^2\}, \]
and analogously for other combinations. Using Lemma~\ref{lemma:count} it is easy to see that
\begin{equation}\label{eq:probs}
\begin{aligned}
\Prob(T_n^3\in[\img5,\img1,\img5]) &= \Prob(T_n^3\in[\img5,\img1,\img6]) = \dotsb
	= \frac{\tau_{n-1}^2\sigma_{n-1}}{\tau_n} = \frac1{18}, \\
\Prob(S_n^3\in[\img5,\img3,\img1]) &= \Prob(S_n^3\in[\img6,\img3,\img1]) = \dotsb
	= \frac{\tau_{n-1}\sigma_{n-1}^2}{\sigma_n} = \frac1{30}, \\
\Prob(S_n^3\in[\img5,\img5,\img4]) &= \Prob(S_n^3\in[\img5,\img6,\img4]) = \dotsb
	= \frac{\tau_{n-1}^2\rho_{n-1}}{\sigma_n} = \frac1{30}, \\
\Prob(R_n^{}\in[\img3,\img2,\img3]) &= \Prob(R_n^{}\in[\img1,\img3,\img3]) = \dotsb
	= \frac{\sigma_{n-1}^3}{\rho_n} = \frac1{50}, \\
\Prob(R_n^{}\in[\img5,\img2,\img4]) &= \Prob(R_n^{}\in[\img6,\img2,\img4]) = \dotsb
	= \frac{\tau_{n-1}\sigma_{n-1}\rho_{n-1}}{\rho_n} = \frac1{50},
\end{aligned}
\end{equation}
where dots indicate combinations in the same ``group'' (group sizes are $18$, $21$, $9$, $14$, and $36$, see Figure~\ref{figure:rectree}). Of course, analogous results also hold for $T_n^1$, $T_n^2$, $S_n^1$, $S_n^2$. Furthermore, note that
\[ \Prob(\pi_2(T_n^3) \in \cdot \mid T_n^3\in[\img5,\img1,\img5]) = \Unif\mathcal S_{n-1}^1, \]
and analogously for other combinations and restrictions. Using this fact we obtain the following result, which relates uniform spanning trees on Sierpi\'nski graphs to a multi-type Galton-Watson process:

\begin{proposition}\label{proposition:tree1}
Let $\mathcal U_\infty$ be one of $\mathcal T_\infty^{}$, $\mathcal T^i_\infty$, $\mathcal S^i_\infty$, $\mathcal R^{}_\infty$, and let $U_\infty$ be the corresponding random object.
\begin{enumerate}[\normalfont(1)]
\item The random tree
	\[ \mathbold\type(U_\infty) = (\type_w(U_\infty))_{w\in\Words^*} \]
	is a labelled multi-type Galton-Watson tree with labels in $\Words^*$ and types in $\mathcal C$.
	The type distribution of the root depends on the specific choice for $\mathcal U_\infty$ and
	is given by $\Unif\{\type_\emptyword(f) \,:\, f\in\mathcal U_\infty\}$.
	The set of individuals is deterministic and equal to $\Words^*$.
	Each individual has three children with suffixes $1,2,3$.
	For $x\in\mathcal C$ set
	\[ \mathcal D(x)
		= \bigl\{ (\type_1(f),\type_2(f),\type_3(f)) \,:\,
			f\in\mathcal Q_1, \type_\emptyword(f)=x \bigr\}
		\subseteq \mathcal C^3. \]
	Then, by Equation~\eqref{eq:probs}, the offspring distribution of an individual of type $x$
	is given by $\Unif\mathcal D(x)$, that is,
	\begin{equation}\label{eq:gwdist}
	\Prob((\type_{w1}(U_\infty),\type_{w2}(U_\infty),\type_{w3}(U_\infty)) \in \cdot \mid \type_w(U_\infty) = x)
		= \Unif\mathcal D(x).
	\end{equation}
\item $(\mathbold\type^\num_n(U_\infty))_{n\ge0}$ is a multi-type Galton-Watson process with seven types,
	which is non-singular, positively regular, and supercritical.
	The type distribution of the root is given by the uniform distribution
	$\Unif\{\nu(\type_\emptyword(f)) \,:\, f\in\mathcal U_\infty\}$.
	The offspring generating function is easily computed by means of Equation~\eqref{eq:gwdist}:
	using the abbreviation $s=\frac13(z_1+z_2+z_3)$, we have
	\begin{align*}
	\mathbold f(\mathbold z) = \Bigl(
		&\tfrac12 s^2(z_5+z_6), \; \tfrac12 s^2(z_4+z_6), \; \tfrac12 s^2(z_4+z_5), \vphantom{\Big(}\\
		&\tfrac1{10} s\bigl(3z_4^2 + 2z_4(z_5+z_6) + 3sz_7\bigr), \vphantom{\Big(}\\
		&\tfrac1{10} s\bigl(3z_5^2 + 2z_5(z_4+z_6) + 3sz_7\bigr), \vphantom{\Big(}\\
		&\tfrac1{10} s\bigl(3z_6^2 + 2z_6(z_4+z_5) + 3sz_7\bigr), \vphantom{\Big(}\\
		&\tfrac1{25} \bigl(z_4^2z_5 + z_4z_5^2 + z_4^2z_6 + z_4z_6^2 + z_5^2z_6 + z_5z_6^2 \vphantom{\Big(}\\
		&\qquad\qquad + z_4z_5z_6 + 6s(z_4+z_5+z_6)z_7\bigr) \Bigr).
	\end{align*}
	Its mean matrix $\mathbold M$ is given by
	\[ \mathbold M = \frac1{150}\begin{pmatrix}
			100 & 100 & 100 &   0 &  75 &  75 &   0 \\
			100 & 100 & 100 &  75 &   0 &  75 &   0 \\
			100 & 100 & 100 &  75 &  75 &   0 &   0 \\
			 65 &  65 &  65 & 150 &  30 &  30 &  45 \\
			 65 &  65 &  65 &  30 & 150 &  30 &  45 \\
			 65 &  65 &  65 &  30 &  30 & 150 &  45 \\
			 36 &  36 &  36 &  78 &  78 &  78 & 108 \\
		\end{pmatrix}. \]
	The dominating eigenvalue of $\mathbold M$ is equal to $3$.
	The corresponding right and left eigenvectors are
	\[ \mathbold v_R=(1,1,1,1,1,1,1)^t \qquad\text{and}\qquad
		\mathbold v_L=\tfrac1{288}(53,53,53,38,38,38,15), \]
	respectively. $\mathbold v_R$ and $\mathbold v_L$ are normalized
	so that $\mathbold v_L \cdot \mathbold v_R = 1$ and $\norm{\mathbold v_L}_1=1$.
\item Since $\norm{\mathbold\type^\num_n(U_\infty)}_1 = \mathbold\type^\num_n(U_\infty) \cdot \mathbold v_R = 3^n$,
	it follows that
	\[ \lim_{n\to\infty} 3^{-n} \mathbold\type^\num_n(U_\infty) = \mathbold v_L \]
	almost surely, see Theorem~1.8.3 in \cite{mode1971multitype}.
\end{enumerate}
\end{proposition}

\begin{remark}\label{remark:coll1}
In order to sample a uniform spanning tree on $G_n$, we may simulate the $n$\ndash th generation of a labelled multi-type Galton-Watson tree as described above. In this process, we have to choose one of \img5, \img6, \img7 with equal probability as the type for the ancestor $\emptyword$ of the tree. It is possible to postpone this choice from the beginning to the $n$\ndash th generation. To this end, collapse the three types \img5, \img6, \img7 into one type \img0. This yields again a labelled multi-type Galton-Watson tree, but now with five types $\{\img0,\img1,\img2,\img3,\img4\}$. In order to obtain a uniform spanning tree on $G_n$, consider the $n$\ndash th generation of this simplified labelled multi-type Galton-Watson tree and replace each occurrence of \img0 independently by one of \img5, \img6, \img7 with equal probability. This modified $n$\ndash th generation describes a spanning tree on $G_n$, whose distribution is uniform. Figure~\ref{figure:example-spanning} shows an example of a randomly generated spanning tree on $G_5$.
\end{remark}

\begin{figure}[htb]
\centering
\includegraphics[scale=0.4]{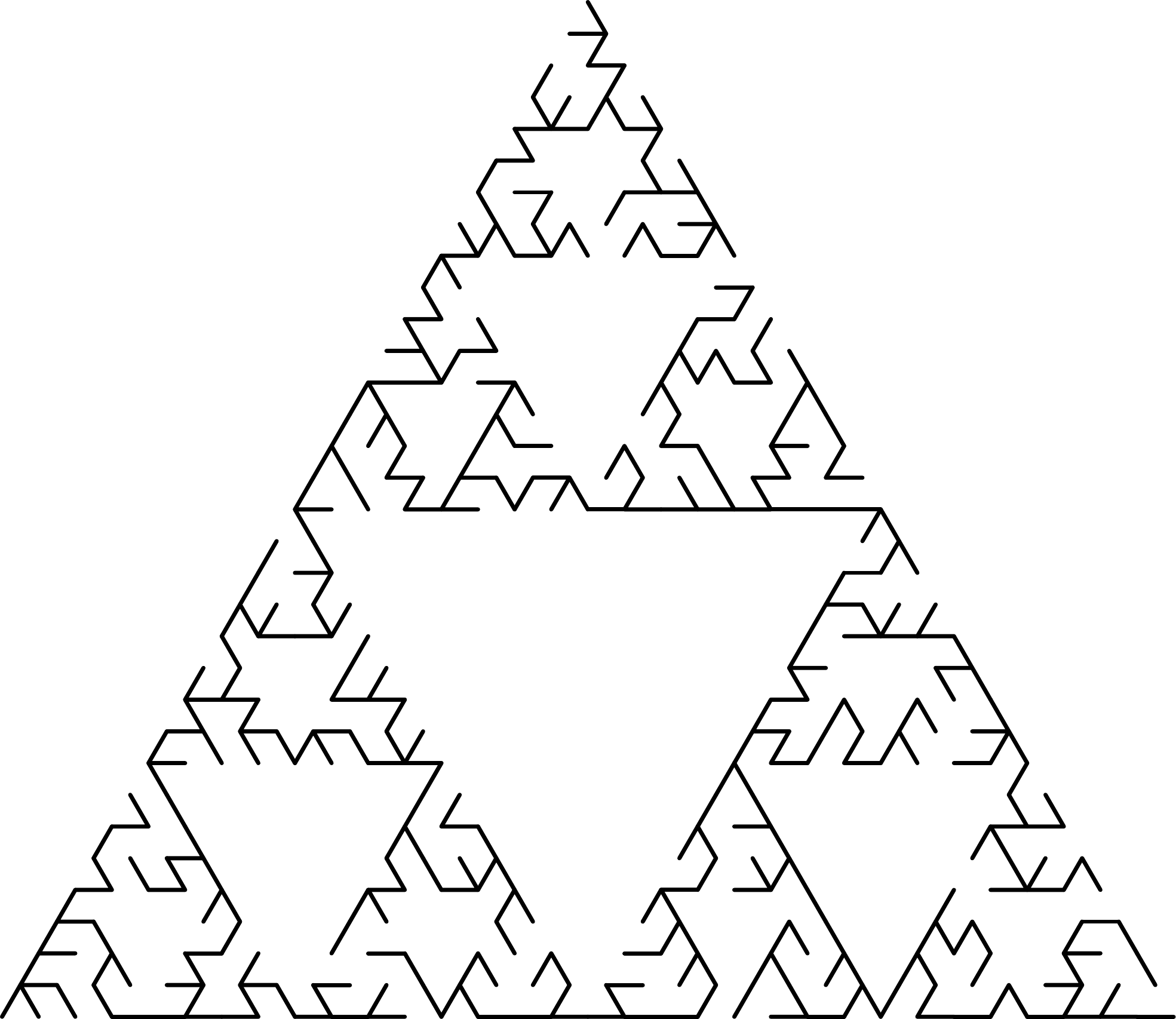}
\caption{A randomly generated spanning tree on $G_5$.}
\label{figure:example-spanning}
\end{figure}

\begin{remark}\label{remark:func1}
Suppose $\lambda$ is a parameter of spanning trees in $G_n$, and we are interested in the behaviour of $\lambda(T_n)$ as $n\to\infty$. When $\lambda(T_n)$ is a functional of $\mathbold\type^\num_n(T_\infty)$, say $\lambda(T_n) = h(\mathbold\type^\num_n(T_\infty))$ for some linear function $h$, then $3^{-n}\lambda(T_n) \to h(\mathbold v_L)$ almost surely. Of course, this generalizes to positive homogeneous functions $h$. As a simple example consider the number of $n$\ndash parts with $i$ connected components. For $f\in\mathcal Q_\infty$, $i\in\{1,2,3\}$, and $n\ge0$, let us denote this quantity by $c^\num_{i,n}(f) = \card{\{w\in\Words^n \,:\, c(\type_w(f))=i\}}$. Then
\[ \mathbold c^\num_n(f) = \bigl(c^\num_{1,n}(f),c^\num_{2,n}(f),c^\num_{3,n}(f)\bigr)
	= \mathbold\type^\num_n(f) \cdot \begin{pmatrix}
		1 & 1 & 1 & 0 & 0 & 0 & 0 \\
		0 & 0 & 0 & 1 & 1 & 1 & 0 \\
		0 & 0 & 0 & 0 & 0 & 0 & 1
	\end{pmatrix}^t \]
and therefore $3^{-n} \mathbold c^\num_n(U_\infty) \to \frac1{96}(53,38,5)$ almost surely as $n\to\infty$ if $U_\infty$ is one of $T_\infty^{}$, $T_\infty^i$, $S_\infty^i$, $R_\infty^{}$ for $i\in\{1,2,3\}$. Note that, due to symmetry, $(\mathbold c^\num_n(U_\infty))_{n\ge0}$ is a multi-type Galton-Watson process in its own right.

A straightforward calculation shows that the variance of $\mathbold\type^\num_n$ and $\mathbold c^\num_n$ is of order $3^n$, so that Chebyshev's inequality yields
\begin{equation}\label{eq:cheby}
\Prob(\lVert\mathbold\type^\num_n(U_{\infty}) - 3^n \mathbold v_L \rVert_1 \geq \alpha^n) \ll 3^{n} \alpha^{-2n}
\end{equation}
for any $\alpha \in (\sqrt{3},3)$, and an analogous inequality for $\mathbold c^\num_n$.
\end{remark}

In the following we study two quantities of a random spanning forest of $G_n$, which need more work as the previous remark does not apply directly: The first quantity are the component sizes in $S_n^1,S_n^2,S_n^3,R_n^{}$. In this case it turns out that components can be described using an augmented labelled multi-type Galton-Watson tree. Secondly, we study the degree distribution in $T_n$. Here the recursive description of uniform spanning trees in $G_n$ (see Figure~\ref{figure:rectree} and Proposition~\ref{proposition:tree1}) and the rapid decay of tail probabilities given by \eqref{eq:cheby} is used.

\subsection{Component sizes}\label{subsection:component}
Spanning trees only have one component, but for random spanning forests $S_n^i$ or $R_n$, the sizes (number of vertices or edges) of the components are interesting random variables. Let us briefly explain how their limiting distribution can be obtained.

First, we need some notation. For a non-empty subset $B$ of $VG_0$, let $f$ be an element of $\mathcal Q_\infty$, and assume that $B$ is the vertex set of the union of some connected components of $\Tr^\infty_0 f$. Write $C_n(f,B)$ to denote the union of those components of $\Tr^\infty_n f$ having non-empty intersection with $B$. For example, if $f \in S_{\infty}^1$, then $C_n(f,\{u_1\})$ is the component of $\Tr^\infty_n f$ that contains $u_1$, $C_n(f,\{u_2,u_3\})$ is the component that contains $u_2$ and $u_3$, and $C_n(f,\{u_1,u_2,u_3\})$ is the entire spanning forest $\Tr^\infty_n f$.

We are interested in the size of $C_n(f,B)$, which unfortunately is not a linear functional of $\mathbold\type^\num_n(f)$. However, it is possible to define a subtree of the Galton-Watson-tree $\mathbold\type(f)$ that encodes $f$ and to add extra information to the types $\mathcal C$ that records the evolution of the components in $C_n(f,B)$. If $f$ is randomly chosen, the resulting subtree with augmented types describes another labelled multi-type Galton-Watson tree, as will be shown in the following. For $n\ge0$, let
\[ \hat W_n(f,B) = \bigl\{ w\in\Words^n \,:\, C_n(f,B) \cap \psi_w(G_0) \ne \emptyset \bigr\} \]
be the set of those words $w\in\Words^n$ for which the corresponding $n$\ndash part $\psi_w(G_0)$ of the Sierpi\'nski gasket has non-empty intersection with $C_n(f,B)$. Their union
\[ \hat W(f,B) = \bigcup_{n\ge0} \hat W_n(f,B) \]
induces a subtree of $\Words^*$, and each word in $\hat W(f,B)$ has one, two or three children. For $w\in\hat W(f,B)$, write $\hat\kappa_w(f,B)$ to denote the vertex set $V(\pi_w(\psi_w(G_0)\cap C_n(f,B)))$ (in words: the vertices of the $n$\ndash part $\psi_w(G_0)$ that are in common components with vertices of $B$, projected back to $G_0$). To each $w\in\hat W_n(f,B)$, we assign one of the following nineteen types
\[ \mathcal{\hat C} = \{
	\comp5y--, \comp6y--, \comp7y--,
	\comp1yy-, \comp1ny-, \comp1yn-,
	\comp2yy-, \comp2ny-, \comp2yn-,
	\comp3yy-, \comp3ny-, \comp3yn-,
	\comp4yyy,
	\comp4nyy, \comp4yny, \comp4yyn,
	\comp4ynn, \comp4nyn, \comp4nny \} \]
encoding two pieces of information: $\type_w(f)$ (structure of the restriction of $f$ to the respective $n$\ndash part) and $\hat\kappa_w(f,B)$ (black parts indicate which of the corner vertices are in common components with elements of $B$). We denote this assignment by $\hat\type_w(f,B)$, see Table~\ref{table:comptypes} for a precise definition of $\hat\type_w(f,B)$ in terms of $\type_w(f)$ and $\hat\kappa_w(f,B)$.

\begin{table}[htb]
\caption{The type $\hat\type_w(f,B)$, given $\type_w(f)$ and $\hat\kappa_w(f,B)$.}
\label{table:comptypes}
\centering
\begin{tabular}{@{}*{8}{c}@{}}
\toprule
                    & \multicolumn{7}{c}{$\type_w(f)$} \\
                      \cmidrule(l){2-8}
$\hat\kappa_w(f,B)$ & \img5     & \img6     & \img7     & \img1     & \img2     & \img3     & \img4 \\
\cmidrule(r){1-1}     \cmidrule(l){2-8}
$\{u_1,u_2,u_3\}$   & \comp5y-- & \comp6y-- & \comp7y-- & \comp1yy- & \comp2yy- & \comp3yy- & \comp4yyy \\
$\{u_2,u_3\}$       &           &           &           & \comp1ny- &           &           & \comp4nyy \\
$\{u_1,u_3\}$       &           &           &           &           & \comp2ny- &           & \comp4yny \\
$\{u_1,u_2\}$       &           &           &           &           &           & \comp3ny- & \comp4yyn \\
$\{u_1\}$           &           &           &           & \comp1yn- &           &           & \comp4ynn \\
$\{u_2\}$           &           &           &           &           & \comp2yn- &           & \comp4nyn \\
$\{u_3\}$           &           &           &           &           &           & \comp3yn- & \comp4nny \\
\bottomrule
\end{tabular}
\end{table}

Finally set $\mathbold{\hat\type}(f,B) = (\hat\type_w(f,B))_{w\in\hat W(f,B)}$. It is easy to see that it is possible to reconstruct the graph $C_n(f,B)$ from $\mathbold{\hat\type}(f,B)$: formally,
\[ C_n(f,B) = \bigcup_{w\in\hat W_n(f,B)} \psi_w(\hat\eta(\hat\type_w(f,B))), \]
where $\hat\eta$ is given in Table~\ref{table:subcomp}.

\begin{table}[htb]
\caption{The mappings $\hat\eta$ and $\hat c$.}
\label{table:subcomp}
\centering
\begin{tabular}{@{}*{13}{c}@{}}
\toprule
$x$ & \comp5y-- & \comp6y-- & \comp7y-- %
    & \comp1yy- & \comp1ny- & \comp1yn- %
    & \comp2yy- & \comp2ny- & \comp2yn- %
    & \comp3yy- & \comp3ny- & \comp3yn- \\
\midrule
$\hat\eta(x)$ & \tree yyynyy & \tree yyyyny & \tree yyyyyn %
              & \tree yyyynn & \tree nyyynn & \tree ynnnnn %
              & \tree yyynyn & \tree ynynyn & \tree nynnnn %
              & \tree yyynny & \tree yynnny & \tree nnynnn \\
$\hat c(x)$ & 1 & 1 & 1 & 2 & 4 & 5 & 2 & 4 & 5 & 2 & 4 & 5 \\
\midrule
$x$ & \comp4yyy %
    & \comp4nyy & \comp4yny & \comp4yyn %
    & \comp4ynn & \comp4nyn & \comp4nny \\
\midrule
$\hat\eta(x)$ & \tree yyynnn %
              & \tree nyynnn & \tree ynynnn & \tree yynnnn %
              & \tree ynnnnn & \tree nynnnn & \tree nnynnn \\
$\hat c(x)$ & 3 & 6 & 6 & 6 & 7 & 7 & 7 \\
\bottomrule
\end{tabular}
\end{table}

Now let us define $\hat c(x)$ for $x\in\mathcal{\hat C}$ as in Table~\ref{table:subcomp}. For $i\in\{1,\dotsc,7\}$ and $n\ge0$ set $\hat c^\num_{i,n}(f,B) = \card{\{w\in\hat W_n(f,B) \,:\, \hat c(\hat\type_w(f,B)) = i\}}$ and
\[ \mathbold{\hat c}^\num_n(f,B) = \bigl( \hat c^\num_{1,n}(f,B), \dotsc, \hat c^\num_{7,n}(f,B) \bigr). \]
The vector $\mathbold{\hat c}^\num_n(f,B)$ counts the number of words in $\hat W(f,B)$ of given type up to symmetry. Note that the number of edges in $C_n(f,B)$ can be determined from $\mathbold{\hat c}^\num_n(f,B)$: it is given by
\[ \card{EC_n(f,B)} = \mathbold{\hat c}^\num_n(f,B) \cdot (2,1,0,1,0,0,0)^t, \]
and the number of vertices in $C_n(f,B)$ satisfies
\[ 1 \le \card{VC_n(f,B)} - \card{EC_n(f,B)} \le 3, \]
the precise value of the difference depending on the type of $f$ and the set $B$. Now let $U_\infty$ be one of $T_\infty^{}, T_\infty^i, S_\infty^i, R_\infty^{}$ ($i\in\{1,2,3\}$), and choose $B\subseteq VG_0$ so that $B$ is the vertex set of the union of some components of $\Tr^\infty_0 U_\infty$. Then $\mathbold{\hat\type}(U_\infty,B)$ is a labelled multi-type Galton-Watson tree with types in $\mathcal{\hat C}$, and $(\mathbold{\hat c}^\num_n(U_\infty,B))_{n\ge0}$ is a multi-type Galton-Watson process with seven types. The offspring generating function $\mathbold{\hat f}(\mathbold z)$ of the process is given by
\begin{align*}
\mathbold{\hat f}(\mathbold z) = \Bigl(
	& z_1^2 z_2, \; \tfrac7{10} z_1 z_2^2 + \tfrac3{10} z_1^2 z_3, \;
		\tfrac7{25} z_2^3 + \tfrac{18}{25} z_1 z_2 z_3, \\
	& \tfrac2{10} z_1 z_4 z_5 + \tfrac4{10} z_1 z_2 z_4 + \tfrac1{10} z_4^2 + \tfrac3{10} z_1^2 z_6, \\
	& \tfrac2{10} z_4 z_5 + \tfrac4{10} z_5 + \tfrac1{10} z_1 z_5^2 + \tfrac3{10} z_7, \\
	& \tfrac3{25} z_2^2 z_4 + \tfrac2{25} z_2 z_4 z_5 + \tfrac1{25} z_4 z_5^2 + \tfrac1{25} z_5^2
		+ \tfrac6{25} z_1 z_2 z_6 + \tfrac3{25} z_5 z_7 \\
	& \qquad\qquad + \tfrac3{25} z_1 z_3 z_4 + \tfrac3{25} z_4 z_6 + \tfrac3{25} z_1 z_5 z_6, \\
	& \tfrac6{25} z_5 + \tfrac1{25} z_2 z_4^2 + \tfrac2{25} z_4 z_5 + \tfrac1{25} z_4^2 z_5
		+ \tfrac6{25} z_7 + \tfrac3{25} z_1 z_4 z_6 + \tfrac3{25} z_1 z_5 z_7 + \tfrac3{25} z_4 z_7 \Bigr),
\end{align*}
and the mean matrix is given by
\[ \mathbold{\hat M} = \frac1{50} \cdot \begin{pmatrix}
		100 & 50 &  0 &  0 &  0 &  0 &  0 \\
		 65 & 70 & 15 &  0 &  0 &  0 &  0 \\
		 36 & 78 & 36 &  0 &  0 &  0 &  0 \\
		 60 & 20 &  0 & 40 & 10 & 15 &  0 \\
		  5 &  0 &  0 & 10 & 40 &  0 & 15 \\
		 24 & 28 &  6 & 24 & 24 & 24 &  6 \\
		 12 &  2 &  0 & 24 & 24 &  6 & 24 \\
		\end{pmatrix}. \]
Hence the multi-type Galton-Watson process is non-singular, but not positively regular, as the mean matrix is reducible. The dominating eigenvalue is $3$, which belongs to the $3\times3$ block of $\mathbold{\hat M}$ in the upper left corner. It has multiplicity $1$, and the corresponding right and left eigenvectors are
\[ \mathbold{\hat v}_R = \bigl(1,1,1,\tfrac56,\tfrac16,\tfrac23,\tfrac13\bigr)^t, \qquad
	\mathbold{\hat v}_L = \tfrac1{96} \cdot (53,38,5,0,0,0,0), \]
respectively. $\mathbold{\hat v}_R$ and $\mathbold{\hat v}_L$ are normalized so that $\mathbold{\hat v}_L \cdot \mathbold{\hat v}_R = 1$ and $\norm{\mathbold{\hat v}_L}_1 = 1$. Intuitively, the fact that only the first three entries in $\mathbold{\hat v}_L$ are nonzero (and that $\mathbold{\hat M}$ is dominated by the upper left $3 \times 3$\ndash block) can be explained by the fact that $n$\ndash parts of types such as \comp1ny-, \comp4nny, etc. (some vertices belong to $C_n(f,B)$, others do not) can only occur at the ``borders'' between the components of the forest $\Tr_n^{\infty} f$, which only make up a very small part of the entire graph $G_n$.

For every choice of the boundary vertices $B$, there is a non-negative random variable $\hat\theta(U_\infty,B)$ such that
\[ 3^{-n} \mathbold{\hat c}^\num_n(U_\infty,B) \to \mathbold{\hat v}_L \hat\theta(U_\infty,B) \]
holds almost surely as $n\to\infty$, see Theorem~2.4.1 in \cite{mode1971multitype}. By symmetry, there are eight different limit distributions, one for each of the following groups:
\begin{gather*}
\bigl\{ \hat\theta(T^{}_\infty,\{u_1,u_2,u_3\}) \bigr\}, \\
\bigl\{ \hat\theta(T^1_\infty,\{u_1,u_2,u_3\}),\hat\theta(T^2_\infty,\{u_1,u_2,u_3\}),
	\hat\theta(T^3_\infty,\{u_1,u_2,u_3\}) \bigr\}, \\
\bigl\{ \hat\theta(S^1_\infty,\{u_1,u_2,u_3\}),\hat\theta(S^2_\infty,\{u_1,u_2,u_3\}),
	\hat\theta(S^3_\infty,\{u_1,u_2,u_3\}) \bigr\}, \\
\bigl\{ \hat\theta(R^{}_\infty,\{u_1,u_2,u_3\}) \bigr\}, \\
\bigl\{ \hat\theta(S^1_\infty,\{u_2,u_3\}),\hat\theta(S^2_\infty,\{u_1,u_3\}),
	\hat\theta(S^3_\infty,\{u_1,u_2\}) \bigr\}, \\
\bigl\{ \hat\theta(S^1_\infty,\{u_1\}),\hat\theta(S^2_\infty,\{u_2\}),\hat\theta(S^3_\infty,\{u_3\}) \bigr\}, \\
\bigl\{ \hat\theta(R^{}_\infty,\{u_2,u_3\}),\hat\theta(R^{}_\infty,\{u_1,u_3\}),
	\hat\theta(R^{}_\infty,\{u_1,u_2\}) \bigr\}, \\
\bigl\{ \hat\theta(R^{}_\infty,\{u_1\}),\hat\theta(R^{}_\infty,\{u_2\}),\hat\theta(R^{}_\infty,\{u_3\}) \bigr\}.
\end{gather*}
Let us write $\hat\theta_i$ ($i\in\{0,\dotsc,7\}$) for a random variable having the same distribution as a random variable of the respective group above. Of course, $\hat\theta_0,\dotsc,\hat\theta_3$ (the cases when $B = \{u_1,u_2,u_3\}$) are almost surely constant, i.e.
\[ \hat\theta_0=\hat\theta_1=\hat\theta_2=\hat\theta_3=1 \]
almost surely. The remaining variables $\hat\theta_4,\dotsc,\hat\theta_7$ have continuous densities, and
\[ \Expect(\hat\theta_i) = \hat v_{i,R} \]
for $i\in\{4,\dotsc,7\}$, where $\hat v_{i,R}$ is the $i$\ndash coordinate of $\mathbold{\hat v}_R$. Note also that $1 - \hat\theta_4$ and $\hat\theta_5$ have the same distribution, and the same holds for $1 - \hat\theta_6$ and $\hat\theta_7$. The limits of the renormalised component sizes can be expressed in terms of these random variables. To be precise,
\[ \lim_{n\to\infty} 3^{-n}\card{VC_n(U_\infty,B)}
	= \lim_{n\to\infty} 3^{-n}\card{EC_n(U_\infty,B)} = \tfrac32 \hat\theta(U_\infty,B) \]
almost surely. In particular, the component $C_n(S^1_\infty,\{u_2,u_3\})$ is on average approximately five times larger than the complementary component $C_n(S^1_\infty,\{u_1\})$, since $\hat v_{4,R} = \frac56 = 5\hat v_{5,R}$.

\subsection{Degree distribution}\label{subsection:degree}
The distribution of the vertex degrees in a random spanning tree of the Sierpi\'nski graph $G_n$ was studied at length by Chang and Chen in their recent paper \cite{chang2010structure}. In particular, they determined the precise probability distribution of the degree of a given vertex, and determined the average proportion of the number of vertices of given degree as $n \to \infty$. Here we provide a somewhat different approach to this problem with the advantage that it also allows us to prove almost sure convergence of this proportion to a limit.

The number of vertices with a certain degree in a random spanning tree $T_n$ is again not a simple functional of the types. In fact, the degree distribution of a vertex $v \in VG_n$ depends not only on $n$, but also on the level of the vertex $v$ itself: by the level of a vertex, we mean the smallest $k$ such that $v \in VG_k$. Let us first consider the degree distribution of the corner vertices. By symmetry, it is obviously sufficient to consider one of them. Let $\mathbold d_n(h)$ be the vector of the probabilities that the degree $\deg_{U_n} u_1$ of the lower-left corner vertex $u_1$ in a random spanning forest $U_n$ is equal to $h \in \{0,1,2\}$ for $U_n \in \{T_n^1,T_n^2,T_n^3,S_n^1,S_n^2,S_n^3,R_n^{}\}$. The entries are denoted by $d_n(\img5,h),d_n(\img6,h)$, etc. Thus $d_n(\img5,h) = \Prob(\deg_{T_n^1} u_1 = h)$, and the other entries are defined analogously. Then it is obvious that
\begin{align*}
\mathbold d_0(0) &= (0,0,0,1,0,0,1)^t, \\
\mathbold d_0(1) &= (0,1,1,0,1,1,0)^t, \\
\mathbold d_0(2) &= (1,0,0,0,0,0,0)^t.
\end{align*}
Moreover, we deduce from the recursive structure (Figure~\ref{figure:rectree}) that $\mathbold d_n(h) = \mathbold D \mathbold d_{n-1}(h)$, where $\mathbold D$ is the matrix
\[ \mathbold D = \frac1{150} \begin{pmatrix}
		50 & 50 & 50 &  0 &  0 &  0 &  0 \\
		25 & 25 & 25 &  0 &  0 & 75 &  0 \\
		25 & 25 & 25 &  0 & 75 &  0 &  0 \\
		 5 &  5 &  5 & 60 & 15 & 15 & 45 \\
		30 & 30 & 30 &  0 & 45 & 15 &  0 \\
		30 & 30 & 30 &  0 & 15 & 45 &  0 \\
		12 & 12 & 12 & 36 & 21 & 21 & 36
	\end{pmatrix}. \]
This matrix has eigenvalues $1,\frac35,\frac15,\frac1{15},\frac1{25},0,0$, and we easily find that
\begin{equation}\label{eq:degprob}
\begin{aligned}
\mathbold d_n(0) &= \tfrac{11}{28} \cdot \bigl(\tfrac35\bigr)^n \cdot (0,0,0,3,0,0,2)^t
			- \tfrac1{28} \cdot \bigl(\tfrac1{25}\bigr)^n \cdot (0,0,0,5,0,0,-6)^t, \\[3pt]
\mathbold d_n(1) &= \tfrac{11}{14} \cdot (1,1,1,1,1,1,1)^t
			- \tfrac2{7} \cdot \bigl(\tfrac35\bigr)^n \cdot (0,0,0,3,0,0,2)^t \\
		 &\quad + \tfrac1{14} \cdot \bigl(\tfrac1{15}\bigr)^n \cdot (-25,10,10,-4,3,3,3)^t
			+ \tfrac1{14} \cdot \bigl(\tfrac1{25}\bigr)^n \cdot (0,0,0,5,0,0,-6)^t, \\[3pt]
\mathbold d_n(2) &= \tfrac3{14} \cdot (1,1,1,1,1,1,1)^t
			- \tfrac3{28} \cdot \bigl(\tfrac35\bigr)^n \cdot (0,0,0,3,0,0,2)^t \\
		 &\quad - \tfrac1{14} \cdot \bigl(\tfrac1{15}\bigr)^n \cdot (-25,10,10,-4,3,3,3)^t
		 	- \tfrac1{28} \cdot \bigl(\tfrac1{25}\bigr)^n \cdot (0,0,0,5,0,0,-6)^t
\end{aligned}
\end{equation}
for $n \geq 1$. In particular, we see that the degree of a corner vertex is $1$ in a random spanning tree of $G_n$ with probability tending to $\frac{11}{14}$, and the degree is $2$ with probability tending to $\frac{3}{14}$.

If now $v \in VG_n$ is a vertex of level $k>0$, then there is a unique copy $H$ of $G_{n-k+1}$ in $G_n$ such that $v$ is the midpoint of one of its sides. The degree distribution of $v$ in a random spanning tree $T_n$ now only depends on $k$ and the type of the restriction of $T_n$ to $H$. For example, if $v$ is the midpoint of the horizontal side of $H$, and the restriction is of type \img5, then the probability that $v$ has degree $h$ in $T_n$ is
\begin{multline*}
\frac16 \sum_{\ell=0}^h \bigl(d_{n-k}(\img5,\ell) + d_{n-k}(\img6,\ell) + d_{n-k}(\img7,\ell) \bigl)
		d_{n-k}(\img3,h-\ell) \\
+ \frac1{18} \sum_{\ell=0}^h \bigl(d_{n-k}(\img5,\ell) + d_{n-k}(\img6,\ell) + d_{n-k}(\img7,\ell)\bigr) \\
	\times \bigl(d_{n-k}(\img5,h-\ell) + d_{n-k}(\img6,h-\ell) + d_{n-k}(\img7,h-\ell)\bigr),
\end{multline*}
where we set $d_n(\cdot,h)=0$ if $h>2$. It follows immediately that for any fixed $k>0$ (or even more generally, if $n-k \to \infty$), the probabilities of the possible degrees $1,2,3,4$ of a level-$k$ vertex converge to
\[ 0, \quad \frac{121}{196}, \quad \frac{33}{98}, \quad \frac{9}{196}, \]
respectively. Intuitively, this means that leaves typically only occur at high levels.

Let now $W_n(h)$ denote the number of vertices of degree $h$ in a random spanning tree $T_n$. We prove that $3^{-n}W_n(h) \to w(h)$ almost surely, where
\[ w(1) = \frac{10957}{26976}, \quad w(2) = \frac{6626035}{9090912}, \quad
	w(3) = \frac{2943139}{9090912}, \quad w(4) = \frac{124895}{3030304}. \]
Fix some $\alpha \in (\sqrt{3},3)$. For any $r \geq 0$, the number of copies of $G_{r+1}$ occurring in $G_n$ is $3^{n-r-1}$. By \eqref{eq:cheby}, the number of such copies which have type \img5, \img6 or \img7 is $\frac{53}{96} \cdot 3^{n-r-1} + O(\alpha^{n-r-1})$ with probability $1 - O((3/\alpha^2)^{n-r-1})$. The same is true for the types \img1, \img2, \img3 and type \img4, with the constant $\frac{53}{96}$ replaced by $\frac{19}{48}$ and $\frac{5}{96}$, respectively. Now the distribution of the degrees of the midpoints in each of the copies of $G_{r+1}$ only depends on the type, and the different copies are pairwise independent. Let $m_r(\img5,h)$ be the expectation of the random variable that counts how many of the three ``midpoints''
\[ \tfrac12(u_2+u_3), \qquad \tfrac12(u_3+u_1), \qquad \tfrac12(u_1+u_2)\]
have degree $h$ in a random spanning forest of type \img5 in $G_r$, and define $m_r(\img6,h)$, etc. analogously. By symmetry,
\begin{align*}
m_r(\img5,h) &= m_r(\img6,h) = m_r(\img7,h), \\
m_r(\img1,h) &= m_r(\img2,h) = m_r(\img3,h).
\end{align*}
By independence and another application of Chebyshev's inequality, we find that the total number of vertices of degree $h$ among all level-$(n-r)$ vertices in a random spanning tree $T_n$ is
\[ 3^{n-r-1} \bigl( \tfrac{53}{96} m_{r+1}(\img5,h)
			+ \tfrac{19}{48} m_{r+1}(\img1,h)
			+ \tfrac{5}{96} m_{r+1}(\img4,h) \bigr)
	+ O(\alpha^{n-r-1}) \]
for any $r \geq 0$ with probability $1 - O((3/\alpha^2)^{n-r-1})$. Since there are only $O(3^{n/2})$ vertices at levels $\leq n/2$, we can safely ignore them, and we obtain that the total number of vertices of degree $h$ in a random spanning tree $T_n$ is
\begin{align*}
W_n(h) &= \sum_{r=0}^{\lfloor n/2 \rfloor} 3^{n-r-1}
	\bigl( \tfrac{53}{96} m_{r+1}(\img5,h)
		+ \tfrac{19}{48} m_{r+1}(\img1,h)
		+ \tfrac{5}{96} m_{r+1}(\img4,h) \bigr)
	+ O(\alpha^n) \\
&= 3^n \sum_{r=0}^\infty 3^{-r-1}
	\bigl( \tfrac{53}{96} m_{r+1}(\img5,h)
		+ \tfrac{19}{48} m_{r+1}(\img1,h)
		+ \tfrac{5}{96} m_{r+1}(\img4,h) \bigr)
	+ O(\alpha^n)
\end{align*}
with probability $1 - O((3/\alpha^2)^{n/2})$, from which almost sure convergence of $3^{-n}W_n(h)$ follows immediately. It remains to find the values of the constants. Let us for instance determine $m_{r+1}(\img5,1)$:
\begin{align*}
m_{r+1}(\img5,1) &= \tfrac29 \bigl(d_r(\img5,1)+d_r(\img6,1)+d_r(\img7,1)\bigr)
				\bigl(d_r(\img5,0)+d_r(\img6,0)+d_r(\img7,0)\bigr) \\
		&\qquad + \tfrac13 \bigl(d_r(\img5,1)+d_r(\img6,1)+d_r(\img7,1)\bigr)
				\bigl(d_r(\img1,0)+d_r(\img2,0)\bigr) \\
		&\qquad + \tfrac13 \bigl(d_r(\img5,0)+d_r(\img6,0)+d_r(\img7,0)\bigr)
				\bigl(d_r(\img1,1)+d_r(\img2,1)\bigr)
\end{align*}
by the same argument that was used earlier to determine the probabilities of the different degrees. Using \eqref{eq:degprob} we find
\[ m_{r+1}(\img5,1) = \tfrac{1}{1176} \cdot 375^{-r}\cdot (33 \cdot 15^r - 5)^2, \]
and thus
\[ \sum_{r=0}^\infty 3^{-r-1} m_{r+1}(\img5,1) = \frac{49595}{166352}. \]
All other sums are obtained similarly. It follows that the proportion of vertices of degree $1,2,3,4$ in a random spanning tree $T_n$ converges almost surely to
\begin{gather*}
\frac{10957}{40464} \approx 0.270784, \qquad
\frac{6626035}{13636368} \approx 0.485909, \\[3pt]
\frac{2943139}{13636368} \approx 0.215830, \qquad
\frac{124895}{4545456} \approx 0.0274769,
\end{gather*}
respectively. These constants were already determined in \cite{chang2010structure} as the limits of the mean values, but our arguments show that we even have almost sure convergence.

\section{Loop-erased random walk on Sierpi\'nski graphs}
\label{section:lerw}

This section is devoted to the analysis of loop-erased random walks on Sierpi\'nski graphs and their limit process. Let us first recall some definitions, see for instance \cite{lawler2010random}. Let $G$ be a finite and connected graph. The \emph{(chronological) loop erasure} of a walk $x=(x_0,\dotsc,x_n)$ in $G$ yields a new walk $\LE(x)$ which is defined as follows:
\begin{itemize}
\item Set $\iota(0) = \max\{j\le n \,:\, x_j = x_0\}$.
\item If $\iota(k) < n$, then set $\iota(k+1) = \max\{j\le n \,:\, x_j = x_{\iota(k)+1}\}$,
	otherwise set $\iota(k+1)=n$.
\item If $K=\min\{k \,:\, \iota(k)=n\}$, then $\LE(x)=(x_{\iota(0)},\dotsc,x_{\iota(K)})$.
\end{itemize}
It is clear from the definition that $\LE(x)$ is self-avoiding.

\emph{Simple random walk} $(X_n)_{n\ge0}$ on a finite and connected graph $G$ is a Markov chain with state space $VG$ and transition probabilities $p(x,y)$ from vertex $x$ to vertex $y$ given by
\[ p(x,y) = \begin{cases}
	\frac{1}{\deg x} & \text{if $x$ and $y$ are adjacent,} \\
	0 & \text{otherwise.}
\end{cases} \]
For any $B\subseteq VG$, the \emph{hitting time} $\hit(B)$ is given by
\[ \hit(B) = \inf\{ n \,:\, X_n \in B \}. \]
Since $G$ is finite and connected, the hitting time $\hit(B)$ is almost surely finite. Fix a vertex $x\in VG$ and some set $B\subseteq VG$ with $x\notin B$ and consider simple random walk $(X_n)_{n\ge0}$ starting at $x$. The random self-avoiding walk $\LE((X_n)_{0\le n\le\hit(B)})$ is called \emph{loop-erased random walk} from $x$ to $B$. Figure~\ref{figure:walk} shows instances of loop-erased random walks from one corner vertex to another on $G_5$ and $G_8$, respectively. The aim of this and the following section is to study some of the properties of loop-erased walk on $G_n$ and its limit process.

\begin{figure}[htb]
\centering
\includegraphics[width=6cm]{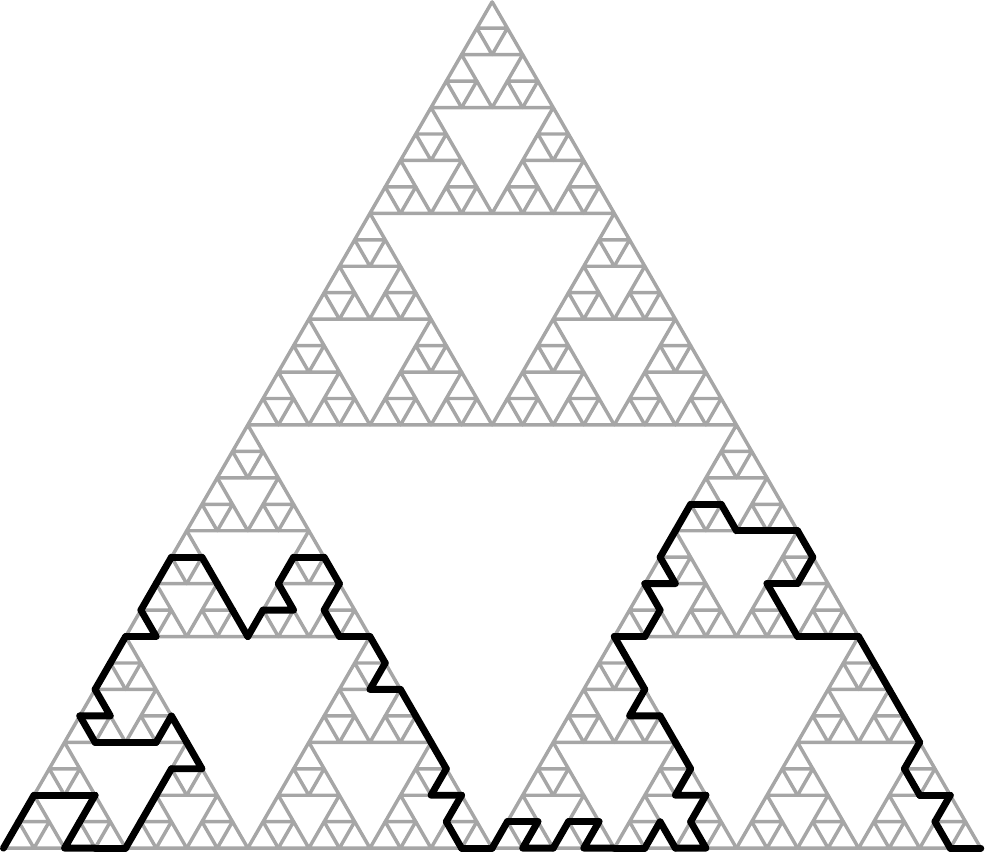}\qquad
\includegraphics[width=6cm]{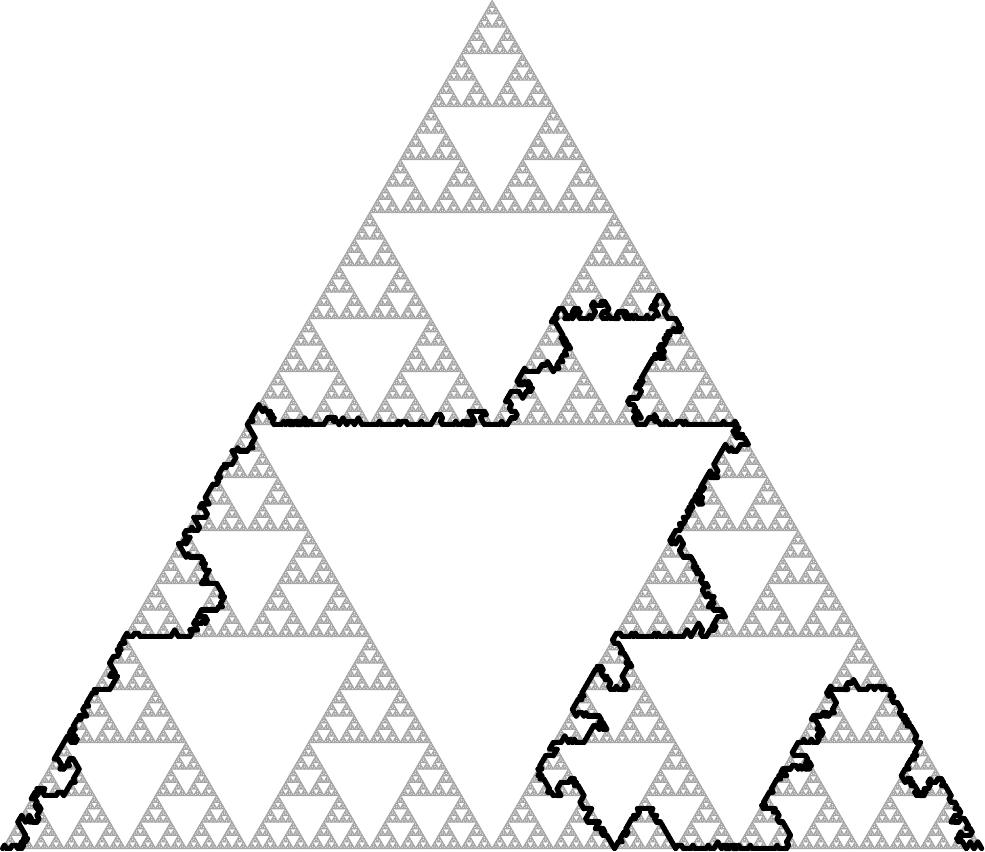}
\caption{Instances of loop-erased random walk on $G_5$ (left) and $G_8$ (right).}
\label{figure:walk}
\end{figure}

Uniform spanning trees and loop-erased random walk are strongly connected concepts. A particular application of this connection is \emph{Wilson's algorithm}~\cite{wilson1996generating}, which is an efficient method for sampling uniform spanning trees of a graph $G$. Fix some ordering of the vertex set $VG$, and let $\{(X_n^x)_{n\ge0} \,:\, x\in VG\}$ be a family of independent simple random walks on $G$, where $(X_n^x)_{n\ge0}$ starts at $x$. We define a sequence $T_0,T_1,\dotsc$ of random subtrees of $G$ as follows:
\begin{itemize}
\item $T_0$ consists of the least vertex (according to the selected ordering) in $G$ only.
\item If $T_k$ does not contain all vertices of $G$, let $x$ be the least vertex in $VG\setminus VT_k$ and define
	\[ T_{k+1} = T_k \cup \LE\bigl((X_n^x)_{0\le n\le\hit(VT_k)}\bigr). \]
	If $T_k$ is already spanning, then set $T_{k+1}=T_k$.
\end{itemize}
By construction there is a minimal (random) index $K$ (at most $\card{VG}$) such that $T_K=T_{K+1}$. Then $T_K$ is a uniform spanning tree of $G$. This idea can be reversed: suppose that $T$ is a uniform spanning tree of $G$, and fix two vertices $x,y\in VG$. The random self-avoiding walk $xTy$ turns out to have precisely the same distribution as a loop-erased random walk from $x$ to $y$: this is easy to see from Wilson's algorithm if we assume that $x$ and $y$ are the least and second-least vertices in our ordering.

In the following we use this connection to study loop-erased random walk on Sierpi\'nski graphs $G_n$ in more detail: For example, if $T$ is a uniformly chosen spanning tree on $G_n$, then $u_1T_nu_2$ is a loop-erased random walk in $G_n$ from $u_1$ to $u_2$. The description of $T_\infty$ as a labelled multi-type Galton-Watson tree can be extended to describe the evolution of loop-erased random walks $u_1T_0u_2, u_1T_1u_2, \dotsc$ by a labelled multi-type Galton-Watson tree with twelve types, which capture not only the structure of the spanning tree, but also the unique path between two corner vertices.

The set $\mathcal{\bar C} = \{\conn55, \conn52, \conn53, \conn61, \conn66, \conn63, \conn71, \conn72, \conn77, \conn11, \conn22, \conn33 \}$ encodes the twelve possible types (in a rather obvious way). Fix $k\in\{1,2,3\}$ and let $v,v'$ be the two vertices in $VG_0$ different from $u_k$. Let $f$ be an element in $\mathcal Q_\infty$, so that $v,v'$ are in the same component of the spanning forest $\Tr^\infty_0 f$. Then $v,v'$ are in the same component of $\Tr^\infty_n f$ for any $n\ge0$. For $n\ge0$ consider those $n$\ndash parts of $G_n$ which contain at least one edge of the self-avoiding walk $v(\Tr^\infty_n f)v'$, and let
\[ W_n(f,k) = \bigl\{ w\in\Words^n \,:\, E(v(\Tr^\infty_n f)v') \cap \psi_w(EG_0) \ne \emptyset \bigr\} \]
be the addresses of these $n$\ndash parts. Notice that $W_n(f,k)$ is naturally ordered by the fact that $v(\Tr^\infty_n f)v'$ walks along the $n$\ndash parts $\psi_w(G_0)$ with $w\in W_n(f,k)$. Furthermore,
\[ W(f,k) = \bigcup_{n\ge0} W_n(f,k) \]
induces a subtree of $\Words^*$, where each word in $W(f,k)$ has two or three children. Of course, $\type_w(f) \in \{\img5,\img6,\img7,\img1,\img2,\img3\}$ for any word $w\in W(f,k)$ (the walk has to enter and leave an $n$\ndash part at a corner, which is only possible if at least two of the corners are connected). Moreover, for any $w\in W_n(f,k)$, the restriction of $E(v(\Tr^\infty_n f)v')$ to $\psi_w(EG_0)$ consists of one edge $e=\{x,x'\}$ or two incident edges $e=\{x,y\}$ and $e'=\{y,x'\}$ for some $x,x',y\in\psi_w(VG_0)$. Define $\bar\kappa_w(f,k)$ to be the unique $i\in\{1,2,3\}$ such that $\psi_w(u_i) \neq x,x'$. We encode the two bits of information given by $\type_w(f)$ and $\bar\kappa_w(f,k)$ by one of the twelve types in $\mathcal{\bar C}$ in a natural way. Write $\bar\type_w(f,k)$ to denote this type of the $n$\ndash part $\psi_w(G_0)$ induced by $f$ and $k$, and set
\[ \mathbold{\bar\type}(f,k) = (\bar\type_w(f,k))_{w\in W(f,k)}. \]
For example, $\bar\type_w(f,k)=\conn55$ if $\type_w(f)=\img5$ and $\bar\kappa_w(f,k)=1$. Other types are assigned accordingly, see Table~\ref{table:conntypes}.

\begin{table}[htb]
\caption{The type $\bar\type_w(f,k)$, given $\type_w(f)$ and $\bar\kappa_w(f,k)$.}
\label{table:conntypes}
\centering
\begin{tabular}{@{}*{7}{c}@{}}
\toprule
                & \multicolumn{6}{c}{$\type_w(f)$} \\
                  \cmidrule(l){2-7}
$\bar\kappa_w(f,k)$ %
                & \img5   & \img6   & \img7   & \img1   & \img2   & \img3   \\
\cmidrule(r){1-1} \cmidrule(l){2-7}
$1$             & \conn55 & \conn61 & \conn71 & \conn11 &         &         \\
$2$             & \conn52 & \conn66 & \conn72 &         & \conn22 &         \\
$3$             & \conn53 & \conn63 & \conn77 &         &         & \conn33 \\
\bottomrule
\end{tabular}
\end{table}

In order to reconstruct the self-avoiding walk $v(\Tr^\infty_n f)v'$ from $\mathbold{\bar\type}(f,k)$, let $\bar\eta$ be the map from $\mathcal{\bar C}$ to the set of subgraphs of $G_0$ defined in Table~\ref{table:sub2}. Then
\[ v(\Tr^\infty_n f)v' = \bigcup_{w\in W_n(f,k)} \psi_w(\bar\eta(\bar\type_w(f,k))). \]
It is noteworthy that in general $\mathbold{\bar\type}(f,k)$ contains more information than all the self-avoiding walks $v(\Tr^\infty_n f)v'$ for $n\ge0$ (since it also contains additional structural information on the underlying spanning tree).

\begin{table}[htb]
\caption{The mappings $\bar\eta$ and $\bar\nu$.}
\label{table:sub2}
\centering
\begin{tabular}{@{}*{13}{c}@{}}
\toprule
$x$ %
 & \conn61 & \conn71 %
 & \conn52 & \conn72 %
 & \conn53 & \conn63 %
 & \conn55 & \conn66 & \conn77
 & \conn11 & \conn22 & \conn33 \\
\midrule
$\bar\eta(x)$ %
 & \tree nyyynn & \tree nyyynn %
 & \tree ynynyn & \tree ynynyn %
 & \tree yynnny & \tree yynnny %
 & \tree yyynyy & \tree yyyyny & \tree yyyyyn %
 & \tree nyyynn & \tree ynynyn & \tree yynnny \\
$\bar\nu(x)$ & 1 & 2 & 3 & 4 & 5 & 6 & 7 & 8 & 9 & 10 & 11 & 12  \\
\bottomrule
\end{tabular}
\end{table}

Last but not least, let $\bar\nu$ be the bijection from $\mathcal{\bar C}$ to $\{1,\dotsc,12\}$ given by Table~\ref{table:sub2}. In analogy to the previous section, we define the type-counting functions $\bar\type^\num_{i,n}(f,k) = \card{\{w\in W_n(f,k) \,:\, \bar\nu(\bar\type_w(f,k)) = i\}}$ and
\[ \mathbold{\bar\type}^\num_n(f,k) = \bigl(\bar\type^\num_{1,n}(f,k), \dotsc, \bar\type^\num_{12,n}(f,k)\bigr) \]
for $i\in\{1,\dotsc,12\}$ and $n\ge0$.

\begin{proposition}\label{proposition:tree2}
Let $\mathcal U_\infty$ be one of $\mathcal T_\infty^{}$, $\mathcal T_\infty^i$, or $\mathcal S_\infty^i$ for $i\in\{1,2,3\}$, and let $U_\infty$ be the corresponding random object. Let $k\in\{1,2,3\}$, and assume that $\Tr^\infty_0 U_\infty$ connects the two vertices in $VG_0\setminus\{u_k\}$.
\begin{enumerate}[\normalfont(1)]
\item The random tree
	\[ \mathbold{\bar\type}(U_\infty,k) = (\bar\type_w(U_\infty,k))_{w\in W(U_\infty,k)} \]
	is a labelled multi-type Galton-Watson tree with labels in $\Words^*$ and types in $\mathcal{\bar C}$.
	The type distribution of the root is given by the uniform distribution
	$\Unif\{\bar\type_\emptyword(f,k) \,:\, f\in\mathcal U_\infty\}$.
	Its offspring generation is given in Table~\ref{table:conngen}.
\item $(\mathbold{\bar\type}^\num_n(U_\infty,k))_{n\ge0}$ is a multi-type Galton-Watson process with twelve types,
	which is non-singular, positively regular, and supercritical.
	Using the abbreviations $s_1=\frac13(z_1+z_2+z_7)$, $s_2=\frac13(z_3+z_4+z_8)$, and $s_3=\frac13(z_5+z_6+z_9)$,
	the offspring generating function is given by
	\begin{align*}
	\mathbold{\bar f}(\mathbold z) = \Bigl(
	&\tfrac12 s_1(s_1+z_{10}),\;
	 \tfrac12 s_1(s_1+z_{10}), \vphantom{\Big(}\\
	&\tfrac12 s_2(s_2+z_{11}),\;
	 \tfrac12 s_2(s_2+z_{11}), \vphantom{\Big(}\\
	&\tfrac12 s_3(s_3+z_{12}),\;
	 \tfrac12 s_3(s_3+z_{12}), \vphantom{\Big(}\\
	&\tfrac12 s_1(s_3z_{11}+s_2z_{12}),\;
	 \tfrac12 s_2(s_3z_{10}+s_1z_{12}),\;
	 \tfrac12 s_3(s_2z_{10}+s_1z_{11}), \vphantom{\Big(}\\
	&\tfrac1{10} \bigl(3s_1^2 + 4s_1z_{10} + z_{10}(z_{10} + s_3z_{11} + s_2z_{12})\bigr), \vphantom{\Big(}\\
	&\tfrac1{10} \bigl(3s_2^2 + 4s_2z_{11} + z_{11}(s_3z_{10} + z_{11} + s_1z_{12})\bigr), \vphantom{\Big(}\\
	&\tfrac1{10} \bigl(3s_3^2 + 4s_3z_{12} + z_{12}(s_2z_{10} + s_1z_{11} + z_{12})\bigr) \Bigr).
	\end{align*}
	Its mean matrix $\mathbold{\bar M}$ is
	\[ \mathbold{\bar M}
		= \frac1{30} \begin{pmatrix}
			15 & 15 &  0 &  0 &  0 &  0 & 15 &  0 &  0 & 15 &  0 &  0 \\
			15 & 15 &  0 &  0 &  0 &  0 & 15 &  0 &  0 & 15 &  0 &  0 \\
			 0 &  0 & 15 & 15 &  0 &  0 &  0 & 15 &  0 &  0 & 15 &  0 \\
			 0 &  0 & 15 & 15 &  0 &  0 &  0 & 15 &  0 &  0 & 15 &  0 \\
			 0 &  0 &  0 &  0 & 15 & 15 &  0 &  0 & 15 &  0 &  0 & 15 \\
			 0 &  0 &  0 &  0 & 15 & 15 &  0 &  0 & 15 &  0 &  0 & 15 \\
			10 & 10 &  5 &  5 &  5 &  5 & 10 &  5 &  5 &  0 & 15 & 15 \\
			 5 &  5 & 10 & 10 &  5 &  5 &  5 & 10 &  5 & 15 &  0 & 15 \\
			 5 &  5 &  5 &  5 & 10 & 10 &  5 &  5 & 10 & 15 & 15 &  0 \\
			10 & 10 &  1 &  1 &  1 &  1 & 10 &  1 &  1 & 24 &  3 &  3 \\
			 1 &  1 & 10 & 10 &  1 &  1 &  1 & 10 &  1 &  3 & 24 &  3 \\
			 1 &  1 &  1 &  1 & 10 & 10 &  1 &  1 & 10 &  3 &  3 & 24
		\end{pmatrix}, \]
	whose dominating eigenvalue $\bar\alpha$ is
	$\frac43+\frac1{15}\sqrt{205} \approx 2.287855$.
	The corresponding right and left eigenvectors are
	\begin{align*}
	\mathbold{\bar v}_R
		&= (a_1,a_1,a_1,a_1,a_1,a_1,a_2,a_2,a_2,a_3,a_3,a_3)^t, \\
	\mathbold{\bar v}_L
		&= (a_4,a_4,a_4,a_4,a_4,a_4,a_4,a_4,a_4,a_5,a_5,a_5),
	\end{align*}
	where
	\begin{gather*}
		a_1 = \tfrac{11}{26} + \tfrac{17}{533}\sqrt{205}, \qquad
		a_2 = \tfrac{17}{26} + \tfrac{49}{1066}\sqrt{205}, \qquad
		a_3 = \tfrac12 + \tfrac{13}{410}\sqrt{205}, \\[4pt]
		a_4 = \tfrac1{18}\sqrt{205} - \tfrac{13}{18}, \qquad
		a_5 = \tfrac52 - \tfrac16\sqrt{205}.
	\end{gather*}
	The vectors $\mathbold{\bar v}_R$ and $\mathbold{\bar v}_L$ are normalized
	so that $\mathbold{\bar v}_L \cdot \mathbold{\bar v}_R = 1$ and $\norm{\mathbold{\bar v}_L}_1=1$.
\item There is a non-negative random variable $\bar\theta(U_\infty,k)$ such that
	\[ \bar\alpha^{-n} \mathbold{\bar\type}^\num_n(U_\infty^{},k) \to
		\mathbold{\bar v}_L \bar\theta(U_\infty,k) \]
	almost surely. The distribution of $\bar\theta(U_\infty,k)$ has a continuous density function,
	which is strictly positive on the set of positive reals and zero elsewhere.
	In particular, $\bar\theta(U_\infty,k)$ is almost surely positive.
	By symmetry, there are four different limit distributions, one for each of the following groups:
	\begin{gather*}
		\{ \bar\theta(T_\infty,1), \bar\theta(T_\infty,2), \bar\theta(T_\infty,3) \}, \\
		\{ \bar\theta(T_\infty^1,2), \bar\theta(T_\infty^1,3),
		   \bar\theta(T_\infty^2,1), \bar\theta(T_\infty^2,3),
		   \bar\theta(T_\infty^3,1), \bar\theta(T_\infty^3,2) \}, \\
		\{ \bar\theta(T_\infty^1,1), \bar\theta(T_\infty^2,2), \bar\theta(T_\infty^3,3) \}, \\
		\{ \bar\theta(S_\infty^1,1), \bar\theta(S_\infty^2,2), \bar\theta(S_\infty^3,3) \}.
	\end{gather*}
	We write $\bar\theta_0, \bar\theta_1, \bar\theta_2, \bar\theta_3$ for random variables
	having the same distribution as a random variable in the respective group (ordered as above).
	Their expected values are $\Expect(\bar\theta_0)=\frac23 a_1+\frac13 a_2$,
	$\Expect(\bar\theta_1)=a_1$, $\Expect(\bar\theta_2)=a_2$ and $\Expect(\bar\theta_3)=a_3$, respectively.
	Moreover, $\Prob_{\bar\theta_0} = \frac23 \Prob_{\bar\theta_1} + \frac13 \Prob_{\bar\theta_2}$.
\end{enumerate}
\end{proposition}

\begin{table}[htb]
\caption{Offspring generation of $\mathbold{\bar\type}(U_\infty,k)$ for three types.
	The remaining types are obtained by symmetry taking suffixes into account.}
\label{table:conngen}
\centering
\begin{tabular}{@{}clc@{}}
\toprule
Type & \multicolumn{1}{c}{Offspring types} & Probability \\
 & \multicolumn{1}{c}{with suffixes $(1,2)$ or $(1,2,3)$} & \\
\midrule
\multirow{4}{10pt}{\conn53}
 & $(\conn53,\conn33)$, $(\conn63,\conn33)$, $(\conn77,\conn33)$ & \multirow{1}{7pt}{$\tfrac16$} \\
\cmidrule(){2-3}
 & $(\conn53,\conn53)$, $(\conn53,\conn63)$, $(\conn53,\conn77)$, & \multirow{3}{11pt}{$\tfrac1{18}$} \\
 & $(\conn63,\conn53)$, $(\conn63,\conn63)$, $(\conn63,\conn77)$, & \\
 & $(\conn77,\conn53)$, $(\conn77,\conn63)$, $(\conn77,\conn77)$  & \\
\midrule
\multirow{6}{10pt}{\conn77}
 & $(\conn22,\conn55,\conn53)$, $(\conn22,\conn55,\conn63)$, $(\conn22,\conn55,\conn77)$, %
	& \multirow{6}{11pt}{$\tfrac1{18}$} \\
 & $(\conn22,\conn61,\conn53)$, $(\conn22,\conn61,\conn63)$, $(\conn22,\conn61,\conn77)$, & \\
 & $(\conn22,\conn71,\conn53)$, $(\conn22,\conn71,\conn63)$, $(\conn22,\conn71,\conn77)$, & \\
 & $(\conn52,\conn11,\conn53)$, $(\conn52,\conn11,\conn63)$, $(\conn52,\conn11,\conn77)$, & \\
 & $(\conn66,\conn11,\conn53)$, $(\conn66,\conn11,\conn63)$, $(\conn66,\conn11,\conn77)$, & \\
 & $(\conn72,\conn11,\conn53)$, $(\conn72,\conn11,\conn63)$, $(\conn72,\conn11,\conn77)$  & \\
\midrule
\multirow{8}{10pt}{\conn33}
 & $(\conn52,\conn11,\conn33)$, $(\conn66,\conn11,\conn33)$, $(\conn72,\conn11,\conn33)$, %
	& \multirow{5}{11pt}{$\tfrac1{30}$} \\
 & $(\conn22,\conn55,\conn33)$, $(\conn22,\conn61,\conn33)$, $(\conn22,\conn71,\conn33)$, & \\
 & $(\conn53,\conn53)$, $(\conn53,\conn63)$, $(\conn53,\conn77)$, & \\
 & $(\conn63,\conn53)$, $(\conn63,\conn63)$, $(\conn63,\conn77)$, & \\
 & $(\conn77,\conn53)$, $(\conn77,\conn63)$, $(\conn77,\conn77)$  & \\
\cmidrule(){2-3}
 & $(\conn53,\conn33)$, $(\conn63,\conn33)$, $(\conn77,\conn33)$, & \multirow{2}{11pt}{$\tfrac1{15}$} \\
 & $(\conn33,\conn53)$, $(\conn33,\conn63)$, $(\conn33,\conn77)$  & \\
\cmidrule(){2-3}
 & $(\conn33,\conn33)$ & \multirow{1}{11pt}{$\tfrac1{10}$} \\
\bottomrule
\end{tabular}
\end{table}

\begin{proof}
The first part of this result follows from Proposition~\ref{proposition:tree1}. The second is a consequence of the first: the details are not difficult to verify. For the last part, see Theorem~1.8.2 and Theorem~1.9.1 in \cite{mode1971multitype}.
\end{proof}

\begin{remark}\label{remark:coll2}
Similar to Remark~\ref{remark:coll1}, we can collapse three groups of types into new types:
\begin{itemize}
\item \conn55, \conn61, \conn71 become \conn01,
\item \conn52, \conn66, \conn72 become \conn02,
\item \conn53, \conn63, \conn77 become \conn03.
\end{itemize}
Fix again some $k\in\{1,2,3\}$, and let $f\in\mathcal Q_\infty$ be such that the vertices in $VG_0\setminus\{u_k\}$ are in the same component of $f$. Now for $w\in W(f,k)$, set
\[ \tilde\type_w(f,k)
	= \begin{cases}
		\conn01 & \text{if } \bar\type_w(f,k) \in \{\conn55,\conn61,\conn71\}, \\
		\conn02 & \text{if } \bar\type_w(f,k) \in \{\conn52,\conn66,\conn72\}, \\
		\conn03 & \text{if } \bar\type_w(f,k) \in \{\conn53,\conn63,\conn77\}, \\
		\bar\type_w(f,k) & \text{otherwise},
	\end{cases} \]
and $\mathbold{\tilde\type}(f,k) = (\tilde\type_w(f,k))_{w\in W(f,k)}$. If $U_\infty$ is now one of $T_\infty^{}$, $T_\infty^i$, $S_\infty^i$ for $i\in\{1,2,3\}$, so that the vertices in $VG_0\setminus\{u_k\}$ are in the same component of $\Tr^\infty_0 U_\infty$, then the random tree $\mathbold{\tilde\type}(U_\infty,k)$ is a labelled multi-type Galton-Watson tree with types in $\{\conn01,\conn02,\conn03,\conn11,\conn22,\conn33\}$.

In order to sample a loop-erased random walk in $G_n$ from $u_1$ to $u_2$, we can simulate the $n$\ndash th generation of $\mathbold{\bar\type}(T_\infty,3)$. At first we have to choose one of \conn53, \conn63, \conn77 with equal probability as the type of the ancestor $\emptyword$. As in Remark~\ref{remark:coll1} we may postpone this choice to the $n$\ndash th generation. To do so, consider the $n$\ndash th generation of the simplified tree $\mathbold{\tilde\type}(T_\infty,3)$. Independently replace each occurrence of
\begin{itemize}
\item \conn01 by one of \conn55, \conn61, \conn71,
\item \conn02 by one of \conn52, \conn66, \conn72,
\item \conn03 by one of \conn53, \conn63, \conn77,
\end{itemize}
always with equal probabilities. Then the modified $n$\ndash th generation of $\mathbold{\tilde\type}(T_\infty,3)$ describes a loop-erased random walk in $G_n$ from $u_1$ to $u_2$.
\end{remark}

\begin{remark}\label{remark:count}
We set
\[ \bar c(x) = \begin{cases}
		1 & \text{if } x \in \{\conn61,\conn71,\conn52,\conn72,\conn53,\conn63\}, \\
		2 & \text{if } x \in \{\conn55,\conn66,\conn77\}, \\
		3 & \text{if } x \in \{\conn11,\conn22,\conn33\},
	\end{cases} \]
and once again, we introduce type counters: for $i\in\{1,2,3\}$ and $n\ge0$, define $\bar c^\num_{i,n}(f,k) = \card{\{w\in W_n(f,k) \,:\, \bar c(\bar\type_w(f))=i\}}$ and
\begin{align*}
\mathbold{\bar c}^\num_n(f,k)
	&= \bigl(\bar c^\num_{1,n}(f,k), \bar c^\num_{2,n}(f,k), \bar c^\num_{3,n}(f,k)\bigr), \\
\mathbold{\tilde c}^\num_n(f,k)
	&= \bigl(\bar c^\num_{1,n}(f,k) + \bar c^\num_{2,n}(f,k), \bar c^\num_{3,n}(f,k)\bigr).
\end{align*}
Then $\mathbold{\bar c}^\num_n(f,k)$ and $\mathbold{\tilde c}^\num_n(f,k)$ count the occurrences of types up to symmetry in the $n$\ndash th generation of $\mathbold{\bar\type}(f,k)$ and $\mathbold{\tilde\type}(f,k)$, respectively.

For a random object $U_\infty$ (one of $T_\infty^{}$, $T_\infty^i$, $S_\infty^i$) and suitable $k$, $(\mathbold{\bar c}^\num_n(U_\infty,k))_{n\ge0}$ and $(\mathbold{\tilde c}^\num_n(U_\infty,k))_{n\ge0}$ are multi-type Galton-Watson processes with offspring generating functions
\begin{equation}\label{eq:gdef}
\mathbold{\bar g}(z_1,z_2,z_3) = \bigl(
	\tfrac12 s(s+z_3), \, s^2 z_3, \, \tfrac3{10} s^2 + \tfrac15 sz_3(2+z_3) + \tfrac1{10} z_3^2 \bigr),
\end{equation}
where $s=\frac23z_1+\frac13z_2$, and
\begin{equation}\label{eq:tildegdef}
 \mathbold{\tilde g}(z_1,z_2) = \bigl(
	\tfrac13 z_1(z_1+z_2+z_1z_2), \, \tfrac3{10} z_1^2 + \tfrac15 z_1z_2(2+z_2) + \tfrac1{10} z_2^2 \bigr),
\end{equation}
respectively. If we set
\begin{equation}\label{eq:sigmadef}
\mathbold\Sigma(z_1,z_2,z_3)
	= \bigl(\PGF(\mathbold{\bar\type}^\num_0(T_\infty^{},k),\mathbold z),
	        \PGF(\mathbold{\bar\type}^\num_0(S_\infty^k,k),\mathbold z)\bigr)
	= \bigl(\tfrac23 z_1 + \tfrac13 z_2, z_3\bigr),
\end{equation}
then $\mathbold\Sigma\circ\mathbold{\bar g} = \mathbold{\tilde g}\circ\mathbold\Sigma$. Note also that $\mathbold{\bar c}^\num_n(U_\infty,k)$ and $\mathbold{\tilde c}^\num_n(U_\infty,k)$ depend linearly on $\mathbold{\bar\type}^\num_n(U_\infty,k)$, hence Proposition~\ref{proposition:tree2} implies
\begin{align*}
\bar\alpha^{-n}\mathbold{\bar c}^\num_n(T_\infty^{},k) &\to (6a_4,3a_4,3a_5) \bar\theta(T_\infty^{},k), &
	\bar\alpha^{-n}\mathbold{\bar c}^\num_n(S_\infty^k,k) &\to (6a_4,3a_4,3a_5) \bar\theta(S_\infty^k,k), \\
\bar\alpha^{-n}\mathbold{\tilde c}^\num_n(T_\infty^{},k) &\to (9a_4,3a_5) \bar\theta(T_\infty^{},k), &
	\bar\alpha^{-n}\mathbold{\tilde c}^\num_n(S_\infty^k,k) &\to (9a_4,3a_5) \bar\theta(S_\infty^k,k)
\end{align*}
almost surely.
\end{remark}

\begin{remark}\label{remark:laplace}
Using the previous remark, it is possible to describe the distribution of $\bar\theta_0,\bar\theta_1,\bar\theta_2,\bar\theta_3$. Let
\[ \mathbold{\bar\phi}(z) = (\Expect(e^{z\bar\theta_1}),\Expect(e^{z\bar\theta_2}),\Expect(e^{z\bar\theta_3}))
	\qquad\text{and}\qquad
	\mathbold{\tilde\phi}(z) = (\Expect(e^{z\bar\theta_0}),\Expect(e^{z\bar\theta_3})) \]
be the moment generating functions of $(\bar\theta_1,\bar\theta_2,\bar\theta_3)$ and $(\bar\theta_0,\bar\theta_3)$, respectively. These two functions exists at least for $z\in\C$ with $\operatorname{Re}(z)\le0$. Furthermore, it is well known that
\[ \mathbold{\bar\phi}(\bar\alpha z) = \mathbold{\bar g}(\mathbold{\bar\phi}(z))
	\qquad\text{and}\qquad
	\mathbold{\tilde\phi}(\bar\alpha z) = \mathbold{\tilde g}(\mathbold{\tilde\phi}(z)). \]
holds whenever both sides are finite, see for instance Theorem~1.8.1 of \cite{mode1971multitype}. Since $\mathbold{\bar g}$ and $\mathbold{\tilde g}$ are both polynomials, the moment generating functions $\mathbold{\bar\phi}$ and $\mathbold{\tilde\phi}$ exist for all $z\in\C$ and are entire functions, see \cite{poincare1890classe}. Furthermore, by iterating the offspring generating function it is possible to approximate the densities of $\bar\theta_0,\bar\theta_1,\bar\theta_2,\bar\theta_3$, see Figure~\ref{figure:density}.
\end{remark}

\begin{figure}[htb]
\def\plotdensity#1{%
\subfloat[density of $\bar\theta_#1$]{
\begin{tikzpicture}[scale=1.2]
	\scope
		\clip (0,0) rectangle (3.0,2.5);
		\draw plot[smooth] file {dens#1.tab};
	\endscope
	\draw[very thin,->] (0,0) -- (3.2,0);
	\draw[very thin,->] (0,0) -- (0,2.7);
	\foreach \x in {0,0.5,...,3.0} { \draw[very thin] (\x,-0.05) node[below,font=\tiny] {$\x$} -- (\x,0.05); }
	\foreach \y in {0,0.5,...,2.5} { \draw[very thin] (-0.05,\y) node[left,font=\tiny] {$\y$} -- (0.05,\y); }
\end{tikzpicture}}}
\centering
\plotdensity0
\qquad
\plotdensity1
\\[5pt]
\plotdensity2
\qquad
\plotdensity3
\caption{A plot of the densities of $\bar\theta_i$ for $i\in\{0,1,2,3\}$. The densities are approximated by $n=7$ iterations of the offspring generating function $\bar g$.}
\label{figure:density}
\end{figure}
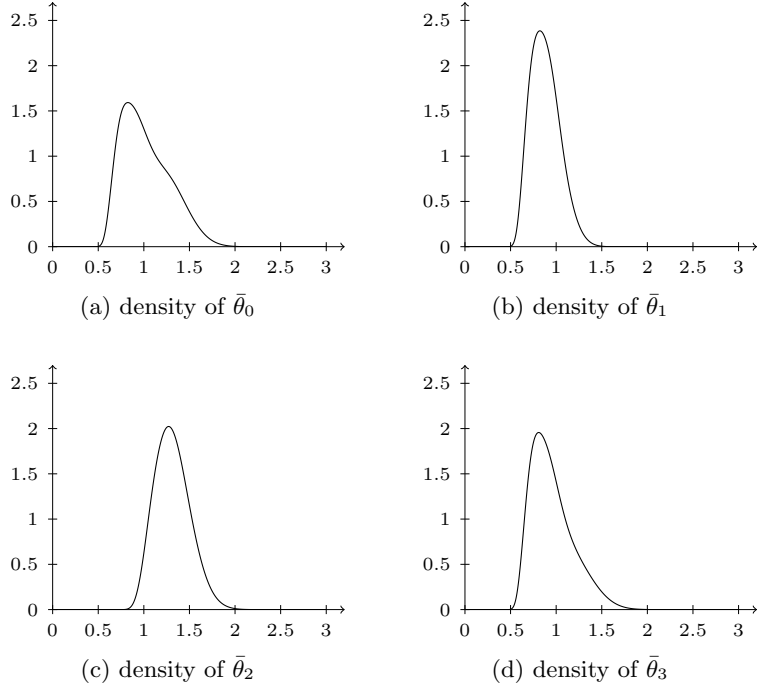

In the following lemma, we prove some estimates for the moment generating functions of $\bar\theta_0,\dotsc,\bar\theta_3$, which lead to estimates for the tails of the distributions. Let us remark that there exist general results concerning tail probabilities (see for instance Jones~\cite{jones2004large}), but our situation does not satisfy the necessary conditions of these results. Thus we follow the arguments in \cite[Proposition~3.1]{barlow1988brownian} and \cite[Proposition~4.2]{kumagai1993construction}. Let the constants $\bar\gamma_\ell$ and $\bar\gamma_r$ be defined by
\[ \bar\gamma_\ell = \frac{\log2}{\log\bar\alpha} \approx 0.837524 \qquad\text{and}\qquad
	\bar\gamma_r = \frac{\log3}{\log\bar\alpha} \approx 1.32744. \]
Thus $-\bar\gamma_\ell/(1-\bar\gamma_\ell)\approx5.154759$ and $\bar\gamma_r/(1-\bar\gamma_r)\approx4.053954$. These constants play an important role in the following lemma:

\begin{lemma}\label{lemma:mg-bounds}
There are constants $C_{1,\ell},C_{2,\ell} > 0$ such that
\[ e^{-C_{1,\ell} \abs{z}^{\bar\gamma_\ell}}
	\le \Expect(e^{z\bar\theta_i})
	\le e^{-C_{2,\ell} \abs{z}^{\bar\gamma_\ell}}
	\rlap{$\quad (i\in\{0,1,2,3\})$} \]
for all $z\le-1$. The upper bounds also hold for $z\in\C$ with $\Re z\le0$ and $\abs{z}\ge1$ (after taking absolute values). Analogously, there are constants $C_{1,r},C_{2,r} > 0$ such that
\[ e^{C_{1,r} z^{\bar\gamma_r}}
	\le \Expect(e^{z\bar\theta_i})
	\le e^{C_{2,r} z^{\bar\gamma_r}}
	\rlap{$\quad (i\in\{0,1,2,3\})$} \]
for all sufficiently large $z\ge0$ (for instance if $\Expect(e^{z\bar\theta_i})\ge4$ for $i\in\{1,2,3,4\}$). As a consequence the following statements hold:
\begin{itemize}
\item There are constants $C_{3,\ell},C_{4,\ell},C_{5,\ell},C_{6,\ell}>0$ such that
	\[ C_{3,\ell} \exp(-C_{4,\ell} s^{-\bar\gamma_\ell/(1-\bar\gamma_\ell)}) \le \Prob(\bar\theta_i \le s)
		\le C_{5,\ell} \exp(-C_{6,\ell} s^{-\bar\gamma_\ell/(1-\bar\gamma_\ell)}) \]
	for all $s\ge0$ and all $i\in\{0,1,2,3\}$.
\item There are constants $C_{3,r},C_{4,r},C_{5,r},C_{6,r}>0$ such that
	\[ C_{3,r} \exp(-C_{4,r} s^{\bar\gamma_r/(1-\bar\gamma_r)}) \le \Prob(\bar\theta_i \ge s)
		\le C_{5,r} \exp(-C_{6,r} s^{\bar\gamma_r/(1-\bar\gamma_r)}) \]
	for all $s\ge0$ and all $i\in\{0,1,2,3\}$.
\item The random variables $\bar\theta_0,\bar\theta_1,\bar\theta_2,\bar\theta_3$ have densities in $C^\infty$.
\end{itemize}
\end{lemma}

\begin{proof}
Set $\C_-=\{z\in\C \,:\, \Re z\le 0\}$. The random variables $\bar\theta_0,\bar\theta_1,\bar\theta_2,\bar\theta_3\ge0$ have positive densities on $(0,\infty)$. Thus $0<\abs{\Expect(e^{z\bar\theta_i})}<1$ for all $z\in\C_-\setminus\{0\}$ and for all $i\in\{0,1,2,3\}$.

We start with the upper bounds of the left tail. Set $M(z) = \max\{\abs{\Expect(e^{z\bar\theta_i})} \,:\, i\in\{0,1,2,3\} \}$. Then $M(\bar\alpha z) \le M(z)^2$ for all $z\in\C_-$ using the functions $\bar g$ and $\tilde g$. Set $H(z) = -\abs{z}^{-\bar\gamma_\ell} \log M(z)$, so that $H(\bar\alpha z) \ge H(z)$ for all $z\in\C_-$. Due to continuity there is a constant $C_{2,\ell}>0$ such that $H(z)\ge C_{2,\ell}$ for all $z\in\C_-$ with $1\le\abs{z}\le\bar\alpha$. This implies $H(z)\ge C_{2,\ell}$ for all $z\in\C_-$ with $\abs{z}\ge1$ and thus $\abs{\Expect(e^{z\bar\theta_i})} \le e^{-C_{2,\ell} \abs{z}^{\bar\gamma_\ell}}$ for all $z\in\C_-$ with $\abs{z}\ge1$ and $i\in\{0,1,2,3\}$.

For the lower bounds of the left tail set $m(z) = \min\{\Expect(e^{z\bar\theta_i}) \,:\, i\in\{0,1,3\} \}$, so that $m(\bar\alpha z) \ge \tfrac1{10} m(z)^2$ for all $z\le0$. If we set $h(z) = -\abs{z}^{-\bar\gamma_\ell} \log m(z)$, then
\[ h(\bar\alpha z) \le \tfrac12 \abs{z}^{-\bar\gamma_\ell} \log10 + h(z) \]
for all $z\le0$. For $n\ge0$ this implies
\[ h(\bar\alpha^n z) \le \bigl( (\tfrac12)^1 + \dotsb + (\tfrac12)^n \bigr) \abs{z}^{-\bar\gamma_\ell} \log10 + h(z)
	\le \abs{z}^{-\bar\gamma_\ell} \log10 + h(z). \]
As before, there is a constant $C_{1,\ell}>0$ such that $\abs{z}^{-\bar\gamma_\ell} \log10 + h(z)\le C_{1,\ell}$ for all $-\bar\alpha\le z\le-1$. This implies $h(z)\le C_{1,\ell}$ for all $z\le-1$ and so $\Expect(e^{z\bar\theta_i}) \ge e^{-C_{1,\ell} \abs{z}^{\bar\gamma_\ell}}$ for all $z\le0$ and $i\in\{0,1,3\}$. If $i=2$, notice that
\[ \bar g_2(z_1,z_2,z_3) \ge \tfrac49 z_1^2 z_3 \]
for all $z_1,z_2,z_3\ge0$. Hence, using the lower bounds above,
\[ \Expect(e^{\bar\alpha z\bar\theta_2}) \ge \tfrac49 e^{-3C_{1,\ell} \abs{z}^{\bar\gamma_\ell}} \]
for all $z\le-1$. By a suitable modification of $C_{1,\ell}$ we get the lower bound for the case $i=2$.

The proof of the bounds for the right tail is very similar to the proof for the left tail, hence we omit the details.

For the remaining statements, see \cite[Proposition~3.2, Lemma~3.4]{barlow1988brownian} and \cite[Corollary~4.12.8]{bingham1987regular}.
\end{proof}

Analogous to Remark~\ref{remark:func1} it is easy to describe the limit behaviour of any parameter of loop-erased random walk in $G_n$ from $u_1$ to $u_2$ that is a functional of $\mathbold{\bar\type}^\num_n(T_\infty)$. As a simple example we consider the length of loop-erased random walk in $G_n$ from $u_1$ to $u_2$, which is given by the distance $d_{T_n}(u_1,u_2)$, where $d_{T_n}$ is the graph metric of the tree $T_n$. We remark that a similar derivation of the expectations below is given in \cite{dhar1997distribution,hattori2012looperased}.

\begin{corollary}\label{corollary:length}
If $n\ge0$, then the probability generating functions of $d_{T_n^{}}(u_1,u_2)$ and $d_{S_n^3}(u_1,u_2)$ are given by, with $\mathbold{\tilde g}$, $\mathbold{\bar g}$ and $\mathbold{\Sigma}$ as defined in~\eqref{eq:gdef},~\eqref{eq:tildegdef},~\eqref{eq:sigmadef},
\[ \bigl(\PGF(d_{T_n^{}}(u_1,u_2),z),\PGF(d_{S_n^3}(u_1,u_2),z)\bigr)
	= \mathbold{\Sigma}(\mathbold{\bar g}^n(z,z^2,z))
	= \mathbold{\tilde g}^n\bigl( \tfrac23 z + \tfrac13 z^2, z \bigr) \]
and the expectations are
\[ \begin{pmatrix} \Expect(d_{T_n^{}}(u_1,u_2)) \\ \Expect(d_{S_n^3}(u_1,u_2)) \end{pmatrix}
	= \begin{pmatrix}
		\tfrac23 + \tfrac5{123}\sqrt{205}    & \tfrac23 - \tfrac5{123}\sqrt{205} \\[4pt]
		\tfrac12 + \tfrac{19}{410}\sqrt{205} & \tfrac12 - \tfrac{19}{410}\sqrt{205}
	\end{pmatrix} \cdot
	\begin{pmatrix}
		\bigl( \tfrac43 + \tfrac1{15} \sqrt{205}\,\bigr)^n \\[4pt]
		\bigl( \tfrac43 - \tfrac1{15} \sqrt{205}\,\bigr)^n
	\end{pmatrix}. \]
Furthermore,
\[ \bar\alpha^{-n} d_{T_n^{}}(u_1,u_2) \to \tfrac16(\sqrt{205}-7) \bar\theta(T_\infty,3), \qquad
	\bar\alpha^{-n} d_{S_n^3}(u_1,u_2) \to \tfrac16(\sqrt{205}-7) \bar\theta(S_\infty,3) \]
almost surely as $n\to\infty$.
\end{corollary}

\begin{proof}
By the description using Galton-Watson trees, see Proposition~\ref{proposition:tree2} and Remark~\ref{remark:count}, we infer that
\begin{align*}
d_{T_n^{}}(u_1,u_2)
	&= \mathbold{\bar c}^\num_n(T_\infty^{},3) \cdot (1,2,1)^t, \\
d_{S_n^3}(u_1,u_2)
	&= \mathbold{\bar c}^\num_n(S_\infty^3,3) \cdot (1,2,1)^t.
\end{align*}
This implies the statement, since $(6a_4,3a_4,3a_5)\cdot (1,2,1)^t = \frac16(\sqrt{205}-7)$.
\end{proof}

\section{Convergence of loop-erased random walk}
\label{section:convergence}

Let $C$ be the set of continuous curves $\gamma\colon[0,\infty]\to K$ with $\gamma(0)=u_1$ and $\gamma(\infty)=u_2$ and set $d_C(\gamma,\delta) = \sup\{ \norm{\gamma(t)-\delta(t)}_2 \,:\, t\in[0,\infty] \}$ for $\gamma,\delta\in C$. Then $(C,d_C)$ is a complete separable metric space. For $\gamma\in C$ set
\[ \hit(\gamma) = \inf\{t \,:\, \gamma(s) = u_2 \text{ for all } s\ge t \} \in (0,\infty]. \]
A curve $\gamma\in C$ is called \emph{self-avoiding} if $\gamma(s)\ne\gamma(t)$ for $0\le s<t\le\hit(\gamma)$. Fix some curve $\gamma$ in $C$ and some integer $n\ge 0$. Then there is a unique integer $m\ge1$ and two unique sequences
\[ 0 = t_0 < \dotsb < t_m = \hit(\gamma) \]
and $w_1,\dotsc,w_m\in\Words^n$ with the following properties:
\begin{itemize}
\item The curve $\gamma$ walks along the $n$\ndash parts $\psi_{w_j}(K)$:
	$\gamma([t_{j-1},t_j])\subseteq\psi_{w_j}(K)$ and
	$\gamma([t_{j-1},t_j])\cap\psi_{w_j}(K\setminus VG_0) \ne \emptyset$
	for all $1\le j\le m$.
\item The quantity $t_j$ is the exit time of $\gamma$ from $\psi_{w_j}(K)$:
	$t_j = \inf\{s > t_{j-1} \,:\, \gamma(s)\notin\psi_{w_j}(K)\}$ for all $1\le j\le m-1$.
\end{itemize}
As a consequence, the intersection of $\psi_{w_{j-1}}(K)$ and $\psi_{w_j}(K)$ consists of one point only, which is equal to $\gamma(t_{j-1})\in VG_n$. We write $\Delta_n(\gamma)$ to denote the number $m$ of $n$\ndash parts traversed, $\time_{j,n}(\gamma)$ to denote the time $t_j$, and we set
\[ W_n(\gamma) = (w_1,\dotsc,w_m). \]
Last but not least set $\spend_{j,n}(\gamma)=\time_{j,n}(\gamma)-\time_{j-1,n}(\gamma)$, which is the time spent in the $n$\ndash part $\psi_{w_j}(K)$. It should be stressed, that $\bigl(\time_{j,n}(\gamma)\bigr)_{j=0,\dotsc,\Delta(n)}$ are in general not equal to the consecutive hitting times on the set $VG_n$, as it might happen, that the curve $\gamma$ enters the part $\psi_{w_j}(K)$ at $\psi_{w_j}(u_1)$, visits $\psi_{w_j}(u_3)$ without leaving $\psi_{w_j}(K)$, and leaves at $\psi_{w_j}(u_2)$.

By linear interpolation and constant extension we can associate to any walk $x=(x_0,\dotsc,x_r)$ in $G_n$ a curve $\LI(x)\colon[0,\infty]\to K$ as follows:
\begin{itemize}
\item Linear interpolation: set $\LI(x)(t) = (k+1-t) x_k + (t-k) x_{k+1}$ if $k\in\{0,\dotsc,r-1\}$ and $k\le t< k+1$.
\item Constant extension: set $\LI(x)(t) = x_r$ for $t\ge r$.
\end{itemize}
If $\lambda>0$, write $\LI(x,\lambda)$ for the curve with rescaled time, i.e., $\LI(x,\lambda)(t) = \LI(x)(\lambda t)$. Note that $\LI(x,\lambda)\in C$ if $x_0=u_1$ and $x_r=u_2$.

\begin{remark}\label{remark:conn}
Let $t\in\mathcal T_\infty$ and set $\gamma_n = \LI(u_1(\Tr^\infty_n t)u_2) \in C$ for $n\ge0$. If $m\ge n$, then the number $\Delta_n(\gamma_m)$ of $n$\ndash parts visited by $\gamma_m$ is given by
\[ \Delta_n(\gamma_m) = \mathbold{\bar c}^\num_n(t,3) \cdot (1,1,1)^t, \]
since $\mathbold{\bar c}^\num_n(t,3)$ counts the $n$\ndash parts on the unique path from $u_1$ to $u_2$ by their type. Moreover, the words in $W_n(\gamma_m)$ associated to the $n$\ndash parts visited by $\gamma_m$ and the labels $W_n(t,3)$ of the $n$\ndash th generation of the tree $\mathbold{\bar\chi}(t,3)$ are equal, if the natural ordering of $W_n(t,3)$ is used. Finally, the length of the self-avoiding walk $u_1(\Tr^\infty_n t)u_2$ is given by
\[ \hit(\gamma_n) = d_{\Tr^\infty_n t}(u_1,u_2) = \mathbold{\bar c}^\num_n(t,3) \cdot (1,2,1)^t, \]
since types $\conn55,\conn66,\conn77$ contribute $2$ to the length while all other types contribute $1$. If $m\ge n$ and $0\le j\le\Delta_n(\gamma_n)$, then
\[ \gamma_m(\time_{j,n}(\gamma_m)) = \gamma_n(\time_{j,n}(\gamma_n)) \in VG_n. \]
Let $x_{j,n} = \gamma_n(\time_{j,n}(\gamma_n)) \in VG_n$ and $W_n(\gamma_n) = (w_1,\dotsc,w_{\Delta_n(\gamma_n)})$. It follows that, for any $0\le i<j\le\Delta_n(\gamma_n)$,
\begin{itemize}
\item the vertices $x_{i,n}$ and $x_{j,n}$ are not the same,
\item at $x_{j-1,n}$ the self-avoiding walk $u_1(\Tr^\infty_m t)u_2$ enters the $n$\ndash part $\psi_{w_j}(K)$ and
	at $x_{j,n}$ it leaves this $n$\ndash part,
\item the quantity $\spend_{j,n}(\gamma_m)$ is the length of the self-avoiding walk $u_1(\Tr^\infty_m t)u_2$
	restricted to the segment from $x_{j-1,n}$ to $x_{j,n}$, i.e.,
	it is equal to the number of edges of this walk inside the $n$\ndash part $\psi_{w_j}(K)$:
	\[ \spend_{j,n}(\gamma_m) = d_{\Tr^\infty_m t}(x_{j-1,n},x_{j,n})
		= \mathbold{\bar c}^\num_{m-n}(\pi_{w_j}(t),\bar\kappa_{w_j}(t,3)) \cdot (1,2,1)^t. \]
\end{itemize}
\end{remark}

The results of Section~\ref{section:lerw} indicate that $\LI(u_1T_nu_2,\bar\alpha^n)$ converges almost surely for $n\to\infty$. The proof of this fact closely follows the arguments of \cite{barlow1988brownian,hattori1991selfavoiding,kumagai1993construction}. In the first two references uni-type Galton-Watson processes are used, whereas in the last reference a Galton-Watson process with four types is used.

A pair $(W,(b_w)_{w\in W})$ with $W\subseteq\Words^n$ and $b_w\in\mathcal{\bar C}$ is called \emph{admissible} of length $n$ if there is an element $t\in\mathcal T_\infty$ such that $W=W_n(t,3)$ and $b_w=\bar\type_w(t)$ for $w\in W$. Notice that $W$ inherits the natural ordering from $W_n(t,3)$. An admissible pair $(W,(b_w)_{w\in W})$ completely describes the self-avoiding walk connecting $u_1$ and $u_2$ in the spanning tree $\Tr^\infty_n t$ for some $t\in\mathcal T_\infty$. Loosely speaking, the following lemma states that conditioning on the $n$\ndash th level, i.e. conditioning on $W_n(T_\infty,3) = W$ and $(\bar\type_w(T_\infty,3))_{w\in W} = (b_w)_{w\in W}$ for some admissible pair $(W,(b_w)_{w\in W})$, the refinements in different $n$\ndash parts are conditionally independent and for each $n$\ndash part the refinement yields again a multi-type Galton-Watson tree.

\begin{lemma}\label{lemma:indep}
Let $(W,(b_w)_{w\in W})$ be an admissible pair of length $n$. Then, under $\Prob(\,\cdot \mid W_n(T_\infty,3) = W, (\bar\type_w(T_\infty,3))_{w\in W} = (b_w)_{w\in W})$, the following holds:
\begin{itemize}
\item The random trees $\mathbold{\bar\type}(\pi_w(T_\infty),\bar\kappa_w(T_\infty,3))$
	for $w\in W$ are independent labelled multi-type Galton-Watson trees
	with labels in $\Words^*$ and types in $\mathcal{\bar C}$
	as described in Proposition~\ref{proposition:tree2}.
\item For $w\in W$,
	$\bar\alpha^{-n}\mathbold{\bar\type}^\num_n(\pi_w(T_\infty),\bar\kappa_w(T_\infty,3))$
	converges almost surely to $\mathbold{\bar v}_L \bar\theta(w)$
	for some non-negative random variable $\bar\theta(w)$,
	which has the same distribution as $\bar\theta_{\bar c(b_w)}$.
	In particular, $\bar\theta(w)$ is almost surely positive.
	The random variables $\bar\theta(w)$ for $w\in W$ are independent.
\item We have $\Delta_n(\LI(u_1T_nu_2)) = \card{W}$ almost surely.
	Let $(w_1,w_2,\dotsc)$ be the natural ordering of $W$ and let $m\ge n$.
	Then the random variables
	\[ \spend_{j,n}(\LI(u_1T_mu_2,\bar\alpha^m))
		= \bar\alpha^{-m}\mathbold{\bar c}^\num_{m-n}(\pi_{w_j}(T_\infty),\bar\kappa_{w_j}(T_\infty,3))
			\cdot (1,2,1)^t \]
	for $1\le j\le\card{W}$ are independent and
	\[ \spend_{j,n}(\LI(u_1T_mu_2,\bar\alpha^m)) \to \tfrac16(\sqrt{205}-7)\bar\alpha^{-n}\bar\theta(w_j) \]
	almost surely as $m\to\infty$.
\end{itemize}
\end{lemma}

\begin{proof}
The first two parts are consequences of Proposition~\ref{proposition:tree2}. The third part follows from the first and the second and from Remark~\ref{remark:conn}.
\end{proof}

In the following we set $X_n = \LI(u_1T_nu_2,\bar\alpha^n)$, so that $X_n\colon\Omega\to C$ is a random element in $C$ and
\[ X_n(\time_{j,n}(X_n)) = X_m(\time_{j,n}(X_m)) \]
for all $m\ge n$. Define
\[ \Omega' = \Bigl\{ \omega\in\Omega \,:\, \lim_{m\to\infty} \spend_{j,n}(X_m) \in (0,\infty)
			\text{ for } n\ge0, 1\le j\le\Delta_n(X_n) \Bigr\}. \]
Then using Lemma~\ref{lemma:indep} we conclude that $\Prob(\Omega')=1$. Fix some $\omega\in\Omega'$. For $n\ge0$ and $1\le j\le\Delta_n(X_n)$ set
\[ \lspend_{j,n} = \lim_{m\to\infty} \spend_{j,n}(X_m). \]
It follows that
\[ \lim_{m\to\infty} \time_{j,n}(X_m)
	= \lim_{m\to\infty} \sum_{1\le k\le j} \spend_{k,n}(X_m)
	= \sum_{1\le k\le j} \lspend_{k,n} \in (0,\infty). \]
We write $\ltime_{j,n}$ to denote this limit. Lastly, note that
\[ \hit(X_m) = \time_{1,0}(X_m) = \time_{\Delta_n(X_n),n}(X_m) \qquad\text{and thus}\qquad
	\ltime_{1,0} = \ltime_{\Delta_n(X_n),n}. \]

\begin{theorem}\label{theorem:conv}
On $\Omega'$ the curve $X_n$ converges uniformly as $n\to\infty$ to a limit curve $X$ in $C$, which satisfies the following properties:
\begin{itemize}
\item $X(\ltime_{j,n}) = X_n(\time_{j,n}(X_n)) \in VG_n$ for all $n\ge0$ and $0\le j\le\Delta_n(X_n)$.
\item If $W_n(X_n) = \{w_1,\dotsc,w_{\Delta_n(X_n)}\}$, then
	\begin{gather*}
		X(\ltime_{i,n})\ne X(\ltime_{j,n}), \qquad
		X(\ltime_{j-1,n}), X(\ltime_{j,n})\in\psi_{w_j}(VG_0), \\
		X([\ltime_{j-1,n},\ltime_{j,n}])\subseteq\psi_{w_j}(K), \qquad
		X([\ltime_{j-1,n},\ltime_{j,n}])\cap\psi_{w_j}(K\setminus VG_0)\ne\emptyset
	\end{gather*}
	for all $0\le i<j\le\Delta_n(X_n)$.
	Hence $\Delta_n(X) = \Delta_n(X_n)$ and $W_n(X) = W_n(X_n)$ for all $n\ge0$.
\end{itemize}
\end{theorem}

\begin{proof}
We closely follow the arguments of \cite{hattori1991selfavoiding}. Fix $\omega\in\Omega'$. We will show that $X_n$ converges uniformly in $[0,\infty]$.

Let $n\ge1$ be a non-negative integer. Then $\Delta_n(X_n)\ge2$. By Definition of $\Omega'$ we have
\[ a = \min\{\lspend_{j,n} \,:\, 1\le j\le\Delta_n(X_n)\} > 0. \]
Hence there is a positive integer $M=M(n,\omega)$ with $M\ge n$ such that
\[ \max\{\abs{\time_{j,n}(X_m) - \ltime_{j,n}} \,:\, 0\le j\le\Delta_n(X_n)\} \le a \]
for all $m\ge M$. For convenience set $\ltime_{\Delta_n(X_n)+1,n} = \time_{\Delta_n(X_n)+1,n}(X_m) = \infty$ for all $0\le n\le m$. Now consider $t\in[0,\infty]$. There is an integer $j$ with $1\le j\le\Delta_n(X_n)+1$ such that $\ltime_{j-1,n}\le t\le\ltime_{j,n}$. Let $m\ge M$ and distinguish the following cases:
\begin{itemize}
\item $1<j<\Delta_n(X_n)$: We infer that
	\begin{gather*}
	\time_{j-2,n}(X_m) \le \ltime_{j-2,n} + a \le \ltime_{j-2,n} + \lspend_{j-1,n} = \ltime_{j-1,n} \le t, \\
	t \le \ltime_{j,n} = \ltime_{j+1,n} - \lspend_{j+1,n} \le \ltime_{j+1,n} - a \le \time_{j+1,n}(X_m).
	\end{gather*}
	Since $X_m([\time_{j-2,n}(X_m),\time_{j+1,n}(X_m)]) \subseteq \psi_{w_1}(K)\cup\psi_{w_2}(K)\cup\psi_{w_3}(K)$
	for some $w_1,w_2,w_3\in\Words^n$ with $\psi_{w_1}(K)\cap\psi_{w_2}(K) = \{X_m(\time_{j-1,n}(X_m))\}$ and
	$\psi_{w_2}(K)\cap\psi_{w_3}(K) = \{X_m(\time_{j,n}(X_m))\}$, we obtain
	\[ \norm{ X_m(t) - X_m(\time_{j-1,n}(X_m)) }_2 \le 2^{1-n}. \]
\item $j=1$: It follows as before that $0\le t \le \time_{2,n}(X_m)$ for all $m\ge M$. Hence
	\[ \norm{ X_m(t) - X_m(\time_{0,n}(X_m)) }_2 \le 2^{1-n}. \]
\item $j=\Delta_n(X_n)$: Again, $\time_{\Delta_n(X_n)-2,n}(X_m) \le t \le \time_{\Delta_n(X_n)+1,n}$ and therefore
	\[ \norm{ X_m(t) - X_m(\time_{\Delta_n(X_n)-1,n}(X_m)) }_2 \le 2^{-n}. \]
\item $j=\Delta_n(X_n)+1$: Then $\time_{\Delta_n(X_n)-1,n}(X_m) \le t \le \time_{\Delta_n(X_n)+1,n}$ and
	\[ \norm{ X_m(t) - X_m(\time_{\Delta_n(X_n),n}(X_m)) }_2 \le 2^{-n}. \]
\end{itemize}
In any case we have
\[ \norm{ X_m(t) - X_m(\time_{j-1,n}(X_m)) }_2 \le 2^{1-n} \]
for $m\ge M$. Now let $m_1,m_2\ge M$. Since $X_{m_1}(\time_{j-1,n}(X_{m_1})) = X_{m_2}(\time_{j-1,n}(X_{m_2}))$, the estimate above implies
\begin{multline*}
\norm{ X_{m_1}(t) - X_{m_2}(t) }_2 \\
\le \norm{ X_{m_1}(t) - X_{m_1}(\time_{j-1,n}(X_{m_1})) }_2
	+ \norm{ X_{m_2}(t) - X_{m_2}(\time_{j-1,n}(X_{m_2})) }_2
\le 2^{2-n}.
\end{multline*}
As $X_m(0)=u_1$ and $X_m(\infty)=u_2$, we have proved that $X_n$ converges uniformly to a limit curve $X$ in $C$.

The first property listed in Theorem~\ref{theorem:conv} follows from the fact that $X_m(\time_{j,n}(X_m)) = X_n(\time_{j,n}(X_n))$ for all $m\ge n$ and $\time_{j,n}(X_m)\to\ltime_{j,n}$, $X_m\to X$ uniformly as $m\to\infty$.

In order to show the second property let $t$ be in $(\ltime_{j-1,n},\ltime_{j,n})$. Then, for sufficiently large $m$, $t\in(\time_{j-1,n}(X_m),\time_{j,n}(X_m))$. Due to Remark~\ref{remark:conn} we have $X_m(t)\in\psi_{w_j}(K)$ for all sufficiently large $m$. As $X_m(t)\to X(t)$ it follows that $X(t)\in\psi_{w_j}(K)$. Thus Remark~\ref{remark:conn} and the first property imply the second.
\end{proof}

Let $\gamma$ be a curve in $C$, $w$ be a word in $\Words^*$, and $\iota$ be a letter in $\Words$. We say that $\gamma$ has a \emph{peak} of type $\iota$ in the $n$\ndash part $\psi_w(K)$ if there are $t_1<t_2$ such that
\begin{itemize}
\item $\gamma([t_1,t_2])\subseteq\psi_w(K)$,
\item $\gamma(t_1)\ne\gamma(t_2)$ and $\gamma(t_1),\gamma(t_2)\in\psi_w(VG_0\setminus\{u_\iota\})$,
\item $\gamma((t_1,t_2))\cap\psi_w(VG_0)=\{\psi_w(u_\iota)\}$.
\end{itemize}
Intuitively speaking, this means that the curve passes through one of the corners of the $n$\ndash part $\psi_w(K)$ without moving on to the adjacent part.

\begin{lemma}\label{lemma:peaks}
Almost surely, the limit curve $X$ has no peaks. In particular,
\[ X([0,\infty])\cap VG_n = \{ u_1 = X(\ltime_{0,n}), X(\ltime_{1,n}), \dotsc, X(\ltime_{\Delta_n(X_n),n}) = u_2 \} \]
almost surely for all $n\ge0$. If $i,j\in\{1,\dotsc,\Delta_n(X_n)\}$ with $i<j-1$, then
\[ X([\ltime_{i-1,n},\ltime_{i,n}]) \cap X([\ltime_{j-1,n},\ltime_{j,n}]) = \emptyset \]
almost surely. Finally, $\card{X([0,\infty])\cap\psi_w(VG_0)}\le2$ for all $w\in\Words^*$ almost surely.
\end{lemma}

\begin{proof}
If $\iota\in\Words$ and $n\ge0$, we write $\iota^n\in\Words^n$ for the $n$\ndash fold repetition of the letter $\iota$ and define $x(\iota)\in\mathcal{\bar C}$ by
\[ x(\iota) = \begin{cases}
		\conn55 & \text{if } \iota=1, \\
		\conn66 & \text{if } \iota=2, \\
		\conn77 & \text{if } \iota=3. \\
	\end{cases} \]
Let $w$ be a word in $\Words^n$ for some $n\ge0$. For $m\ge n$ write $A_m=A_m(w,\iota)$ to denote the event
\[ A_m = \{ w\iota^{m-n}\in W_m(T_\infty,3), \bar\type_{w\iota^{m-n}}(T_\infty,3)=x(\iota) \}. \]
Then, for any $m\ge n$, $A_m\supseteq A_{m+1}$ and
\[ \Prob(A_{m+1} \mid A_m) = \tfrac6{18} = \tfrac13 \]
as one can see easily by inspection of Table~\ref{table:conngen}. Hence
\[ \Prob(A_m) = \bigl(\tfrac13\bigr)^{m-n} \Prob(A_n). \]
Since
\[ \{ X \text{ has a peak of type } \iota \text{ in } \psi_w(K)\} = \bigcap_{m\ge n} A_m, \]
we infer that
\[ \Prob(X \text{ has a peak of type } \iota \text{ in } \psi_w(K)) = 0. \]
This yields
\[ \Prob(X \text{ has a peak})
	\le \sum_{w\in\Words^*}\sum_{\iota\in\Words} \Prob(X \text{ has a peak of type } \iota \text{ in } \psi_w(K))
	= 0. \]
In order to show the last assertion of the lemma, let $w$ be a word in $\Words^n$. For $i\in\{1,2,3\}$ let $w_i$ be the word in $\Words^n$ (if it exists) for which $w_i\ne w$ and $\psi_w(K)\cap\psi_{w_i}(K)=\{\psi_w(u_i)\}$. If $\card{X([0,\infty])\cap\psi_w(VG_0)}=3$ for some $w\in\Words^*$, then $X$ has a peak in one of the parts $\psi_w(K)$, $\psi_{w_1}(K)$, $\psi_{w_2}(K)$, $\psi_{w_3}(K)$. Thus
\[ \Prob(\card{X([0,\infty])\cap\psi_w(VG_0)}=3 \text{ for some } w\in\Words^*)
	\leq \Prob(X \text{ has a peak}) = 0. \]
Consider two indices $i,j\in\{1,\dotsc,\Delta_n(X_n)\}$ with $i<j-1$. Then $X([\ltime_{i-1,n},\ltime_{i,n}])$ and $X([\ltime_{j-1,n},\ltime_{j,n}])$ are contained in distinct $n$\ndash parts of $K$ and, furthermore,
\[ \{X(\ltime_{i-1,n}),X(\ltime_{i,n})\} \cap \{X(\ltime_{j-1,n}),X(\ltime_{j,n})\} = \emptyset. \]
Hence peaks in both $n$\ndash parts are the only possibility for a non-empty intersection. However, this has probability $0$.
\end{proof}

On $\Omega'$ define $\lspend_{*,n}$ for $n\ge0$ by
\[ \lspend_{*,n} = \max\{ \lspend_{j,n} \,:\, 1\le j\le\Delta_n(X_n) \}. \]
Then $\lspend_{*,0} = \lspend_{1,0} = \hit(X)$ and $\lspend_{*,n+1}\le\lspend_{*,n}$ for all $n\ge0$. Therefore the limit $\lim_{n\to\infty}\lspend_{*,n}$ exists and is finite on $\Omega'$.

\begin{lemma}\label{lemma:times}
$\lspend_{*,n} \to 0$ almost surely as $n\to\infty$.
\end{lemma}

\begin{proof}
If $(W,(b_w)_{w\in W})$ is admissible, then write $A(W,(b_w)_{w\in W})$ to denote the event
\[ A(W,(b_w)_{w\in W}) = \{ W_n(T_\infty,3)=W, (\bar\type_w(T_\infty,3))_{w\in W}=(b_w)_{w\in W} \}. \]
Let $\epsilon>0$, then
\begin{align*}
\Prob(\lspend_{*,n}\ge\epsilon)
&= \Prob(\lspend_{j,n}\ge\epsilon \text{ for some } 1\le j\le\Delta_n(X_n)) \\
&\le \sum_{1\le j\le \Delta_n(X_n)} \Prob(\lspend_{j,n}\ge\epsilon) \\
&= \sum_{(W,(b_w)_{w\in W})} \Prob(A(W,(b_w)_{w\in W}))
	\sum_{1\le j\le\card{W}} \Prob(\lspend_{j,n}\ge\epsilon \mid A(W,(b_w)_{w\in W})),
\end{align*}
where the sum is taken over all admissible pairs. For sake of notation set $c=\frac16(\sqrt{205}-7)$. If $(W,(b_w)_{w\in W})$ is admissible, then, under $\Prob(\,\cdot \mid A(W,(b_w)_{w\in W}))$, the random variable $\lspend_{j,n}$ has the same distribution as $c\bar\alpha^{-n}\bar\theta_i$ for some $i\in\{1,2,3\}$, see Lemma~\ref{lemma:indep}. For $s\in\R$ set
\[ M(s) = \max\{\Expect(e^{s\bar\theta_1}), \Expect(e^{s\bar\theta_2}), \Expect(e^{s\bar\theta_3})\}. \]
Fix some $s>0$; then $M(cs)$ is finite due to Remark~\ref{remark:laplace}. Applying Markov's inequality yields
\[ \Prob(c\bar\alpha^{-n}\bar\theta_i\ge\epsilon)
	= \Prob(e^{cs\bar\theta_i}\ge e^{s\epsilon\bar\alpha^n})
	\le e^{-s\epsilon\bar\alpha^n} M(cs) \]
for all $i\in\{1,2,3\}$. Hence we obtain
\begin{align*}
\Prob(\lspend_{*,n}\ge\epsilon)
&\le \sum_{(W,(b_w)_{w\in W})} \Prob(A(W,(b_w)_{w\in W})) \card{W} e^{-s\epsilon\bar\alpha^n} M(cs) \\
&= e^{-s\epsilon\bar\alpha^n} M(cs) \sum_{(W,(b_w)_{w\in W})} \Prob(A(W,(b_w)_{w\in W})) \card{W} \\
&= e^{-s\epsilon\bar\alpha^n} M(cs) \Expect(\Delta_n(X_n))
\end{align*}
using Lemma~\ref{lemma:indep} once again. Since $\Delta_n(X_n) = \mathbold{\bar c}^\num_n(T_\infty,3)\cdot(1,1,1)^t$, a short computation shows that
\[ \Expect(\Delta_n(X_n))
   = \bigl(\tfrac12+\tfrac3{82}\sqrt{205}\,\bigr) \cdot \bigl(\tfrac43+\tfrac1{15}\sqrt{205}\,\bigr)^n
   + \bigl(\tfrac12-\tfrac3{82}\sqrt{205}\,\bigr) \cdot \bigl(\tfrac43-\tfrac1{15}\sqrt{205}\,\bigr)^n
\le 3\bar\alpha^n \]
for all $n\ge0$. Therefore
\[ \Prob(\lspend_{*,n}\ge\epsilon) \le 3\bar\alpha^n e^{-s\epsilon\bar\alpha^n} M(cs) \]
for all $n\ge0$. By monotonicity
\[ \Prob\Bigl(\lim_{n\to\infty}\lspend_{*,n}\ge\epsilon\Bigr) = 0 \]
and, as $\epsilon>0$ is arbitrary, $\lspend_{*,n}\to0$ almost surely.
\end{proof}

Let $\Omega''$ be the set of all $\omega\in\Omega'$ with the property that the assertions of Lemma~\ref{lemma:peaks} and Lemma~\ref{lemma:times} hold. Then $\Prob(\Omega'')=1$. Using the previous preparations we are now able to prove that the curve $X$ is almost surely self-avoiding and that the random times $\ltime_{j,n}$ are almost surely equal to the consecutive hitting times on the set $VG_n$.

\begin{theorem}\label{theorem:prop}
On $\Omega''$ the following holds:
\begin{itemize}
\item The limit curve $X$ is self-avoiding.
\item For any $1\le j\le\Delta_n(X_n)$,
	\[ \ltime_{j,n} = \time_{j,n}(X) = \inf\{t > \time_{j-1,n}(X) \,:\, X(t)\in VG_n\}. \]
\end{itemize}
\end{theorem}

\begin{proof}
Fix $\omega\in\Omega''$ and consider times $0\le t_1<t_2\le\hit(X)$. By Lemma~\ref{lemma:times} there is an integer $n\ge0$ such that $\lspend_{*,n}<\frac13(t_2-t_1)$. Thus there are indices $i,j\in\{1,\dotsc,\Delta_n(X_n)-1\}$ with $i<j-1$ such that $t_1 \in [\ltime_{i-1,n},\ltime_{i,n}]$ and $t_2 \in [\ltime_{j-1,n},\ltime_{j,n}]$. Since $i<j-1$, Lemma~\ref{lemma:peaks} implies that $X(t_1)\ne X(t_2)$, which proves that $X$ is self-avoiding.

The second statement follows immediately using the first statement, Lemma~\ref{lemma:peaks}, and Theorem~\ref{theorem:conv}.
\end{proof}

\begin{remark}
For $\omega\in\Omega''$, the topological closure of the discrete set
\[ \ltime = \{ \ltime_{j,n} \,:\, n\ge 0, 0\le j\le\Delta_n(X_n) \} \]
contains the interval $[0,\hit(X)]$. Hence $X$ is the continuous extension of
\[ \ltime \to K, \quad \ltime_{j,n} \mapsto X_n(\time_{j,n}(X_n)). \]
\end{remark}

\begin{remark}
The map
\[ \mathcal T_n\to C, \quad t\mapsto \LI(u_1tu_2) \]
is not one-to-one. However, it is possible to use this map and the law of the labelled multi-type Galton-Watson tree of Proposition~\ref{proposition:tree2} to describe the law of the process $X$.
\end{remark}

We use the following lemma as a partial substitute for the missing Markov property in order to prove some properties of the process $(X(t))_{t\ge0}$.

\begin{lemma}\label{lemma:events}
For any $n\in\N_0$, the following holds:
\begin{itemize}
\item If $t\ge s$ and $\norm{X(s)}_2 \ge 2^{-n}$, then $\norm{X(t)}_2 \ge 2^{-n}$.
\item If $t\ge s$, then $\norm{X(t)}_2 \ge \frac12 \norm{X(s)}_2$.
\item On $\Omega''$ we have
	\[ \{ \norm{X(t)}_2 \ge 2^{-n} \}
		= \{ \sup\{\norm{X(s)}_2 \,:\, s\le t\} \ge 2^{-n} \}
		= \{ \ltime_{1,n} \le t \}
		= \{ \lspend_{1,n} \le t \} \]
\end{itemize}
\end{lemma}

\begin{proof}
The first statement is a simple consequence of the geometry of $K$ and implies the second. For the third one note that on $\Omega''$ the curve $X$ is self-avoiding, has no peaks and $\ltime_{1,n}=\lspend_{1,n}$ is the hitting time of $\{2^{-n}u_2, 2^{-n}u_3\} = \psi_{11\dotsm1}(\{u_2,u_3\})$, where $11\dotsm1$ is the word of length $n$ whose letters are all equal to $1$. For $n\ge1$, this implies that the first hitting time of $\{2^{-n}u_2, 2^{-n}u_3\} = \psi_{11\cdots1}(\{u_2,u_3\})$ is equal to the last exit time of the set $2^{-n} K = \psi_{11\cdots1}(K)$. This implies the statement.
\end{proof}

\begin{theorem}\label{theorem:prop-lerw}
The following holds:
\begin{enumerate}[\normalfont(1)]
\item There are $C_{7,\ell},C_{8,\ell}>0$ such that for all $s,t\in[0,\infty)$ and all $\delta\in[0,1]$,
	\[ C_{3,\ell} \exp(-C_{7,\ell} (\delta t^{-\bar\gamma_\ell})^{1/(1-\bar\gamma_\ell)})
		\le \Prob( \norm{X(t)}_2 \ge \delta ) \]
	and
	\begin{align*}
	\Prob( \norm{X(s+t) - X(s)}_2 \ge \delta )
		&\le \Prob( \sup\{ \norm{X(s+u)-X(s)}_2 \,:\, 0\le u\le t \} \ge \delta ) \\
		&\le C_{5,\ell} \exp(-C_{8,\ell} (\delta t^{-\bar\gamma_\ell})^{1/(1-\bar\gamma_\ell)}).
	\end{align*}
\item There are $C_{7,r},C_{8,r}>0$ such that for all $t\in[0,\infty)$ and all $\delta\in[0,1]$,
	\begin{align*}
	C_{3,r} \exp(-C_{7,r} (\delta^{-1/\bar\gamma_\ell} t)^{\bar\gamma_r/(\bar\gamma_r-1)})
		&\le \Prob( \sup\{ \norm{X(u)}_2 \,:\, 0\le u\le t \} \le \delta ) \\
		&\le \Prob( \norm{X(t)}_2 \le \delta )
	\end{align*}
	and
	\[ \Prob( \norm{X(t)}_2 \le \delta ) \le
		C_{5,r} \exp(-C_{8,r} (\delta^{-1/\bar\gamma_\ell} t)^{\bar\gamma_r/(\bar\gamma_r-1)}). \]
\item For any $p>0$, there exist constants $C_9(p),C_{10}(p)>0$
	such that for all $s\in[0,\infty)$ and all $t\in[0,1]$,
	\begin{gather*}
	C_9(p) \, t^{p\bar\gamma_\ell} \le \Expect(\norm{X(t)}_2^p) \qquad\text{and}\qquad
	\Expect(\norm{X(s+t)-X(s)}_2^p) \le C_{10}(p) \, t^{p\bar\gamma_\ell}.
	\end{gather*}
\item There are constants $C_{11},C_{12}>0$ such that for all $s\in[0,\infty)$
	\begin{gather*}
	\limsup_{t\searrow 0}
		\frac{\norm{X(s+t)-X(s)}_2}{t^{\bar\gamma_\ell}(\log\log(1/t))^{1-\bar\gamma_\ell}} \le C_{11}
		\rlap{\qquad and} \\
	\liminf_{t\searrow 0}
		\frac{\norm{X(t)}_2}{t^{\bar\gamma_\ell}(\log\log(1/t))^{-\bar\gamma_\ell(1-1/\bar\gamma_r)}} \ge C_{12}
	\end{gather*}
	hold almost surely. Note that $1-\bar\gamma_\ell\approx0.162475>0$ and
	$-\bar\gamma_\ell(1-1/\bar\gamma_r)\approx-0.206594<0$.
\item The Hausdorff dimension $\dim_H X([0,\infty])$ of the path $X([0,\infty])$ almost surely satisfies
	\[ \dim_H X([0,\infty]) = \frac{1}{\bar\gamma_\ell} = \frac{\log\bar\alpha}{\log2} \approx 1.193995. \]
\end{enumerate}
\end{theorem}

\begin{proof}
In order to prove the first statement choose $n\in\N$ such that $2^{-n}\le\delta\le2^{-(n-1)}$. Then, using Lemma~\ref{lemma:events},
\[ \Prob( \lspend_{1,n-1} \le t ) \le \Prob( \norm{X(t)}_2 \ge \delta ) \]
and
\begin{multline*}
\Prob( \sup\{ \norm{X(s+u)-X(s)}_2 \,:\, 0\le u\le t \} \ge \delta ) \\
	\le \Prob( \ltime_{j-1,n+1} \ge s, \lspend_{j,n+1} \le t \text{ for some } j\ge1 ).
\end{multline*}
By conditioning as in Lemma~\ref{lemma:indep}, the distribution of $\lspend_{j,m}$ is equal to the distribution of $\frac16(\sqrt{205}-7)\bar\alpha^{-m}\bar\theta_i$ for some $i\in\{1,2,3\}$ and $\lspend_{j,m}$ is independent of $\ltime_{j-1,m}$. Hence the bounds on the tail probability follow from Lemma~\ref{lemma:mg-bounds}. More or less the same arguments yield the second statement. By integrating the bounds of the first statement we get the bounds on $\Expect(\norm{X(t)}_2^p)$ and $\Expect(\norm{X(s+t)-X(s)}_2^p)$, respectively. The fourth statement follows by the usual Borel-Cantelli argument. The path $X([0,\infty])$ is the limit set of a random recursive construction with multiple types. Thus the formula for the Hausdorff dimension follows from Theorem~3.8 in \cite{hattori2000exact}, where such random sets are studied in general.
\end{proof}

\begin{remark}
The properties (1), (3), (4) proved above are slightly weaker forms of \cite[Theorem~4.3, Corollary~4.4, Theorem~4.7]{barlow1988brownian} (for Brownian motion) and \cite[Theorem~4.5, Corollary~4.6, Theorem~4.8]{kumagai1993construction} (for more general diffusion processes that contain Brownian motion as a special case), respectively. In several cases the statements of the previous theorem are formulated for the special increment $X(t)=X(t)-X(0)$ and not for a general increment $X(s+t)-X(s)$, which is, we only consider the starting time $s=0$. One reason for the weaker statements is the lack of the Markov property. Another difficulty in the general case lies in the fact that parts of the curve that lie in different $k$\ndash parts of the Sierpi\'nski gasket $K$ can still be close to each other near the vertices where these $k$\ndash parts are connected. At the corner (time $s=0$), this cannot happen. It seems plausible, however, that the strong forms of the cited statements also hold in our case. Fortunately, the formula for the Hausdorff dimension does not rely on the first four properties, but only on the fact that the path $X([0,\infty])$ is the limit set of a specific random recursive construction whith multiple types.
\end{remark}

\begin{remark}
We note that all we have proved in this section remains true if we replace $T_n$ by $S_n^3$. In particular, $\LI(u_1S_n^3u_2,\bar\alpha^n)$ converges almost surely in $(C,d_C)$ to a limit curve and the results of \ref{theorem:conv}\nobreakdash--\ref{theorem:prop-lerw} hold with $S_n^3$ in place of $T_n$.
\end{remark}

\section{Limit of the tree metric}
\label{sec:metric}

Consider a generic $\omega\in\Omega$. Then $T_n(\omega)$ is a spanning tree on $G_n$ and it is the trace $\Tr^m_n T_m(\omega)$ for all $m\ge n$. Let $u,v$ be two vertices in $VG_n$ for some $n\ge0$. Their distance $d_{T_m(\omega)}(u,v)$ with respect to the spanning tree $T_m(\omega)$ is well-defined and Corollary~\ref{corollary:length} indicates that $\bar\alpha^{-m} d_{T_m(\omega)}(u,v)$ converges for $m\to\infty$, where $\bar\alpha = \frac43 + \frac1{15}\sqrt{205}$ is the dominating eigenvalue of Proposition~\ref{proposition:tree2}. If this limit exists for all $u,v$ in the countable set
\[ V_* = \bigcup_{n\ge0} VG_n = \bigcup_{w\in\Words^*} \psi_w(VG_0), \]
and it is positive whenever $u \neq v$, then the limit defines a metric $d_{*,\omega}$ on $V_*$:
\[ d_{*,\omega}(u,v) = \lim_{m\to\infty} \bar\alpha^{-m} d_{T_m(\omega)}(u,v) \]
for all $u,v\in V_*$. In the following we show that $d_{*,\omega}$ exists for almost all $\omega\in\Omega$ and yields a random metric $d_*$ on $V_*$. Let $\Metrics(V_*)$ be the set of all metrics on $V_*$. We equip $\Metrics(V_*)$ with the $\sigma$\ndash algebra $\SigmaAlgebraMetrics(V_*)$ which is induced by the mappings
\[ \Metrics(V_*) \to \R, \quad d \mapsto d(u,v) \]
for $u,v\in V_*$. We recall some notions from metric theory, see for instance \cite{chiswell2001introduction}. A metric space $(X,d)$ is \emph{$0$\ndash hyperbolic} if
\[ d(u,v) + d(x,y) \le \max\{ d(u,x) + d(v,y), d(u,y) + d(v,x) \} \]
holds for all $u,v,x,y\in X$ (\emph{four point condition}). A \emph{metric segment} in $(X,d)$ is the image of an isometric embedding $[a,b]\to X$ for some $a,b\in\R$. Finally, $(X,d)$ is called an \emph{$\R$\ndash tree} if, for any $x,y\in X$, there is a unique arc connecting $x,y$ and this arc is a metric segment. We note that $(X,d)$ is an $\R$\ndash tree if and only if $(X,d)$ is connected and $0$\ndash hyperbolic, see \cite[Lemma~2.4.13]{chiswell2001introduction}.

\begin{theorem}\label{theorem:randommetric}
For almost all $\omega\in\Omega$ the limit
\[ d_{*,\omega}(u,v) = \lim_{m\to\infty} \bar\alpha^{-m} d_{T_m(\omega)}(u,v) \]
exists for all $u,v\in V_*$ and yields a metric $d_{*,\omega}$ on the set $V_*$, such that $(V_*,d_{*,\omega})$ is a $0$\ndash hyperbolic and totally bounded metric space. Thus, for a suitable subset $\Omega'''\subseteq\Omega$ of probability $1$,
\[ \Omega'''\to\Metrics(V_*), \quad \omega\mapsto d_{*,\omega} \]
is a random metric in $(\Metrics(V_*),\SigmaAlgebraMetrics(V_*))$. Furthermore, for $\omega\in\Omega'''$ the Cauchy completion of $(V_*, d_{*,\omega})$ is a compact $\R$\ndash tree.
\end{theorem}

\begin{proof}
For $x,y\in VG_0$ and $w\in\Words^n$, define $\Omega(w,x,y)$ to be the set of all $\omega\in\Omega$ such that, whenever $x,y$ are connected in the restriction $\pi_w(T_n(\omega))$, the curve $\LI(x \pi_w(T_m(\omega)) y, \bar\alpha^m)$ converges in $(C,d_C)$ as $m\to\infty$, $m\ge n$, and the assertions of Theorems~\ref{theorem:conv}\nobreakdash--\ref{theorem:prop} hold. The usual conditioning argument shows that $\Prob(\Omega(w,x,y))=1$ for all $w\in\Words^*$ and all $x,y\in VG_0$. Thus
\[ \Omega''' = \bigcap_{w\in\Words^*}\bigcap_{x,y\in VG_0} \Omega(w,x,y) \]
has probability $1$. Fix an element $\omega\in\Omega'''$. Then for all $u,v\in V_*$ the limit
\[ d_{*,\omega}(u,v) = \lim_{m\to\infty} \bar\alpha^{-m} d_{T_m(\omega)}(u,v) \]
exists and is an element of $[0,\infty)$. By construction of $\Omega'''$, we have $d_{*,\omega}(u,v) > 0$ for all $u,v\in V_*$, $u\ne v$, which are neighbours in $G_n$ for some $n$. Hence $d_{*,\omega}(u,v) > 0$ for all $u,v\in V_*$, $u\ne v$. Furthermore, as $d_{T_m(\omega)}$ is the graph metric of the tree $T_m(\omega)$, it satisfies the triangle inequality and the four point condition. Thus the limit $d_{*,\omega}$ also satisfies the triangle inequality and the four point condition. Altogether we have proved that $(V_*,d_{*,\omega})$ is a $0$\ndash hyperbolic metric space if $\omega\in\Omega'''$. For $x,y\in VG_0$ and $w\in\Words^n$ define $A(w,x,y)$ to be the set of all $\omega\in\Omega'''$, such that, whenever $x,y$ are connected in the restriction $\pi_w(T_n(\omega))$, then $d_{*,\omega}(\psi_w(x),\psi_w(y))\le2^{-n}$. Using the Borel-Cantelli lemma together with the bounds of Lemma~\ref{lemma:mg-bounds}, we see that
\[ A_n = \bigcap_{w\in\Words^n}\bigcap_{x,y\in VG_0} A(w,x,y) \]
holds eventually with probability $1$. Hence, for $\omega\in\Omega'''$, there is an $N=N(\omega)$ such that $\omega\in A_n$ for all $n\ge N$. Fix some $n\ge N$. For $x\in VG_n$ let $C_x = C_x(\omega)$ be the set of all $y\in VG_m$ ($m\ge n$), such that all vertices $v$ on the path $x T_m(\omega) y$ satisfy $\norm{v-x}_2\le2^{-n}$. If $y\in C_x\cap VG_n$, then $d_{*,\omega}(x,y)\le 2^{-n}$. If $y\in C_x\setminus VG_n$, then we can find $x=x_n,x_{n+1},\dotsc,x_m=y$, such that $x_k\in VG_k$ and $x_{k-1},x_k$ are either identical or neighbours in $T_k(\omega)$. Thus
\[ d_{*,\omega}(x,y) \le \sum_{k=n+1}^m d_{*,\omega}(x_{k-1},x_k) \le \sum_{k=n+1}^m 2^{-k} \le 2^{-n}. \]
Thus, if $B_{*,\omega}(x,2^{-n})$ denotes the ball of radius $2^{-n}$ centered at $x$ with respect to $d_{*,\omega}$, then $C_x\subseteq B_{*,\omega}(x,2^{-n})$. Hence
\[ V_* = \bigcup_{x\in VG_n} C_x = \bigcup_{x\in VG_n} B_{*,\omega}(x,2^{-n}), \]
which means that $(V_*,d_{*,\omega})$ is totally bounded. To check measurability we note that $\omega\mapsto d_{T_m(\omega)}(u,v)$ is measurable for fixed $u,v\in V_*$ (if $m$ is sufficiently large). Thus the limit $\omega\mapsto d_{*,\omega}(u,v)$ is measurable, too. By definition of $\SigmaAlgebraMetrics(V_*)$, this implies measurability of $\omega\mapsto d_{*,\omega}$.

In order to prove that the Cauchy completion $(\check V_{*,\omega}, \check d_{*,\omega})$ of $(V_*, d_{*,\omega})$ for $\omega\in\Omega'''$ is an $\R$\ndash tree, it is sufficient to show that the completion is connected, as $0$\ndash hyperbolicity is preserved by completion, see \cite[Lemma~2.2.11]{chiswell2001introduction}. We show that the completion contains a path from $u_1$ to any $x$. Let $x_1,x_2,\dotsc$ be a Cauchy sequence in $V_*$ with $x_n\to x$. Denote by $\alpha_n\colon[0,\infty]\to K$ the limit curve of $\LI(u_1 T_m(\omega) x_n, \bar\alpha^m)$ as $m\to\infty$, which exists by construction of $\Omega'''$. Then $D_n = \alpha_n^{-1}(V_*)$ is a dense subset of $[0,\infty]$ by Lemma~\ref{lemma:times}. Note that $t=d_{*,\omega}(u_1,\alpha_n(t))$ for all $t\in D_n$ such that $t \le \min\{s \,:\, \alpha_n(s)=x_n\}$. Therefore the restriction $\alpha_n\colon D_n\to V_*$ is continuous with respect to $d_{*,\omega}$ and thus has a continuous extension $\beta_n\colon[0,\infty]\to\check V_{*,\omega}$. Set $s_0=0$ and
\[ s_n = \max\{ t \in [0,\infty] \,:\, \beta_k = \beta_n \text{ on } [0,t] \text{ for all } k\ge n \}. \]
Then we have $s_0\le s_1 \le \dotsb$ and $\beta_n(s_n)\to x$, and
\[ \beta\colon[0,\infty]\to\check V_{*,\omega}, \quad
	\beta(t) = \begin{cases}
				u_1 & \text{if } t=0, \\
				\beta_n(t) & \text{if } s_{n-1}<t\le s_n, \\
				x & \text{otherwise,}
			\end{cases} \]
is a continuous curve connecting $u_1$ and $x$ (whose image is a metric segment). Finally, $(\check V_{*,\omega}, \check d_{*,\omega})$ is compact for $\omega\in\Omega'''$, since it is the completion of the totally bounded metric space $(V_*,d_{*,\omega})$.
\end{proof}

Let $\omega$ be an element of the set $\Omega'''$ defined in the previous proof and let $(\check V_{*,\omega}, \check d_{*,\omega})$ be the Cauchy completion of $(V_*, d_{*,\omega})$. Consider an element $x\in\check V_{*,\omega}$. Suppose that $x_1,x_2,\dotsc$ is a Cauchy sequence in $(V_*,d_{*,\omega})$, such that $x_n\to x$ with respect to $\check d_{*,\omega}$. Then it is easy to see that $x_1,x_2,\dotsc$ is also a Cauchy sequence in $(V_*, \norm{\,\cdot\,}_2)$ and thus has a limit in $(K, \norm{\,\cdot\,}_2)$, which does not depend on the specific Cauchy sequence but only on $x\in \check V_{*,\omega}$. We write $\xi_\omega(x)$ to denote this limit in $(K, \norm{\,\cdot\,}_2)$. Then $\xi_\omega\colon \check V_{*,\omega}\to K$ is a well-defined, continuous map, such that the restriction $\xi_\omega|_{V_*}$ to $V_*$ is the identity.

\begin{lemma}\label{lemma:mult}
Let $\omega$ be in $\Omega'''$. Then $1\le\card{\xi_\omega^{-1}(x)}\le4$ for all $x\in V_*$ and $1\le\card{\xi_\omega^{-1}(x)}\le3$ for all $x\in K\setminus V_*$.
\end{lemma}

\begin{proof}
For every point $x\in K$ we can find a sequence in $V_*$ that converges to this point in $(K, \norm{\,\cdot\,}_2)$ and which is Cauchy in $(V_*,d_{*,\omega})$. Thus the map $\xi_\omega$ is surjective, whence $\card{\xi_\omega^{-1}(x)}\ge1$. As in the previous proof every sequence $x_1,x_2,\dotsc\in V_*$ converging to a point in $\xi_\omega^{-1}(x)$ in $(V_*,d_{*,\omega})$ yields a metric segment connecting $u_1$ and that point. Using the geometry of the Sierpi\'nski gasket it is easy to see that there are at most four (respectively three if $x\notin V_*$) distinct metric segments joining $u_1$ and a point in $\xi_\omega^{-1}(x)$. This proves the claim.
\end{proof}

\begin{theorem}\label{theorem:limit}
Let $\omega$ be an element of $\Omega'''$. Then the hitting time $\hit(X(\omega))$ of the limit curve $X(\omega)$ in $u_2$ is equal to the distance $d_{*,\omega}(u_1,u_2)$. Furthermore, if $\gamma_\omega\colon[0,d_{*,\omega}(u_1,u_2)]\to\check V_{*,\omega}$ is the unique isometric embedding with $\gamma_\omega(0)=u_1$ and $\gamma_\omega(d_{*,\omega}(u_1,u_2))=u_2$, then
\[ X(t,\omega) = \xi_\omega(\gamma_\omega(t)) \]
for all $t\in[0,d_{*,\omega}(u_1,u_2)]$.
\end{theorem}

\begin{proof}
The statement is a consequence of the definition of the limit curve $X(\omega)$ and the limit metric $d_{*,\omega}$, see Theorem~\ref{theorem:conv} and Theorem~\ref{theorem:randommetric}.
\end{proof}

For $\omega\in\Omega'''$ define $A(\omega)$ to be the set $\{ x \in K \,:\, \card{\xi_\omega^{-1}(x)} > 1 \}$. These are points that ``can be reached from two (or more) different directions''. To understand how this happens, it is useful to consider spanning forests with two components: given for instance some $f \in \mathcal S_{\infty}^1$, every element $v$ of $V_*$ can be associated uniquely to one of the components: $v \in V(G_n)$ for some $n$, and $v$ either belongs to the same component as $u_1$ in $\Tr_m^{\infty} f$ for all $m \geq n$ or to the same component as $u_2$ and $u_3$, again for all $m \geq n$. There are, however, some points in the completion $K$ that can be reached as limits from both sides; they form the so-called ``interface''. In a spanning tree, there is only one component, but the same phenomenon can occur at higher levels, within certain $n$\ndash parts on which the spanning tree induces a spanning forest with more than one component.

In the following we give a description of $A(\omega)$ in terms of Galton-Watson trees and show that the Hausdorff dimension $\dim_H A(\omega)$ is strictly less than $1$ for almost all $\omega$. For $f\in \mathcal Q_\infty$ and $n\ge0$ let $\check W_n(f)$ be the set of all $w\in\Words^n$, such that $\psi_w(VG_0)$ contains vertices of two distinct components of $\Tr^\infty_n f$. The union
\[ \check W(f) = \bigcup_{n\ge0} \check W_n(f) \]
induces a subtree of $\Words^*$. On a single $n$\ndash part $\psi_w(VG_0)$ with $w\in\check W_n(f)$ we always observe one of the following possibilities:
\begin{itemize}
\item The restriction $\pi_w(\Tr^\infty_n f)$ has two components and
	these two components belong to two distinct components of $\Tr^\infty_n f$.
	In this case we set $\check\type_w(f) = \type_w(f) \in \{\comp1yy-,\comp2yy-,\comp3yy-\}$.
\item The restriction $\pi_w(\Tr^\infty_n f)$ has three components and
	two of them belong to the same component of $\Tr^\infty_n f$.
	In this case we define $\check\type_w(f) \in \{\comp4ynn,\comp4nyn,\comp4nny\}$
	depending on which two of the three components in $\pi_w(\Tr^\infty_n f)$
	belong to the same component of $\Tr^\infty_n f$.
\item The restriction $\pi_w(\Tr^\infty_n f)$ has three components and
	these three components belong to three distinct components of $\Tr^\infty_n f$.
	In this case we set $\check\type_w(f) = \type_w(f) = \comp4yyy$.
\end{itemize}
Let $\mathcal{\check C} = \{\comp1yy-,\comp2yy-,\comp3yy-,\comp4ynn,\comp4nyn,\comp4nny,\comp4yyy\}$ and set
\[ \mathbold{\check\type}(f) = (\check\type_w(f))_{w\in\check W(f)}. \]
As in Section~\ref{subsection:component} it is easy to see that $\mathbold{\check\type}(U_\infty)$ is a labelled multi-type Galton-Watson tree with types in $\mathcal{\check C}$, where $U_\infty$ is one of $S_\infty^1, S_\infty^2, S_\infty^3, R_\infty^{}$. The associated counting process $(\mathbold{\check c^\num}(U_\infty))_{n\ge0}$, which counts type occurrences in one generation up to symmetry, is a multi-type Galton-Watson process with three types, offspring generating function
\begin{align*}
\mathbold{\check g}(\mathbold z) = \Bigl(
		&\tfrac1{10} (4z_1+3z_1^2+3z_2), \\
		&\tfrac1{25} (6z_1+3z_1^2+z_1^3+6z_2+9z_1z_2), \\
		&\tfrac1{25} z_1(3z_1+4z_1^2+9z_2+9z_3)
	\Bigr)
\end{align*}
and mean matrix
\[ \mathbold{\check M} = \frac1{50} \cdot \begin{pmatrix}
			50 & 15 &  0 \\
			48 & 30 &  0 \\
			72 & 18 & 18
		\end{pmatrix}. \]
This mean matrix has the dominating eigenvalue $\check\alpha = \frac35\bar\alpha = \frac45 + \frac1{25}\sqrt{205} \approx 1.372712$. Define
\[ I(f) = \bigcap_{n\ge0} \bigcup_{w\in\smash{\check W_n(f)}} \psi_w(K). \]
Then $I(f)$ is the limit set of the component boundaries and
\[ \dim_H I(U_\infty) \le \frac{\log\check\alpha}{\log2}
	= \frac{\log\bar\alpha}{\log2} - \frac{\log\frac53}{\log2} \approx 0.457029 \]
holds almost surely using \cite[Proposition~3.9]{tsujii1991markov}. It seems that other results on the Hausdorff dimension do not apply to this specific random recursive construction, so that we only obtain an upper bound. Of course, $I(T_\infty)=\emptyset$ and so $\dim_H I(T_\infty)=0$.

\begin{proposition}\label{proposition:interface}
For $\omega\in\Omega'''$ we have
\[ A(\omega) = \bigcup_{w\in\Words^*} \psi_w(I(\pi_w(T_\infty(\omega)))) \]
and thus
\[ \dim_H A(\omega) \le \frac{\log\check\alpha}{\log2}
	= \frac{\log\bar\alpha}{\log2} - \frac{\log\frac53}{\log2} \approx 0.457029 \]
for almost all $\omega$.
\end{proposition}

\begin{proof}
Note that $A(\omega)$ contains $\psi_w(I(\pi_w(T_\infty(\omega))))$ for all $w\in\Words^*$. On the other hand, if $x\in A(\omega)$, then $\xi_\omega^{-1}(x)$ contains at least two distinct points in $\check V_{*,\omega}$, say $x_1$ and $x_2$. Denote by $\overline{u_1x_1}$ (respectively $\overline{u_1x_2}$) the metric segment connecting $u_1$ and $x_1$ (respectively $x_2$). Then there is a word $w\in\Words^*$ such that $x\in\psi_w(K)$ and
\[ \overline{u_1x_1} \cap \overline{u_1x_2} \cap \xi_\omega^{-1}(\psi_w(K)) = \emptyset. \]
This implies that $x\in\psi_w(I(\pi_w(T_\infty(\omega))))$. The usual conditioning argument shows that
\[ \dim_H \psi_w(I(\pi_w(T_\infty(\omega)))) \le \frac{\log\check\alpha}{\log2} \]
for almost all $\omega$. As $\Words^*$ is a countable set and the Hausdorff dimension behaves nicely under countable unions the claim follows.
\end{proof}

\begin{remark}
Note the occurrence of the constant $\frac53$, which is the \emph{resistance scaling factor} of the Sierpi\'nski gasket. It also occurs prominently in the formula for the number of spanning trees (see \cite{teufl2011resistance} for the connection between resistance scaling and the number of spanning trees): if we regard $G_n$ as an electrical network, where each edge represents a unit resistor, then the effective resistance between two of the boundary vertices $u_1,u_2,u_3$ is $\frac23 \cdot (\frac53)^n$. There is a simple heuristic explanation why the identity
\[ \log\check\alpha = \log\bar\alpha - \log\tfrac53 \]
must hold: it is well known (cf.~\cite[p.~44, Theorem~1]{bollobas1998modern}) that the effective resistance between two vertices equals the number of \emph{thickets}, i.e., spanning forests with two components each containing one of the two vertices, divided by the number of spanning trees. For every spanning tree of $G_n$, one can obtain a thicket by removing an edge from the unique path between $u_1$ and $u_2$; conversely, we can turn a thicket into a spanning tree by inserting an edge that connects the two components at the interface. The identity now follows (at least heuristically) from a simple double-counting argument.
\end{remark}

\section{Other self-similar graphs}
\label{sec:other}

The same ideas apply to other self-similar graphs as well: it was shown in \cite{teufl2011number} that the recursions for counting spanning trees and forests in self-similar sequences of graphs have simple explicit solutions as for the Sierpi\'nski graphs if the number of ``boundary'' vertices is two (as for example in the case of the graphs associated with the modified Koch curve, see Figure~\ref{figure:koch}) or three (as for the Sierpi\'nski graphs), provided that the automorphism group acts with either full symmetry or like the alternating group on the set of boundary vertices. For two boundary vertices, this technical condition is always satisfied. The explicit counting formulae guarantee that the projections will still be measure-preserving, and all other arguments can be carried out in the same way as in the previous sections.

\begin{figure}[htb]
\centering
\def\fig#1#2#3#4{%
	\scope[shift={#1}]
		\mytikzkoch{}{#2}{#3}
		\node[above left] at (60:1/2) {#4};
	\endscope}
\begin{tikzpicture}[scale=2.5]
	\fig{(0,0)}{}{\draw (0:0) node[vertex] {} -- (0:1) node[vertex] {};}{$G_0$}
	\fig{(1.2,0)}{x}{\draw (0:0) node[vertex] {} -- (0:1) node[vertex] {};}{$G_1$}
	\fig{(2.4,0)}{xx}{\draw (0:0) node[vertex] {} -- (0:1) node[vertex] {};}{$G_2$}
	\fig{(4,0)}{xxxx}{\draw (0:0)  -- (0:1);}{$K$}
\end{tikzpicture}
\caption{The modified Koch curve.}
\label{figure:koch}
\end{figure}
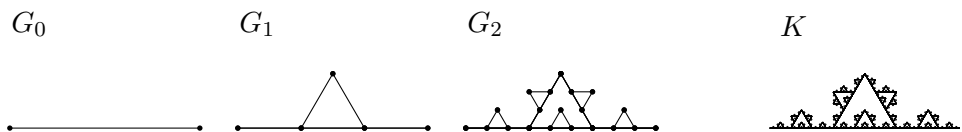

For two boundary vertices, the rescaling factor is precisely the average length of loop-erased random walk from one boundary vertex to the other in $G_1$ (the initial graph $G_0$ being a single edge), which is always a rational number. For example, for the sequence of graphs in Figure~\ref{figure:koch}, the rescaling constant is $\frac{10}{3}$ (in other words, the length of loop-erased random walk from one boundary vertex of $G_n$ to the other grows like $(\frac{10}{3})^n$). It follows that the Hausdorff dimension of the limit curve is almost surely $\log(\frac{10}{3})/\log3 \approx 1.095903274$ in this example. As a second example, consider the Sierpi\'nski graphs with two subdivisions on each edge in Figure~\ref{figure:sg2}: in this case, we find that the rescaling factor is $\frac1{735}(1431+\sqrt{1669656}\,)$ (it is a priori clear that it has to be algebraic of degree $\leq 2$, being an eigenvalue of a $2 \times 2$\ndash matrix with rational entries), giving us a Hausdorff dimension of $\approx 1.192117286$ for the limit curve of loop-erased random walk.

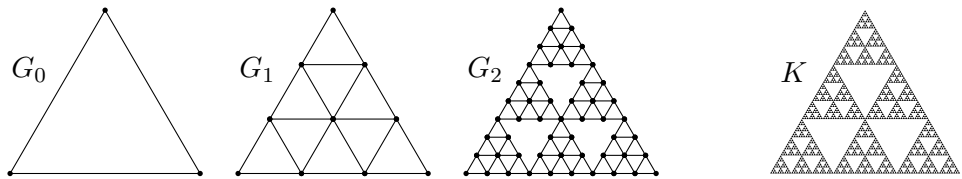
\begin{figure}[htb]
\centering
\def\fig#1#2#3#4{%
	\scope[shift={#1}]
		\mytikzsgb{}{#2}{#3}
		\node[above left] at (60:1/2) {#4};
	\endscope}
\begin{tikzpicture}[scale=2.5]
	\fig{(0,0)}{}{\draw \mytikztri{--}{node[vertex] {}} -- cycle;}{$G_0$}
	\fig{(1.2,0)}{x}{\draw \mytikztri{--}{node[vertex] {}} -- cycle;}{$G_1$}
	\fig{(2.4,0)}{xx}{\draw \mytikztri{--}{node[vertex] {}} -- cycle;}{$G_2$}
	\fig{(4,0)}{xxxx}{\fill \mytikztri{--}{} -- cycle;}{$K$}
\end{tikzpicture}
\caption{Sierpi\'nski graphs with two subdivisions.}
\label{figure:sg2}
\end{figure}

If the number of boundary vertices is four or more (which happens, for instance, for the higher-dimensional analogues of the Sierpi\'nski graphs), then more different types of spanning forests have to be considered, and there are generally no exact counting formulae. However, asymptotic formulae should hold in such cases, making the projections ``asymptotically measure-preserving'', so that analogous results hold in such cases. The details might be quite intricate though, and new geometric phenomena arise as well: for instance, with four boundary vertices, it becomes possible that a loop-erased random walk on $G_n$ enters and leaves some of the copies of $G_k$ ($k < n$) more than once, which is not possible in the case of Sierpi\'nski graphs that we considered.

\def\doi#1{\href{http://dx.doi.org/#1}{\protect\nolinkurl{doi:#1}}}
\bibliographystyle{amsplainurl}
\bibliography{cfg}

\parindent=0pt


\end{document}